\UseRawInputEncoding
\documentclass[10pt]{amsart}
\usepackage[english]{babel}
\usepackage{amssymb}
\usepackage{amsfonts}
\usepackage{amsmath}
\usepackage[hidelinks]{hyperref}
\usepackage{bm}
\usepackage{latexsym}
\usepackage{epsfig}
\numberwithin{equation}{section}

\newtheorem{theorem}{Theorem}[section]
\newtheorem{cor}[theorem]{Corollary}
\newtheorem{lemma}[theorem]{Lemma}
\newtheorem{proposition}[theorem]{Proposition}
\newtheorem{definition}[theorem]{Definition}
\newtheorem{remark}[theorem]{Remark}

\numberwithin{equation}{section}
\def\beq{\begin{equation}}
\def\eeq{\end{equation}}
\def\ben{\begin{enumerate}}
\def\een{\end{enumerate}}

\def\cA{{\mathcal A}}

\def\cC{{\mathcal C}}
\def\cE{{\mathcal E}}
\def\cD{{\mathcal D}}
\def\cF{{\mathcal F}}
\def\cG{{\mathcal G}}
\def\cH{{\mathcal H}}

\def\cO{{\mathcal O}}

\def\cN{{\mathcal N}}
\def\cM{{\mathcal M}}

\def\cR{{\mathcal R}}
\def\cS{{\mathcal S}} 
\def\cT{{\mathcal T}}

\def\cV{{\mathcal V}}
\def\cW{{\mathcal W}}

\def\cW{\mathcal W}

\newcommand{\RR}{{\mathbb R}}
\newcommand{\CC}{{\mathbb C}}

\newcommand{\KK}{{\mathbb K}}
\newcommand{\NN}{{\mathbb N}}

\newcommand{\fR}{{\mathfrak R}}

\newcommand{\lam}{{\lambda}}

\newcommand{\mtm}{m\times m}
\newcommand{\ntn}{n\times n}
\newcommand{\ptp}{p\times p}

\newcommand{\sts}{s\times s}

\newcommand{\wt}{\widetilde}
\newcommand{\wh}{\widehat}

\newcommand{\ul}{\underline}

\newcommand{\bB}{\mathbf{B}}

\newcommand{\bT}{\mathbf{T}}
\newcommand{\bA}{\mathbf{A}}

\newcommand{\fC}{\mathfrak{C}}
\newcommand{\fO}{\mathfrak{O}}

\newcommand{\fB}{\mathfrak{B}}
\newcommand{\fA}{\mathfrak{A}}
\newcommand{\fa}{\mathfrak{a}}
\newcommand{\fm}{\mathfrak{m}}
\newcommand{\fr}{\mathfrak{r}}

\newcommand{\obs}[2]{\cN\cO_{#1,#2}}
\DeclareMathOperator{\Ima}{Im}

\newcommand{\introthmname}{}
\newtheorem{introthminn}{\introthmname}
\newenvironment{introthm}[1]
  {\renewcommand{\introthmname}{#1}\begin{introthminn}}
  {\end{introthminn}}

\swapnumbers

\usepackage{color}
\usepackage[normalem]{ulem}
\makeatletter
\def\moverlay{\mathpalette\mov@rlay}
\def\mov@rlay#1#2{\leavevmode\vtop{
    \baselineskip\z@skip \lineskiplimit-\maxdimen
    \ialign{\hfil$#1##$\hfil\cr#2\crcr}}}
\makeatother
\newcommand{\plangle}{\moverlay{(\cr<}}
\newcommand{\prangle}{\moverlay{)\cr>}}

\begin{document}

\title[Realizations of nc rational functions 
around a matrix centre, I]
{Realizations of Non-Commutative Rational 
Functions around a matrix centre, I:
synthesis, minimal realizations and evaluation on stably finite algebras}
\thanks{The research of both authors was 
partially supported by the US--Israel 
Binational Science Foundation (BSF) 
Grant No. 2010432, Deutsche Forschungsgemeinschaft (DFG) 
Grant No. SCHW 1723/1-1, 
and Israel Science Foundation (ISF) 
Grant No. 2123/17.}
\author[M. Porat]{Motke Porat}
\address{(MP) Department of Mathematics\\
Ben-Gurion University of the Negev\\
P.O. Box 653,
Beer-Sheva 84105\\ Israel}
\email{motpor@gmail.com}
\author[V. Vinnikov]{Victor Vinnikov}
\address{(VV) Department of Mathematics\\
Ben-Gurion University of the Negev\\
P.O. Box 653,
Beer-Sheva 84105\\ Israel}
\email{vinnikov@math.bgu.ac.il}
%
%
\begin{abstract}
In this paper we generalize classical results 
regarding minimal realizations  of 
non-commutative (nc) rational functions using  nc 
Fornasini--Marchesini
realizations which are centred at an arbitrary matrix point.
We prove the existence and uniqueness 
of a minimal realization for every nc 
rational function, centred at 
an arbitrary matrix point in its domain of regularity.
Moreover, we show that using this realization we can evaluate the function on 
\textbf{all} of its domain (of matrices of all sizes) and also w.r.t \textbf{any} stably finite algebra. 
As a corollary we obtain a new proof of the theorem by Cohn and Amitsur, that equivalence of two rational expressions over matrices implies the expressions are equivalent over all stably finite algebras. 
Applications to the matrix valued and the symmetric cases are presented as well.
\end{abstract}
\maketitle
\tableofcontents
\section*{Introduction}
Noncommutative (nc, for short) rational functions 
are a skew field of fractions
--- more precisely, the universal skew field of fractions --- 
of the
ring of nc polynomials, i.e., polynomials in 
noncommuting indeterminates (the free associative algebra).
Essentially, they are obtained
by starting with nc polynomials and applying
successive arithmetic operations;
a considerable amount of technical details is necessary here since
in contrast to the commutative case there is no canonical coprime
fraction representation for a nc rational function. NC rational functions originated from several sources: the general theory of free rings and of skew fields
(see 
\cite{Co61,Hu70,Co71a,Co72,Le74,Linn93,Lich00},
\cite{Co71,Co82,Co06,Co95} 
for comprehensive expositions, and
\cite{R99,Linn06} 
for good surveys); 
the theory of rings with
rational identities (see 
\cite{AM66}, 
also 
\cite{Be70} and
\cite[Chapter 8]{Row80}); 
and rational former power series in the
theory of formal languages and finite automata (see
\cite{Kle,Schutz61,Schutz62b,Fliess70,Fliess74a,Fliess74b} 
and
\cite{BR} for a good survey). 

Much like in the case of rational functions of a single variable 
\cite{BGK,KAF}
(and unlike the case of several commuting variables 
\cite{G1,Kac1}),
nc rational functions that are regular at 
$\ul{0}$
admit a good state space realization theory, 
see in particular Theorem 
\ref{thm:realization} below.
This was first established in the context of finite automata and recognizable power series,
and more recently reformulated, with additional details, in the context of 
transfer functions of multidimensional systems with evolution
along the free monoid 
(see 
\cite{BV1,BGM1,BGM2,BGM3,AK0,BK-V}).
State space realizations of nc rational functions have figured
prominently in work on robust control of linear systems subjected
to structured possibly time-varying uncertainty (see 
\cite{Beck,
BeckDoyleGlover, LuZhouDoyle}). 
Another important application of nc rational functions
appears in the area of Linear Matrix Inequalities (LMIs, see,
e.g., 
\cite{NN,N06,SIG97}). 
Most optimization problems of system theory and
control are dimensionless in the sense that the natural variables
are matrices, and the problem involves nc rational expressions in these matrix
variables which have therefore the same form independent of matrix sizes (see
\cite{CHSY,H1,H3}).
State space realizations are exactly what is
needed to convert (numerically unmanageable) rational matrix
inequalities into (highly manageable) linear matrix inequalities
(see \cite{HMcCV}). 

Coming from a different direction, the method of state space realizations,
also known as the linearization trick, found important recent applications in free
probability, see 
\cite{BMS,HMS,Sp1,Sp3}. 
Here it is crucial to evaluate nc rational expressions
on a general algebra --- which is stably finite in many important cases ---
rather than on matrices of all sizes. Stably finite algebras appeared in this context in the work of Cohn
\cite{Co06} and they play an important and not surprising role in our analysis.

Here is  a full characterization of nc rational functions which are regular at $\ul{0}$
and their (matrix) domains of regularity, 
in terms of their minimal realizations
(for the proofs, see 
\cite{BGM1,BGM2,Fliess70,Fliess74a,Fliess74b,KV3,KV4,Schutz61,Schutz62b}).
\begin{introthm}{Theorem}
\label{thm:realization}
If 
$\fR$
is a  nc rational function of 
$x_1,\ldots,x_d$ 
and 
$\fR$ 
is regular at 
$\ul{0}$, 
then
$\fR$ 
admits a unique (up to unique similarity) 
minimal nc Fornasini--Marchesini realization
\begin{equation*}
\fR(x_1,\ldots,x_d)=D+C
\Big(I_L-
\sum_{k=1}^d A_kx_k
\Big)^{-1}
\sum_{k=1}^d B_kx_k,
\end{equation*}
where 
$A_1,\ldots,A_d\in\KK^{L\times L} 
,B_1,\ldots,B_d\in\KK^{L\times 1}, 
C\in\KK^{1\times L},D=\fR(\ul{0})\in\KK$ 
and 
$L\in\NN$. 
Moreover, for all 
$m\in\NN:$ 
$(X_1,\ldots,X_d)\in(\KK^{\mtm})^d$ 
is in the domain of regularity of 
$\fR$ if and only if  
$\det\left(I_{Lm}-X_1\otimes A_1-\ldots-
X_d\otimes A_d\right)\ne0;$ in that case
\begin{equation*}
\fR(X_1,\ldots,X_d)=
I_m\otimes D+(I_m\otimes C)
\Big( I_{Lm}-\sum_{k=1}^d
X_k\otimes A_k\Big)^{-1}\sum_{k=1}^d
X_k\otimes B_k.
\end{equation*}
\end{introthm}
Here a realization is called minimal if the state space dimension 
$L$ 
is as small as possible; equivalently,   
the realization is observable, i.e., 
$$\bigcap_{0\le k}
\bigcap_{\,1\le i_1,\ldots,i_k\le d}\ker(CA_{i_1}\cdots A_{i_k})
=\{\ul{0}\},$$ and controllable, i.e.,
$$\bigvee_{0\le k}
\bigvee_{\,1\le i_1,\ldots, i_k,j\le d,} A_{i_1}\cdots A_{i_k}B_j=\KK^L.$$
Theorem 
\ref{thm:realization} 
is strongly related to  
expansions of nc rational functions which are regular at 
$\ul{0}$ 
into formal nc power series around 
$\ul{0}$; 
that is why it is \textbf{not} 
applicable for all nc rational functions. 
For example, the nc rational expression
$R(x_1,x_2)=(x_1x_2-x_2x_1)^{-1}$
is not defined at $\ul{0}$, nor at any 
pair
$(y_1,y_2)\in\KK^2$, 
therefore one can not consider 
realizations of 
$R$
which are centred at 
$\ul{0}$ 
as in Theorem 
\ref{thm:realization}, nor at any scalar point 
(a tuple of scalars). A realization theory for such 
expressions (and hence functions) 
is required  in particular for all of the applications 
mentioned above.  Such a theory is 
presented here, using the ideas of the general 
theory of nc functions.

The theory of \textbf{nc functions} has its roots in the works by Taylor \cite{T1,T2} on noncommutative
spectral theory. 
It was further developed by Voiculescu 
\cite{VDN,Voi04,Voi09} 
and Kalyuzhnyi-Verbovetskyi--Vinnikov \cite{KV1},
including a detailed discussion on nc difference-differential calculus.
The main underlying idea is that
a function of 
$d$ 
non-commuting variables is a function of 
$d-$tuples of square matrices of all sizes
that respects direct sums and simultaneous similarities.
See also
the work of Helton--Klep--McCullough 
\cite{HKMcC1,HKMcC2}, 
of Popescu 
\cite{Po06,Po10}, 
of Muhly--Solel 
\cite{MS}, and of Agler--McCarthy 
\cite{AgMcC1,AgMcC2,AgMcC6}.
A crucial fact 
\cite[Chapters 4-7]{KV1} 
is that nc functions admit power series expansions,
called Taylor--Taylor series in honor of 
Brook Taylor and of Joseph L. Taylor,
around an arbitrary
matrix point in their domain.
This motivates us to generalize realizations 
as in Theorem 
\ref{thm:realization}
to the
case where the centre is a 
$d-$tuple of matrices
rather than 
$\ul{0}$ 
or a 
$d-$tuple of scalars.
\\

This is the first in a series of papers with the goal
of generalizing the theory of (Fornasini--Marchesini) realizations centred at $\ul{0}$ (or at a scalar point), to the case of (Fornasini--Marchesini) realizations centred at an arbitrary matrix point in the domain of regularity of a nc rational function. 
In particular, we present a generalization of Theorem
\ref{thm:realization} 
(see Theorem  
\ref{thm:realization2} below) 
namely the existence and uniqueness of a minimal realization,
together with the inclusion of the domain of the nc rational function in the domain of any of its minimal realizations. (The other inclusion and hence the equality of the two domains is presented in a follow-up paper
\cite{PV2})

Other types of realizations of nc rational functions that 
are not necessary regular at 
$\ul{0}$ 
have been considered in 
\cite{CR1,CR2}  
and in 
\cite{V2}, 
see also the recent papers
\cite{Schr1,Schr2,Schr3,Schr4}.
We will consider further the relation between our representations and those of
\cite{CR1,CR2}
in our follow-up paper
\cite{PV2}.
\\

\uline{Here is an \textbf{outline} of the paper}:
In Section
\ref{sec:preliminaries} we give
some preliminaries on nc rational functions and evaluations over general algebras.

In Section
\ref{sec:real} we present the setting of nc Fornasini--Marchesini realizations centred at a matrix point
$\ul{Y}\in(\KK^{\sts})^d$ and generalize classical results which are well known in the scalar case ($s=1$) to the case where $s\ge1$. We prove, using synthesis, the existence of such realizations for any nc rational \textbf{expression} (Theorem
\ref{thm:21Ap17a}), and
introduce the terms of observability and controllability (Subsection
\ref{subsec:ContObs}) 
analogously to the scalar case as in
\cite{KAF}.
The uniqueness of minimal realizations, up to unique similarity, is then proved
(Theorems \ref{thm:2May16b} and \ref{thm:9May19a}), followed by
 a Kalman decomposition argument
(Theorem
\ref{thm:2May16a}). An example of an explicit construction of a minimal realization is presented in Subsection
\ref{subsec:ex} for the nc rational expression 
$(x_1x_2-x_2x_2)^{-1}$.
During the whole section we carry on 
the results also in a more generalized settings 
of evaluations  w.r.t arbitrary unital stably finite 
$\KK-$algebra; as a corollary we obtain 
a new proof of a theorem  of Cohn that equivalence of two rational expressions over matrices implies their equivalence over all stably finite algebras 
(Theorem \ref{thm:22Oct18a}). 
Finally, in Subsection
\ref{subsec:McMillan}
we define the McMillan degree of a nc rational expression using  minimal Fornasini--Marchesini realizations and show that it does not depend on the centre of the realization.

Section 
\ref{sec:main}
contains the main result of the paper, that is a partial  generalization of
Theorem
\ref{thm:realization} 
for  nc rational \textbf{functions} not necessary regular at a scalar point: \begin{introthm}{Theorem}
[Theorem
\ref{thm:30Jan18c}, Corollary \ref{cor:13Feb18a}]
\label{thm:realization2}
If 
$\fR$ 
is a nc rational function of 
$x_1,\ldots,x_d$ 
over 
$\KK$, 
then for every 
$\ul{Y}=(Y_1,\ldots,Y_d)\in dom_s(\fR)$ 
there exists a unique 
(up to unique similarity) 
minimal (observable and controllable) 
nc Fornasini--Marchesini realization
\begin{align*}
\cR_{\cF\cM}(X_1,\ldots,X_d)= D+C
\Big(I_{L}-\sum_{k=1}^d 
\bA_k(X_k-Y_k)\Big)^{-1}
\sum_{k=1}^d \bB_k(X_k- Y_k)
\end{align*}
centred at 
$\ul{Y}$, 
such that for every 
$m\in\NN$ 
and 
$(X_1,\ldots,X_d)\in dom_{sm}(\fR)$:
\begin{multline*}
\fR(X_1,\ldots,X_d)
=I_m\otimes D+(I_m\otimes C)
\Big(I_{Lm}-\sum_{k=1}^d 
(X_k-I_m\otimes Y_k)\bA_k\Big)^{-1}
\sum_{k=1}^d (X_k-I_m\otimes Y_k)\bB_k.
\end{multline*}
Moreover, using the realization 
$\cR_{\cF\cM}$
we can evaluate $\fR$ 
on every matrix point in the domain of regularity of
$\fR$
as well as w.r.t any unital stably finite $\KK-$algebra.
\end{introthm}
The strength 
of Theorem
\ref{thm:30Jan18c}
is that we can evaluate any nc rational function on all of its domain and also w.r.t any unital stably finite
$\KK-$algebra, by using a minimal realization of any nc rational expression which represents the function, that is centred at any point from its domain. %
As a corollary (Corollary \ref{cor:14Oct18a}) we  provide a
proof of Theorem 
\ref{thm:realization} 
which--- unlike the
original proof in 
\cite{KV4}---
does not make any use of the difference-differential calculus of nc functions, but  only the results from Sections
\ref{sec:real}
and \ref{sec:main}.

Generalizations of the main results from Sections 
\ref{sec:real}
and
\ref{sec:main}
to the matrix valued nc rational functions  are briefly summarized in Section
\ref{sec:mvf}. 

Finally, in Section
\ref{sec:hermitian}
we provide a full and precise parameterization ((\ref{eq:1Feb19a}) in Theorem \ref{thm:22Feb19a})
of hermitian nc rational functions in terms of their minimal nc Fornasini--Marchesini realizations centred at a matrix point. A short discussion and some parameterizations are given for descriptor realizations as well.
\\
One of the difficulties which arises when
moving from a scalar to a matrix centre,  is that a minimal nc Fornasini--Marchesini realization 
$\cR_{\cF\cM}$
of a nc rational expression is no longer a nc rational expression 
by itself (cf. Remark
\ref{rem:22Sep18a}).  
However, in the sequel paper
\cite{PV2},
we show that under some constraints (called the linearized lost abbey conditions) on the coefficients of the realization---
which follow immediately when $\cR_{\cF\cM}$ 
is a minimal nc Fornasini--Marchesini realization of a nc rational expression---
$\cR_{\cF\cM}$ 
is actually the restriction of a nc rational function
$\fR$ with 
$DOM_s(\cR_{\cF\cM})=dom_s(\fR)$. This will imply  the opposite  inclusion of the domains in Theorem 
\ref{thm:realization2} 
and thereby complete the proof that the domain of a nc rational function coincides with the domain of any of its minimal realizations, centred at an arbitrary matrix point. 
As a corollary, also in 
\cite{PV2}, we will prove that the domain of a nc rational function is equal to its stable extended domain.
In a slightly different direction, we will use the the theory of realizations with a matrix centre developed in this paper, together with the results from %
\cite{PV2},  
to present
an explicit construction of the free skew field
$\KK\plangle\ul{x}\prangle$,
with a self-contained proof that it is the universal skew field of fractions of the ring of nc polynomials. 
Moreover, we will construct a functional model and use it to provide a different one step proof for the existence of a realization formula for nc rational functions, without using synthesis. Furthermore, we will establish a generalization of the Kronecker--Fliess theorem, which gives a full characterization of nc rational functions in terms of their formal nc generalized power series expansions around a matrix point. These results will appear in
\cite{PV3}.
\\
Finally, we point out that instead 
of working with Fornasini--Marchesini realizations
(for the settings in the 
commutative original version see
\cite{FM1,FM2}) 
one can consider structured realizations
as in 
\cite{BGM1} and obtain similar results. 
This is true also for descriptor realizations; for more details see Remark
\ref{rem:9May19b}.
\\\\

\textbf{Acknowledgments.}
The authors would like to thank Joseph Ball, 
Bill Helton,
Dmitry Kalyuzhnyi-Verbovetskyi, Roland Speicher  and Juri Vol\v{c}i\v{c}
for their helpful comments and discussions.
The authors would also like to thank the referees for their valuable comments which helped to
improve the paper.
\section{Preliminaries}
\label{sec:preliminaries}
\textbf{Notations:} 
$d$ 
will stand for the number of non-commuting variables, 
which will be usually denoted by
$x_1,\ldots,x_d$, we often abbreviate non-commuting by nc.
For a positive integer 
$d$, 
we denote by 
$\cG_d$ 
the free monoid generated by 
$d$ 
generators 
$g_1,\ldots,g_d$,
we say that a word 
$\omega=g_{i_1}\ldots g_{i_\ell}\in\cG_d$ 
is of length 
$|\omega|=\ell$
if $\ell\ge1$ and 
$\omega=\emptyset$ 
is of length 
$0$.
For a field
$\KK$ 
and 
$n\in\NN$, let 
$\KK^{\ntn}$ 
be the vector space of 
$\ntn$ 
matrices over 
$\KK$, 
let 
$\{\ul{e}_1,\ldots,\ul{e}_n\}$ 
be the standard basis of 
$\KK^n$ 
and let 
$\cE_n=\left\{ E_{ij}=\ul{e}_i\ul{e}_j^T: 
1\le i,j\le n\right\}$ 
be the standard basis of 
$\KK^{\ntn}$.
The tensor (Kronecker) product of two matrices 
$P\in\KK^{n_1\times n_2}$
and
$Q\in\KK^{n_3\times n_4}$ 
is  the 
$n_1n_3\times n_2n_4$ 
block matrix 
$P\otimes Q=
\begin{bmatrix}
p_{ij}Q\\
\end{bmatrix}_{1\le i\le n_1,1\le j\le n_2}.$  
The range of a matrix 
$P$, 
that is the span of all of its columns, denoted by 
$\Ima(P)$.

We denote operators on matrices by bold letters such as 
$\bA,\bB$, 
and the action of 
$\bA$ 
on 
$X$ by 
$\bA(X)$. If $\bA$ is defined on 
$\sts$ matrices we extend 
$\bA$ to act on
$sm\times sm$ 
matrices for any 
$m\in\NN$, 
by viewing an
$sm\times sm$ 
matrix 
$X$ 
as an 
$\mtm$ 
matrix with 
$\sts$ 
blocks and by evaluating
$\bA$ 
on the 
$\sts$ 
blocks; In that case we denote the evaluation by 
$(X)\bA$. 
If 
$C$ 
is a constant matrix and 
$\bA$ 
is an operator, then 
$C\cdot\bA$ 
and 
$\bA\cdot C$ 
are two operators, defined by  
$(C\cdot\bA)(X):=C\bA(X)$ 
and 
$(\bA\cdot C)(X):=\bA(X) C$.
For every 
$n_1,n_2\in\NN$, we define the permutation matrix
$$E(n_1,n_2)=
\begin{bmatrix}
E_{ij}^T\\
\end{bmatrix}_{1\le i\le n_1,1\le j\le n_2}\in\KK^{n_1n_2
\times n_1n_2}$$
and use these matrices to change the order of factors in the Kronecker product of two matrices by the following rule 
\begin{align}
\label{eq:24Jan19a}
P\otimes Q=E(n_1,n_3)(Q\otimes P)E(n_2,n_4)^T,
\end{align}
for all
$n_1,n_2,n_3,n_4\in\NN,\,
Q\in\KK^{n_1\times n_2}
$
and
$P\in\KK^{n_3\times n_4}$;
for more details
see 
\cite[pp. 259--261]{HJ}. 
If 
$P=
\begin{bmatrix}
P_{ij}\\
\end{bmatrix}_{1\le i,j\le m}\in(\KK^{\sts})^{\mtm}$
and
$Q=
\begin{bmatrix}
Q_{ij}\\
\end{bmatrix}_{1\le i,j\le m}
\in(\KK^{\sts})^{\mtm}$, then we use the notation
$$P\odot_s Q:=\begin{bmatrix}
\sum_{k=1}^m P_{ik}\otimes Q_{kj}
\end{bmatrix}_{1\le i,j\le m}$$
for the so-called faux product of 
$P$ and $Q$,
viewed as
$\mtm$ matrices over the tensor algebra of 
$\KK^{\sts}$,
where 
$P_{ik}\otimes Q_{kj}$ 
denotes the element of 
$\KK^{\sts}\otimes \KK^{\sts}$, rather than the Kronecker product of the matrices; 
see 
\cite[page 241]{Pau} 
for the exact definition and 
\cite{Ef} 
for its origins in operator spaces.
If 
$\ul{X}=(X_1,\ldots, X_d)\in(\KK^{sm\times sm})^d$ 
and 
$\omega=g_{i_1}\ldots g_{i_{\ell}}$, then
\begin{align}
\label{eq:21Mar19a}
\ul{X}^{\odot_s\omega}:=X_{i_1}\odot_s\cdots\odot_s X_{i_{\ell}}.
\end{align}

We use
$\fR,\cR,R$ and
$\fr$ 
for nc rational function, nc Fornasini--Marchesini realization, nc
rational expression and matrix valued nc rational function, respectively.
Likewise, we use $\fa$ to denote elements in an algebra $\cA$
and
$\fA$ to denote matrices over $\cA$. 
Throughout the paper, we use underline to denote vectors or $d-$tuples.
\subsection{NC rational functions}
If 
$\cV$ 
is a vector space over a field
$\KK$, 
then  
$\cV_{nc}$, 
the nc space over 
$\cV$, 
consists of all square matrices over 
$\cV$, 
i.e., 
$$\cV_{nc}=\coprod_{n=1}^\infty 
\cV^{\ntn}.$$
For every 
$\Omega\subseteq \cV_{nc}$ 
and
$n\in\NN$ 
we use the notation 
$\Omega_n:=\Omega\cap \cV^{\ntn}$.
A subset 
$\Omega\subseteq\cV_{nc}$ 
is called a nc set if it is closed under direct sums, i.e., if
$X\in\Omega_n,Y\in\Omega_m$
then  
$X\oplus Y:=
\begin{bmatrix}
X&0\\0&Y\\
\end{bmatrix}\in\Omega_{n+m},
\forall m,n\in\NN$. 
In the special case where 
$\cV=\KK^d$, 
we have the identification
\begin{equation*}
\left(\KK^d\right)_{nc}=\coprod_{n=1}^\infty
(\KK^d)^{\ntn}\cong
\coprod_{n=1}^\infty 
(\KK^{\ntn})^d,
\end{equation*}
that is the nc space of all 
$d-$tuples of square matrices over 
$\KK$.
Let 
$\cV,\cW$ 
be vector spaces over a field 
$\KK$ 
and 
$\Omega\subseteq\cV_{nc}$ 
be an nc set, then 
$f:\Omega\rightarrow\cW_{nc}$ 
is called a
\textbf{nc function} if
$f$ 
is graded, i.e., if
$n\in\NN$
and
$X\in\Omega_n$, 
then
$f(X)\in\cW^{\ntn},$ and
\begin{itemize}
\item[1.]
$f$ 
respects direct sums, i.e., if 
$X,Y\in\Omega$, 
then
$f(X\oplus Y)=f(X)\oplus f(Y);$
\item[2.]
$f$ respects similarities, i.e.,  
if
$n\in\NN
,\,X\in\Omega_n$
and
$T\in\KK^{\ntn}$ 
is invertible such that 
$T\cdot X\cdot T^{-1}\in\Omega$, 
then
$f(T\cdot X\cdot T^{-1})=
T\cdot f(X)\cdot T^{-1}$.
\end{itemize}
Notice that if 
$X\in\Omega_n$ 
and 
$T\in\KK^{\ntn}$, 
by the products 
$T\cdot X$ 
and 
$X\cdot T$ 
we mean the standard matrix multiplication and we use the action of 
$\KK$ 
on 
$\cV$. 
In particular, if 
$\cV=\KK^d,\,
\ul{X}=(X_1,\ldots,X_d)\in(\KK^{\ntn})^d$
and
$T\in\KK^{\ntn}$, 
the products are given by
\begin{equation*}
T\cdot\ul{X}:=(TX_1,\ldots,TX_d)
\text{ and }\ul{X}\cdot 
T:=(X_1T,\ldots,X_dT).
\end{equation*}

An important and central example of nc
 functions are \textbf{nc rational expressions}. We denote by
$\KK\langle x_1,\ldots,x_d\rangle$ 
the 
$\KK-$algebra of nc polynomials in the 
$d$ 
nc variables 
$x_1,\ldots,x_d$ 
over 
$\KK$.
We obtain nc rational expressions by applying 
successive arithmetic operations 
(addition, multiplication and taking inverse)
on 
$\KK\langle x_1,\ldots,x_d\rangle$.
For a nc rational expression 
$R$ 
and 
$n\in\NN$, 
let 
$dom_n(R)$
be the set of all 
$d-$tuples of 
$\ntn$ 
matrices over 
$\KK$ 
for which all the inverses in 
$R$ 
exist; 
the \textbf{domain of regularity} of 
$R$ 
is then defined by
\begin{align*}
dom(R):=
\coprod_{n=1}^\infty dom_n(R).
\end{align*} 
A nc rational expression 
$R$
is called non-degenerate if
$dom(R)\ne\emptyset$.
For example,
$R(\ul{x})=
\left(x_2+(1-x_1)^{-1}
(x_3^{-1}x_1-x_2)\right)x_1$
is a nc rational expression in 
$x_1,x_2,x_3$, 
while its domain of regularity is given by
\begin{equation*}
dom(R)=\coprod_{n=1}^\infty
\big\{(X_1,X_2,X_3)
\in\left(\KK^{\ntn}\right)^3:
\det(I_n-X_1),\det(X_3)\ne0\big\}.
\end{equation*}
Every nc rational expression 
$R$ 
is a nc function from 
$dom(R)\subseteq(\KK^d)_{nc}$ 
to 
$\KK_{nc}$. For a detailed discussion of nc rational 
expressions and their domains of
regularity, see 
\cite{KV1}.

What comes now is the definition of a nc rational function.
Let 
$R_1$ 
and 
$R_2$ 
be nc rational expressions in 
$x_1,\ldots,x_d$ 
over 
$\KK$. 
We say that 
$R_1$ 
and 
$R_2$ 
are 
$(\KK^d)_{nc}-$evaluation equivalent,
if
$R_1(\ul{X})=R_2(\ul{X})$
for every 
$\ul{X}\in 
dom(R_1)\cap dom(R_2)$.
A 
\textbf{nc rational function} is an 
equivalence class of non-degenerate nc 
rational expressions. For every nc rational function 
$\fR$, 
define its 
\textbf{domain of regularity}
\begin{align}
\label{eq:9Mar19a}
dom(\fR):=\bigcup_{R\in\fR} dom(R).
\end{align}
The 
$\KK-$algebra of all nc rational functions of 
$x_1,\ldots,x_d$ 
over 
$\KK$ is denoted by 
$\KK\plangle x_1,\ldots,x_d\prangle$ 
and it is a skew field, called the free 
skew field. Moreover,
$\KK\plangle x_1,\ldots, x_d\prangle$  
is the universal skew field of fractions of 
$\KK\langle x_1,\ldots,x_d\rangle$. See 
\cite{AM66,Be70,Co71a,Co72,Row80} for the original proofs and  
\cite{Co06} 
for a more modern reference, while a proof of the equivalence with the evaluations over matrices is presented in 
\cite{KV3,KV4}. 
\subsection{Evaluations over algebras}
Let
$\cA$ be a unital 
$\KK-$algebra. If
$\ul{\fa}=(\fa_1,\ldots,\fa_d)\in\cA^d$ 
and 
$\omega=g_{i_1}\ldots g_{i_\ell}\in\cG_d$, 
then we use the notations 
$\ul{\fa}^\omega:=\fa_{i_1}\cdots 
\fa_{i_\ell}$ 
and
$\ul{\fa}^{\emptyset}=1_{\cA}$, where $1_{\cA}$ is the unit element in $\cA$. 
We recall the definitions of evaluation and 
domain of nc rational expressions over 
$\cA$. 
For more details see
\cite{HMS}.
\begin{definition}[$\cA-$Domains and Evaluations]
For any nc rational expression 
$R$ 
in 
$x_1,\ldots,x_d$ 
over 
$\KK$, 
its 
$\cA-$domain
$dom^{\cA}(R)\subseteq \cA^d$ 
and its evaluation 
$R^{\cA}(\ul{\fa})$ 
at any
$\ul{\fa}=(\fa_1,\ldots,\fa_d)
\in dom^{\cA}(R)$ 
are defined by:
\begin{enumerate}
\item[1.]
If
$R=\sum_{\omega\in\cG_d}
r_{\omega}\ul{x}^\omega$
is a nc polynomial ($r_\omega\in\KK$), 
then 
\begin{equation*}
dom^{\cA}(R)=\cA^d 
\text{ and } 
R^{\cA}(\ul{\fa})=\sum_{\omega\in\cG_d}
r_\omega \ul{\fa}^\omega.
\end{equation*}
\item[2.]
If 
$R=R_1R_2$
where 
$R_1$
and
$R_2$ 
are nc rational expressions, then 
\begin{equation*}
dom^{\cA}(R)=
dom^{\cA}(R_1)\cap dom^{\cA}(R_2) 
\text{ and } 
R^{\cA}(\ul{\fa})=
R_1^{\cA}(\ul{\fa})R_2^{\cA}(\ul{\fa}).
\end{equation*}
\item[3.]
If 
$R=R_1+R_2$ where $R_1$
and
$R_2$ are nc rational expressions, then 
\begin{equation*}dom^{\cA}(R)=
dom^{\cA}(R_1)\cap dom^{\cA}(R_2) 
\text{ and } 
R^{\cA}(\ul{\fa})=
R_1^{\cA}(\ul{\fa})+ R_2^{\cA}(\ul{\fa}).
\end{equation*}
\item[4.]
If
$R=R_1^{-1}$
where
$R_1$
is a nc rational expression, then
\begin{equation*}dom^{\cA}(R)=
\left\{ \ul{\fa}\in 
dom^{\cA}(R_1): R_1^{\cA}(\ul{\fa})
\text{ invertible in }\cA\right\} 
\text{ and } 
R^{\cA}(\ul{\fa})=
\big(R_1^{\cA}(\ul{\fa})\big)^{-1}.
\end{equation*}
\end{enumerate}
\end{definition}

\begin{remark}
\label{rem:20Mar19a}
Let 
$n\in\NN$
and consider the 
$\KK-$algebra 
$\cA_n=\KK^{\ntn}$. 
Then, it is easily seen that
$dom^{\cA_n}(R)=dom_n(R)$
and 
$R(\ul{\fA})=R^{\cA_n}(\ul{\fA})$
for every 
$\ul{\fA}\in dom_n(R)$.
\end{remark} 

As it will be pointed out later (cf. Theorem
\ref{thm:2May16a}), we are interested in a certain family of algebras, called stably finite algebras. 
A unital 
$\KK-$algebra 
$\cA$ 
is called 
\textbf{stably finite} 
if for every 
$m\in\NN$ 
and 
$\fA,\fB\in\cA^{\mtm}$, we have
\begin{equation*}
\fA\fB=I_m\otimes 1_{\cA}
\Longleftrightarrow
\fB\fA=I_m\otimes 1_{\cA}.
\end{equation*} 
If $\cA$ 
is a unital 
$\CC^*-$algebra with a faithful trace, then 
$\cA$ 
is stably finite. The following is a 
characterization of stably 
finite algebras that we find useful in a later stage of the paper; 
see 
\cite[Lemma 5.2]{HMS}
for its proof.
\begin{lemma}
\label{lem:19Oct18a}
Let
$\cA$ 
be a unital 
$\KK-$algebra. 
The following are equivalent:
\begin{itemize}
\item[1.]
$\cA$ is stably finite.
\item[2.]
For every 
$n\in\NN,\,
m_1,\ldots,m_n\in\NN$ 
and 
$\fA_{i,j}\in\cA^{m_i\times m_j},\, 
i,j=1,\ldots,n$, 
if the upper (or lower) 
triangular block matrix
$$\begin{bmatrix}
\fA_{11}&*&*&*\\
0&\fA_{22}&*&*\\
\vdots&\ddots&\ddots&*\\
0&\ldots&0&\fA_{nn}\\
\end{bmatrix}$$
is invertible, then 
the matrices
$\fA_{11},\ldots,\fA_{nn}$ 
are invertible.
\end{itemize}
\end{lemma}

\section{Realizations of NC Rational 
Expressions}
\label{sec:real}
Non-commutative Fornasini--Marchesini 
realizations, see 
\cite{BGM1,KV4} 
and 
\cite{FM1,FM2} 
for the original commutative version, 
apply to nc rational expressions 
which are regular at 
$\ul{0}$. 
By translation, the point 
$\ul{0}$ 
can be replaced by any scalar point. 
In this section we
develop analogous realization formulas 
for nc rational expressions, centred 
at an arbitrary 
\textbf{matrix point} in the domain 
of regularity of the expression. 
\begin{definition}
\label{def:13Apr19a}
Let 
$s,L\in\NN,\,
\ul{Y}=(Y_1,\ldots,Y_d)\in(\KK^{\sts})^d$, 
\begin{align*}
\bA_1,\ldots,\bA_d:\KK^{\sts}
\rightarrow\KK^{L\times L}
\text{ and }
\bB_1,\ldots,\bB_d:\KK^{\sts}
\rightarrow\KK^{L\times s}
\end{align*} 
be linear mappings, 
$C\in\KK^{s\times L}$ 
and
$D\in\KK^{\sts}$. 
Then
\begin{equation}
\label{eq:18Aug18a}
\cR(X_1,\ldots,X_d)= 
D+C\Big(I_{L}-\sum_{k=1}^d 
\bA_k(X_k-Y_k)\Big)^{-1}
\sum_{k=1}^d 
\bB_k(X_k- Y_k)
\end{equation}
is called a 
\textbf{nc Fornasini--Marchesini realization} 
centred at 
$\ul{Y}$
and it is defined for every 
$\ul{X}=(X_1,\ldots,X_d)\in DOM_s(\cR)$,
where
\begin{equation*}
DOM_s(\cR):=\Big\{ 
\ul{X}\in(\KK^{\sts})^d:\det
\Big(
I_L-\sum_{k=1}^d
\bA_k(X_k-Y_k)
\Big)\ne0
\Big\}.
\end{equation*}
In that case we say that the realization 
$\cR$ 
is \textbf{described by the tuple} 
$(L,D,C,\ul{\bA},\ul{\bB})$.
\end{definition}
\begin{remark}
\label{rem:22Sep18a}
If
$s=1$, 
then 
$\cR$
is a 
$1\times1$ 
matrix valued nc rational expression 
(see Remark 
\ref{rem:20Mar19b} 
for details) and 
$DOM_s(\cR)=dom_s(\cR)$.
However, this is not the case for
$s>1$ 
and that is why we use the notation 
$DOM_s(\cR)$
instead of 
$dom_s(\cR)$. 
\end{remark}
Let
$s_1,s_2,s_3,s_4\in\NN$. 
If 
$\bT:\KK^{s_1\times s_2}
\rightarrow
\KK^{s_3\times s_4}$ 
is a linear mapping and 
$m\in\NN$, 
then 
$\bT$ 
can be naturally extended to a linear mapping 
$\bT:\KK^{s_1m\times s_2m}
\rightarrow
\KK^{s_3m\times s_4m}$, 
by the following rule:
\begin{equation*}  
X=
\begin{bmatrix}
X_{ij}\\
\end{bmatrix}_{1\le i,j\le m}
\in\KK^{s_1m\times s_2m}
\Longrightarrow
(X)\bT=
\begin{bmatrix}
\bT(X_{ij})\\
\end{bmatrix}_{1\le i,j\le m},
\end{equation*}
i.e., 
$(X)\bT$  
is an 
$\mtm$ 
block matrix with entries in 
$\KK^{s_3\times s_4}$.
Therefore, we can extend the realization
(\ref{eq:18Aug18a}) 
to act on 
$d-$tuples of 
$sm\times sm$ 
matrices: for every 
$\ul{X}=(X_1,\ldots,X_d)$ 
in
\begin{equation*}
DOM_{sm}(\cR):=
\Big\{\ul{X}
\in(\KK^{sm\times sm})^d:
\det\Big(I_{Lm}-\sum_{k=1}^d 
(X_k-I_m\otimes Y_k)
\bA_k\Big)\ne0
\Big\},
\end{equation*}
define
\begin{equation*}
\cR(\ul{X}):=
I_m\otimes D+
(I_m\otimes C)
\Big(
I_{Lm}-\sum_{k=1}^d 
(X_k-I_m\otimes Y_k)
\bA_k\Big)^{-1}
\sum_{k=1}^d 
(X_k-I_m\otimes Y_k)\bB_k.
\end{equation*}

In addition, if
$\cA$
is a unital 
$\KK-$algebra, a linear mapping
$\bT:\KK^{s_1\times s_2}\rightarrow
\KK^{s_3\times s_4}$
can be also naturally extended to a linear mapping
$\bT^{\cA}:\cA^{s_1\times s_2}
\rightarrow\cA^{s_3\times s_4}$ 
by the following rule:
\begin{equation*}
\fA
=\sum_{i=1}^{s_1}
\sum_{j=1}^{s_2} 
E_{ij}\otimes \fa_{ij}
\in\cA^{s_1\times s_2}
\Longrightarrow (\fA)\bT^{\cA}
=\sum_{i=1}^{s_1}
\sum_{j=1}^{s_2}
\bT(E_{ij})
\otimes \fa_{ij}\in
\cA^{s_3\times s_4},
\end{equation*}
where
$E_{ij}=\ul{e}_i\ul{e}_j^T
\in\KK^{s_1\times s_2}$
and
$\fa_{ij}\in\cA$.
If 
$\cR$ 
is a nc Fornasini--Marchesini realization centred at 
$\ul{Y}$, 
as in 
(\ref{eq:18Aug18a}), 
define its 
\textbf{$\cA-$domain} 
to be the subset of 
$(\cA^{\sts})^d$
given by
\begin{equation*}
DOM^{\cA}(\cR)
:=\Big\{
\ul{\fA}\in(\cA^{\sts})^d:
\Big(I_L\otimes 1_{\cA}-\sum_{k=1}^d
(\fA_k-Y_k\otimes 1_{\cA})
\bA_k^{\cA}\Big)
\text{ is invertible in }
\cA^{L\times L}\Big\}
\end{equation*}
and for every 
$\ul{\fA}=(\fA_1,\ldots,\fA_d)\in 
DOM^{\cA}(\cR)$ 
define 
the evaluation of 
$\cR$ 
at 
$\ul{\fA}$
by 
\begin{multline*}
\cR^{\cA}(\ul{\fA}):=
D\otimes 1_\cA
+(C\otimes 1_\cA)
\Big(I_L\otimes 
1_\cA-\sum_{k=1}^d 
\big[(\fA_k)\bA_k^{\cA}-
\bA_k(Y_k)\otimes 1_\cA\big]
\Big)^{-1}\\
\sum_{k=1}^d 
\big[(\fA_k)\bB_k^{\cA}-
\bB_k(Y_k)\otimes 1_\cA\big].
\end{multline*}
\subsection{Existence}
\label{subsec:synthesis}
The way we define what is a realization of 
a nc rational expression is
different than the usual definition. 
In the usual case, the expression and the 
realization coincide whenever they are 
both defined, while in our definition we include the fact 
that the domain of the expression is 
contained in the domain of the realization. 
We begin with the definition of a nc 
rational expression 
admitting a realization, both in the 
usual way (over matrices) and in 
the case of evaluations w.r.t an algebra. 
\begin{definition}
\label{def:25Sep18a}
Let 
$R$ 
be a nc rational expression in
$x_1,\ldots,x_d$ 
over 
$\KK,\,\ul{Y}=(Y_1,\ldots,Y_d)
\in dom_s(R)$,
$\cR$ 
be a nc Fornasini--Marchesini realization centred at
$\ul{Y}$ 
and 
$\cA$ 
be a unital
$\KK-$algebra. 
We say that:
\begin{enumerate}
\item[1.] 
$R$ 
\textbf{admits the realization} 
$\cR$, 
or that 
$\cR$ 
is a realization of 
$R$, 
if 
\begin{align*}
dom_{sm}(R)\subseteq DOM_{sm}(\cR)
\text{ and }R(\ul{X})=\cR(\ul{X}),\,
\forall \ul{X}\in dom_{sm}(\cR)
\end{align*} 
for every
$m\in\NN$.
\item[2.]
$R$
\textbf{admits the realization 
$\cR$ 
with respect to (w.r.t)}
$\cA$,
or that
$\cR$ 
is a realization of 
$R$ 
w.r.t 
$\cA$, 
if for every
$\ul{\fa}=(\fa_1,\ldots,\fa_d)\in 
dom^{\cA}(R)$:
\begin{align*}
I_s\otimes\ul{\fa}:=
(I_s\otimes\fa_1,\ldots,I_s\otimes\fa_d)
\in DOM^{\cA}(\cR)
\end{align*}
and
$I_s\otimes 
R^{\cA}(\ul{\fa})=
\cR^{\cA}(I_s\otimes\ul{\fa})$.
\end{enumerate}
\end{definition}
We begin by showing the existence of 
a nc Fornasini--Marchesini
realization for every nc rational 
expression 
$R$, 
centred at any 
$\ul{Y}\in dom_s(R)$, 
that is also a realization of 
$R$
w.r.t any unital
$\KK-$algebra.
\begin{theorem}
\label{thm:21Ap17a}
Let 
$R$
be a nc rational expression in
$x_1,\ldots,x_d$ 
over 
$\KK$
and let 
$\ul{Y}=(Y_1,\ldots,Y_d)\in dom_s(R)$. 
There exists a nc Fornasini--Marchesini realization 
$\cR$ 
of
$R$
centred at 
$\ul{Y}$, 
such that 
$\cR$
is a realization of 
$R$
w.r.t any 
unital $\KK-$algebra. 
\end{theorem}
The proof is done by synthesis, which is going back to ideas from automata theory
\cite{BR,Schutz61,Schutz62b}
and system theory  
\cite{CR1,CR2}.
We also use the following technical fact: let
$\ul{X}=(X_1,\ldots,X_d)
\in\left(
\KK^{sm\times sm}\right)^d$ 
and write
\begin{align}
\label{eq:15Oct18a}
X_k=\sum_{i,j=1}^m
E_{ij}\otimes 
X^{(k)}_{ij},
\text{ with }E_{ij}\in\KK^{\mtm},\,
X^{(k)}_{ij}\in\KK^{\sts},
\,1\le k\le d
\end{align}
then
$(X_k-I_m\otimes Y_k)\bA_k=
\sum_{i,j=1}^m
E_{ij}
\otimes
\bA_k
(X^{(k)}_{ij})
-I_m\otimes
\bA_k(Y_k)$ 
and hence
\begin{align*}
(X_k-I_m\otimes Y_k)\bA_k
=P_{2}\Big(
\sum_{i,j=1}^m
\bA_k(X^{(k)}_{ij})
\otimes E_{ij}-
\bA_k(Y_k)\otimes I_m
\Big)P_{2}^{-1}
\end{align*}
and similarly
\begin{align*}
(X_k-I_m\otimes Y_k)\bB_k=P_{2}
\Big(
\sum_{i,j=1}^m
\bB_k
(X^{(k)}_{ij} )\otimes E_{ij}-
\bB_k(Y_k)
\otimes I_m 
\Big)P_1^{-1},
\end{align*}
where 
$P_1=E(m,s)$ and $P_2=E(m,L)$ 
are the shuffle matrices defined in
(\ref{eq:24Jan19a}). 
Therefore for every 
$\ul{X}\in DOM_{sm}(\cR)$ 
we have
\begin{multline}
\label{eq:17Feb19a}
P_1^{-1}\cR(\ul{X})P_1=
D\otimes I_m+P_1^{-1}(I_m\otimes C)
P_{2}
\Big(I_{Lm}-\sum_{k=1}^d
\Big[
\sum_{i,j=1}^m
\bA_k(X^{(k)}_{ij})
\otimes E_{ij}-
\bA_k(Y_k)\otimes I_m\Big]
\Big)^{-1}\\
\sum_{k=1}^d
\Big[\sum_{i,j=1}^m
\bB_k
(X^{(k)}_{ij} 
)
\otimes E_{ij}-
\bB_k(Y_k)\otimes I_m\Big]
=D\otimes I_m+(C\otimes I_m)
\Big(I_{Lm}-\sum_{k=1}^d
\Big[
\sum_{i,j=1}^m\\
\bA_k(X^{(k)}_{ij})
\otimes E_{ij}-
\bA_k(Y_k)\otimes I_m\Big]
\Big)^{-1}
\sum_{k=1}^d
\Big[\sum_{i,j=1}^m
\bB_k
(X^{(k)}_{ij})\otimes E_{ij}-
\bB_k(Y_k)\otimes I_m\Big].
\end{multline}

\begin{proof} 
We first show that the theorem is true for all monomials 
$x_1,\ldots,x_d$ 
and constants, then  we show that if 
it is true for two rational expressions, so it is also true for 
their summation, multiplication and their inversion, 
when they exist. 

\textbf{\uline{1.\,Constants:}}
Let 
$R_0(\ul{x})=K\in\KK$. 
If
$\ul{\fa}\in 
dom^{\cA}(R_0)
=\cA^d$, 
then 
$I_s\otimes 
\ul{\fa}\in DOM^{\cA}(\cR_0)
=(\cA^{\sts})^d$,
where the realization 
$\cR_0$ 
is centred at 
$\ul{Y}\in dom_s(R_0)=
(\KK^{\sts})^d$ 
and described by
\begin{equation}
\label{eq:20Mar19a}
L=1,\,D=I_s\otimes K,\,
C=0,\,
\bA_1=\ldots=\bA_d=0,\,
\bB_1=\ldots=\bB_d=0
\end{equation} 
and
$I_s\otimes R_0^{\cA}(\ul{\fa})
=I_s\otimes (K\otimes 1_{\cA})
=D\otimes 1_{\cA}=
\cR_0^{\cA}(I_s\otimes 
\ul{\fa})$.
Moreover,
\begin{align*}
dom_{sm}(R_0)=
(\KK^{sm\times sm})^d=
DOM_{sm}(\cR_0)
\end{align*}
and for every
$\ul{X}\in dom_{sm}(R_0)$ 
we have
$R_0(\ul{X})=
I_{sm}\otimes K
=I_m\otimes D
=\cR_0(\ul{X})$.

 \textbf{\uline{2.\,Monomials:}}
Let
$R_j(\ul{x})=x_j$ 
for 
$1\le j\le d$.
If
$\ul{\fa}\in dom^{\cA}(R_j)
=\cA^d$, 
then
$I_s\otimes\ul{\fa}
\in DOM^{\cA}(\cR_j)
=(\cA^{\sts})^d$,
where the realization 
$\cR_j$ 
is centred at 
$\ul{Y}\in dom_s(R_j)=
(\KK^{\sts})^d$ 
and described by
\begin{equation}
\label{eq:20Mar19b}
L=s,\,
D=Y_j,\,
C=I_s,\,
\bA_1=\ldots=\bA_d=0,\, 
\bB_j=Id,\, 
\bB_k=0\,
(\forall k\ne j)
\end{equation} 
and 
$$I_s\otimes 
 R_j^{\cA}(\ul{\fa})
 =I_s\otimes \fa_j
=Y_j\otimes 1_{\cA}+
(I_s\otimes \fa_j-Y_j\otimes 1_{\cA})
\bB_j^{\cA}=
\cR_j^{\cA}(I_s\otimes\ul{\fa}).$$

Moreover,
$dom_{sm}(R_j)=
(\KK^{sm\times sm})^d
=DOM_{sm}(\cR_j)$
and for every
$\ul{X}\in dom_{sm}(R_j)$ we have
$R(\ul{X})=
X_j=I_m\otimes Y_j
+(X_j-I_m\otimes Y_j)\bB_j
=\cR_j(\ul{X})$.

 \textbf{\uline{3.\,Addition:}}
Suppose 
$R_1$ 
and 
$R_2$ 
are two nc rational expressions admitting  realizations 
$\cR_1$ 
and 
$\cR_2$ both centred at 
$\ul{Y}$, 
described by the tuples 
$(L_1,D^1,C^1,\ul{\bA}^1,\ul{\bB}^1)$ 
and 
$(L_2,D^2,C^2,\ul{\bA}^2,\ul{\bB}^2)$, 
respectively, and also w.r.t 
any unital
$\KK-$algebra 
$\cA$. 

Thus,
$\ul{\fa}\in dom^{\cA}(R_1+R_2)
=dom^{\cA}(R_1)\cap
dom^{\cA}(R_2)$
implies that
$I_s\otimes \ul{\fa}\in DOM^{\cA}
(\cR_1)\cap DOM^{\cA}(\cR_2)$, 
\begin{multline*}
I_s\otimes 
(R_1+R_2)^{\cA}(\ul{\fa})=
\cR_1^{\cA}(I_s\otimes \ul{\fa})+
\cR_2^{\cA}(I_s\otimes \ul{\fa})
=(\cR_1^{\cA}+\cR_2^{\cA})
(I_s\otimes \ul{\fa})=
D^{\text{par}}\otimes 
1_{\cA}
\\+(C^{\text{par}}
\otimes 1_{\cA})
\Big(
I_{L}\otimes 1_{\cA}-
\sum_{k=1}^d 
\big[\bA_k^{\text{par}}
(I_s)\otimes\fa_k-
\bA_k^{\text{par}}(Y_k)\otimes 1_{\cA}
\big]\Big)^{-1}
\sum_{k=1}^d \big[
\bB_k^{\text{par}}
(I_s)\otimes\fa_k-\\
\bB_k^{\text{par}}(Y_k)
\otimes 1_{\cA}\big]:=
\left(
\cR^{\text{par}}
\right)^{\cA}
(I_s\otimes\ul{\fa})
\end{multline*}
and
$I_s\otimes\ul{\fa}\in 
DOM^{\cA}(\cR^{\text{par}})$,
when 
$\cR^{\text{par}}$
is the nc Fornasini--Marchesini realization centred at
$\ul{Y}$ 
described by
\begin{multline}
\label{eq:20Mar19c}
L=L_1+L_2,\,
D^{\text{par}}=D^1+D^2,\, 
C^{\text{par}}=
\begin{bmatrix}
C^1&C^2\\
\end{bmatrix},\,
\bA_k^{\text{par}}
=\begin{bmatrix}
\bA_k^1&0\\
0&
\bA_k^2\\
\end{bmatrix}\\
\text{ and }
\bB_k^{\text{par}}=
\begin{bmatrix}
\bB_k^1\\
\bB_k^2\\
\end{bmatrix},\,
k=1,\ldots,d
\end{multline}

Also, for every 
$m\in\NN$,
$\ul{X}\in  dom_{sm}(R_1+R_2)= dom_{sm}(R_1)
\cap dom_{sm}(R_2)$
implies
$\ul{X}\in DOM_{sm}(\cR_i)$
and 
$R_i(\ul{X})=\cR_i(\ul{X})$
for 
$i=1,2$
and hence
$(R_1+R_2)(\ul{X})=
\cR_1(\ul{X})+\cR_2(\ul{X})$.
Write 
$\ul{X}=(X_1,\ldots,X_d)
\in(\KK^{sm\times sm})^d$ 
and use 
(\ref{eq:17Feb19a})
to obtain that 
\begin{multline*}
P_1^{-1}\left(\cR_1(\ul{X})
+\cR_2(\ul{X})\right)P_1
=
D^{\text{par}}\otimes I_m+(C^{\text{par}}\otimes I_m)
\Big(I_{Lm}-\sum_{k=1}^d 
\Big[\sum_{i,j=1}^m
\bA_k^{\text{par}}
(X^{(k)}_{ij})
\\ \otimes E_{ij}-
\bA_k^{\text{par}}(Y_k)\otimes I_m
\Big]
\Big)^{-1}
\sum_{k=1}^d
\Big[
\sum_{i,j=1}^m
\bB_k^{\text{par}}(
X^{(k)}_{ij} )
\otimes E_{ij}-
\bB_k^{\text{par}}(Y_k)
\otimes I_m \Big]
=D^{\text{par}}\otimes I_m
\\+(C^{\text{par}}\otimes I_m)P_{2}^{-1}
\Big(
I_{Lm}-
\sum_{k=1}^d
\Big[ \sum_{i,j=1}^m
E_{ij}\otimes
\bA_k^{\text{par}}(X^{(k)}_{ij})
-I_m\otimes 
\bA_k^{\text{par}}(Y_k)
\Big]\Big)^{-1}P_{2}
P_{2}^{-1}
\Big(
\sum_{k=1}^d\\
\Big[\sum_{i,j=1}^m
E_{ij}\otimes 
\bB_k^{\text{par}}
(X^{(k)}_{ij})-I_m\otimes
\bB_k^{\text{par}}(Y_k)\Big]
\Big)P_1
=P_1^{-1}\cR^{\text{par}}(\ul{X})P_1,
\end{multline*}
i.e., 
$\cR_1(\ul{X})+\cR_2(\ul{X})
=\cR^{\text{par}}(\ul{X})$ 
while it is easily seen that 
$\ul{X}\in DOM_{sm}\left(
\cR^{\text{par}}\right)$.

\textbf{\uline{4.\,Multiplication:}}
Suppose 
$R_1$ 
and 
$R_2$ 
are two nc rational expressions admitting  realizations 
$\cR_1$ 
and 
$\cR_2$ 
both centred at 
$\ul{Y}$, 
described by the tuples 
$(L_1,D^1,C^1,\ul{\bA}^1,\ul{\bB}^1)$ 
and 
$(L_2,D^2,C^2,\ul{\bA}^2,\ul{\bB}^2)$, 
respectively, and also w.r.t any unital
$\KK-$algebra
$\cA$.

Thus, 
$\ul{\fa}\in dom^{\cA}(R_1R_2)
=dom^{\cA}(R_1)\cap
dom^{\cA}(R_2)$
implies that
$I_s\otimes 
\ul{\fa}\in DOM^{\cA}
(\cR_1)\cap DOM^{\cA}(\cR_2)$, 
\begin{multline*}
I_s\otimes 
(R_1R_2)^{\cA}(\ul{\fa})=
\cR_1^{\cA}(I_s\otimes \ul{\fa})
\cR_2^{\cA}(I_s\otimes \ul{\fa})
=(\cR_1\cR_2)^{\cA}(I_s\otimes\ul{\fa})
=D^{\text{ser}}\otimes 1_{\cA}
+(C^{\text{ser}}\otimes 1_{\cA})\\
\Big(
I_{L}\otimes 
1_{\cA}-\sum_{k=1}^d 
\big[
\bA_k^{\text{ser}}(I_s)
\otimes\fa_k
-
\bA_k^{\text{ser}}(Y_k)
\otimes 1_{\cA}\big]
\Big)^{-1}
\sum_{k=1}^d 
\big[\bB_k^{\text{ser}}
(I_s)\otimes\fa_k-
\bB_k^{\text{ser}}(Y_k)
\otimes 1_{\cA}
\big]
\\:=\left(\cR^{\text{ser}}
\right)^{\cA}(I_s\otimes\ul{\fa})
\end{multline*}
and
$I_s\otimes\ul{\fa}\in 
DOM^{\cA}(\cR^{\text{ser}})$, 
when 
$\cR^{\text{ser}}$
is the nc Fornasini--Marchesini realization centred at
$\ul{Y}$ described by
\begin{multline}
\label{eq:20Mar19d}
L=L_1+L_2,\,
D^{\text{ser}}=D^1D^2, \,
C^{\text{ser}}=
\begin{bmatrix}
C^1&D^1C^2\\
\end{bmatrix},
\bA_k^{\text{ser}}
=\begin{bmatrix}
\bA_k^1&\bB_k^1\cdot C^2\\
0&
\bA_k^2\\
\end{bmatrix}\\
\text{ and }
\bB_k^{\text{ser}}=
\begin{bmatrix}
\bB_k^1\cdot D^2\\
\bB_k^2\\
\end{bmatrix},\,
k=1,\ldots,d.
\end{multline}

Also,
for every 
$m\in\NN$,
$\ul{X}\in dom_{sm}(R_1R_2)
=dom_{sm}(R_1)
\cap dom_{sm}(R_2)$
implies that 
$\ul{X}\in DOM_{sm}(\cR_i)$
and 
$R_i(\ul{X})=\cR_i(\ul{X})$
for 
$i=1,2$
and
$(R_1R_2)(\ul{X})=\cR_1(\ul{X})
\cR_2(\ul{X}).$
Now, let 
$\ul{X}=(X_1,\ldots,X_d)\in(\KK^{sm\times sm})^d$ 
as in 
(\ref{eq:15Oct18a}), so similar computation shows that 
$\ul{X}\in DOM_{sm}(\cR)$ and
$(R_1R_2)(\ul{X})=\cR_1(\ul{X})\cR_2(\ul{X})=
\cR^{\text{ser}}(\ul{X})$.

\textbf{\uline{5.\,Inverses:}}
Suppose 
$R$ 
is a nc rational expression admitting a realization 
$\cR$ 
centred at 
$\ul{Y}$,
described by the tuple 
$(L,D,C,\ul{\bA},\ul{\bB})$,
also w.r.t
any
unital $\KK-$algebra $\cA$
and 
$R(\ul{Y})=D$ 
is invertible.

Thus,
$\ul{\fa}\in dom^{\cA}(R^{-1})$
implies that
$\ul{\fa}\in dom^{\cA}(R),\,
I_s\otimes\ul{\fa}\in DOM^{\cA}(\cR),
\,R^{\cA}(\ul{\fa})$ 
is invertible and 
$\cR^{\cA}(I_s\otimes\ul{\fa})=I_s\otimes 
R^{\cA}(\ul{\fa})$,
so 
\begin{equation*}
I_s\otimes 
(R^{-1})^{\cA}(\ul{\fa})=
(I_s\otimes 
R^{\cA}(\ul{\fa}))^{-1}=
(\cR^{\cA}(I_s\otimes\ul{\fa}))^{-1}
=(\cR^{\text{inv}})^{\cA}
(I_s\otimes\ul{\fa})
\end{equation*}
and
$I_s\otimes \ul{\fa}\in DOM^{\cA}(\cR^{\text{inv}})$,
when 
$\cR^{\text{inv}}$ 
is the nc Fornasini--Marchesini realization centred at
$\ul{Y}$ described by
\begin{multline}
\label{eq:20Mar19e}
D^{\text{inv}}=D^{-1},\, 
C^{\text{inv}}=D^{-1}C,\,
\bA_k^{\text{inv}}=\bA_k
-\bB_k\cdot(D^{-1}C)
\text{ and }\\
\bB_k^{\text{inv}}
=-\bB_k\cdot D^{-1},\,
k=1,\ldots,d.
\end{multline}

Moreover, if 
$\ul{X}\in dom_{sm}(R^{-1})$,
then
$\ul{X}\in dom_{sm}(R)$ and 
$R(\ul{X})$
is invertible, so 
$\ul{X}\in DOM_{sm}(\cR)$ 
and 
$R(\ul{X})=\cR(\ul{X})$ 
is invertible, therefore the matrices  
$M$
and 
$I_m\otimes 
D+(I_m\otimes C)M^{-1}N$
are invertible,
where 
$$M:=I_{Lm}-\sum_{k=1}^d 
(X_k-I_m\otimes Y_k)\bA_k
\text{ and }
N:=\sum_{k=1}^d 
(X_k-I_m\otimes Y_k)\bB_k.$$
Consider the matrix
\begin{align*}
E:=\begin{bmatrix}
-M&N\\I_m\otimes C&I_m\otimes D\\
\end{bmatrix}\in\KK^{(L+s)m\times (L+s)m}
\end{align*}
together with its two Schur complements decompositions 
\begin{multline*}
E=\begin{bmatrix}
I_{Lm}&0\\
-(I_m\otimes C)M^{-1}&I_{sm}\\
\end{bmatrix}\begin{bmatrix}
-M&0\\0&I_m\otimes D+
(I_m\otimes C)M^{-1}N\\
\end{bmatrix}
\begin{bmatrix}
I_{Lm}&-M^{-1}N\\
0&I_{sm}\\
\end{bmatrix}
\\=
\begin{bmatrix}
I_{Lm}&N(I_m\otimes D)^{-1}\\
0&I_{sm}\\
\end{bmatrix}
\begin{bmatrix}
-M-N(I_m\otimes D^{-1} C)&0\\
0& I_m\otimes D\\
\end{bmatrix}
\begin{bmatrix}
I_{Lm}&0\\
I_m\otimes (D^{-1}C)&I_{sm}\\
\end{bmatrix}.
\end{multline*}
As 
$M$ 
and 
$I_m\otimes D+(I_m\otimes C)M^{-1}N$ 
are invertible, it follows that 
$E$ 
is invertible and hence
\begin{equation*}
M+N(I_m\otimes (D^{-1}C))
=I_{Lm}-\sum_{k=1}^d 
(X_k-I_m\otimes Y_k)\bA_k^{\text{inv}}
\end{equation*}
is invertible, i.e., 
$\ul{X}\in DOM_{sm}
(\cR^{\text{inv}})$. 
Thus
$dom_{sm}(R^{-1})\subseteq 
DOM_{sm}(\cR^{\text{inv}})$ 
and for every 
$\ul{X}\in dom_{sm}(R^{-1})$, 
we have
$R^{-1}(\ul{X})=
\cR(\ul{X})^{-1}=
\cR^{\text{inv}}(\ul{X}).$
\end{proof}
We finish this subsection by comparing the two parts of 
Definition
\ref{def:25Sep18a} 
for the $\KK-$algebra 
$\cA_n=\KK^{\ntn}$
(cf. Remark 
\ref{rem:20Mar19a}). 
This will imply (see Corollary \ref{cor:13Feb19a}) 
that for every nc rational expression $R$, the realization that we have constructed in Theorem
\ref{thm:21Ap17a}--- centred at a 
$d-$tuple of
$\sts$ matrices---
allows us to evaluate $R$ 
at every point in its domain of 
regularity and not only at the points 
whose dimension is a multiple of 
$s$. 
An alternative way to evaluate
a nc rational expression on all of 
its domain of regularity, will be given later in 
Theorem 
\ref{thm:30Jan18c}. 
We define 
$P_1=E(n,s)$ 
and 
$P_2=E(n,L)$, 
correspondingly to 
$(\ref{eq:24Jan19a})$.
\begin{proposition}
\label{prop:18Oct18a}
Let 
$n\ge1,\,\cA_n=\KK^{\ntn},\,
\cR$
be a nc Fornasini--Marchesini 
realization centred at
$\ul{Y}\in(\KK^{\sts})^d$ 
and 
$\ul{X}\in(\KK^{sn\times sn})^d$. 
Then
\begin{align}
\label{eq:13Feb19b}
\ul{X}\in DOM^{\cA_n}(\cR)
\iff 
P_1\cdot\ul{X}\cdot P_1^{-1}\in DOM_{sn}(\cR)
\end{align}
and
$\cR^{\cA_n}(\ul{X})=
P_1^{-1}\cR(P_1\cdot\ul{X}
\cdot P_1^{-1}) P_1$, 
whenever 
$(\ref{eq:13Feb19b})$
holds.
\end{proposition}
\begin{proof}
Let 
$\ul{X}=(X_1,\ldots,X_d)
\in(\KK^{sn\times sn})^d$
and consider the decomposition 
(\ref{eq:15Oct18a}), 
where
$X^{(k)}_{ij}\in
\KK^{\ntn}=\cA_n$ 
for 
$1\le i,j\le s$ 
and 
$1\le k\le d$.
As
\begin{multline*}
I_L\otimes 1_{\cA}
-\sum_{k=1}^d
(X_k-Y_k\otimes 1_{\cA})\bA_k^{\cA_n}
=I_{Ln}-\sum_{k=1}^d
\Big[
\sum_{i,j=1}^{s}\bA_k(E_{ij})
\otimes X^{(k)}_{ij}
-\bA_k(Y_k)\otimes I_n
\Big]
=\\P_{2}^{-1}
\Big(
I_{Ln}-\sum_{k=1}^d
\Big[
\sum_{i,j=1}^{s}
X^{(k)}_{ij}
\otimes\bA_k(E_{ij})
-I_n\otimes \bA_k(Y_k)\Big]
\Big)
P_{2}
=P_{2}^{-1}\Big(
I_n\otimes I_{L}-\sum_{k=1}^d 
\Big[
\sum_{i,j=1}^{s}\\
(X^{(k)}_{ij}
\otimes E_{ij})\bA_k-
(I_n\otimes Y_k)\bA_k\Big]\Big)P_{2}
=P_{2}^{-1}\Big(
I_n\otimes I_{L}-\sum_{k=1}^d 
(P_1X_kP_1^{-1}-
I_n\otimes Y_k)\bA_k
\Big)P_{2},
\end{multline*}
we have
$\ul{X}\in DOM^{\cA_n}(\cR)$
if and only if
$P_1\cdot\ul{X}\cdot 
P_1^{-1}\in DOM_{sn}(\cR)$.

Similar computation shows that
\begin{multline*}
(X_k)\bB_k^{\cA_n}-\bB_k(Y_k)\otimes I_n=
\sum_{i,j=1}^{s}
\bB_k(E_{ij})
\otimes X^{(k)}_{ij}-\bB_k(Y_k)\otimes I_n
=P_{2}^{-1}
\Big( 
\sum_{i,j=1}^{s}
X^{(k)}_{ij}\otimes
\bB_k(E_{ij})
\\-I_n\otimes \bB_k(Y_k)
\Big)
P_1
=P_{2}^{-1}(P_1X_kP_1^{-1}
-I_n\otimes Y_k)\bB_k
P_1
\end{multline*}
and hence
$\ul{X}\in DOM^{\cA_n}(\cR)$ implies 
\begin{multline*}
\cR^{\cA_n}(\ul{X})=
D\otimes I_n
+(C\otimes I_n)
\Big(I_{Ln}-\sum_{k=1}^d
\big[
(X_k)\bA_k^{\cA_n}-\bA_k(Y_k)\otimes I_n
\big]
\Big)^{-1}
\sum_{k=1}^d
\big[
(X_k)\bB_k^{\cA_n}-
\\\bB_k(Y_k)\otimes I_n
\big]
=D\otimes I_n+(C\otimes I_n)P_{2}^{-1}
\Big(I_n\otimes I_{L}-\sum_{k=1}^d 
(P_1X_kP_1^{-1}-
I_n\otimes Y_k)\bA_k
\Big)^{-1}
\sum_{k=1}^d 
(P_1X_kP_1^{-1}\\-
I_n\otimes Y_k)
\bB_k
P_1=P_1^{-1}
\cR(P_1\cdot \ul{X}\cdot P_1^{-1}) P_1.
\end{multline*}
\end{proof}
\begin{cor}
\label{cor:13Feb19a}
If $n\in\NN$ and
$\cR$
is a nc Fornasini--Marchesini 
realization of a nc rational expression
$R$
w.r.t 
$\cA_n=\KK^{\ntn}$, 
then for every
$\ul{X}\in dom_n(R)$ 
we have 
\begin{equation*}
\ul{X}\otimes 
I_s\in DOM_{sn}(\cR) 
\text{ and }  
R(\ul{X})
\otimes I_s=\cR(\ul{X}\otimes I_s).
\end{equation*}
\end{cor}
\begin{proof}
Let
$\ul{X}\in 
dom_n(R)$, 
thus 
$\ul{X}\in dom^{\cA_n}(R)$ 
and that implies by Definition
\ref{def:25Sep18a}
that 
$I_s\otimes \ul{X}\in 
DOM^{\cA_n}(\cR)$ 
and 
$\cR^{\cA_n}(I_s\otimes \ul{X})
=I_s\otimes 
R^{\cA_n}(\ul{X})$. 
Using Proposition
\ref{prop:18Oct18a}, 
we get
\begin{equation*}
\ul{X}\otimes I_s=
P_1\cdot (I_s\otimes \ul{X})
\cdot P_1^{-1}\in DOM_{sn}(\cR)
\end{equation*}
and 
\begin{equation*}
P_1^{-1}\cR(\ul{X}\otimes I_s)P_1=
\cR^{\cA_n}(I_s\otimes \ul{X})=
I_s\otimes 
R(\ul{X})
\end{equation*}
 which implies that 
 $\cR(\ul{X}\otimes I_s)=
R(\ul{X})\otimes I_s$.
\end{proof}
\subsection{Controllability and observability}
\label{subsec:ContObs}
To consider the notion of minimal realizations as 
in the classical realization theory, we first 
introduce the definitions of controllability and 
observability in the case of nc 
Fornasini--Marchesini realizations centred 
at a matrix point, 
which are generalizations of the definitions
in the case of nc Fornasini--Marchesini realizations centred at a scalar point. For a reference
of the later ones see 
\cite{BGM1}.

Given the linear mappings 
$\bA_1,\ldots,\bA_d:\KK^{\sts}
\rightarrow\KK^{L\times L},\,
\bB_k:\KK^{\sts}\rightarrow
\KK^{L\times s}$
and a word
$\omega=g_{i_1}\dots 
g_{i_\ell}\in\cG_d$ 
of length 
$|\omega|=\ell$, 
define the 
multilinear mapping 
$\ul{\bA}^\omega:(\KK^{\sts})^{\ell}
\rightarrow\KK^{L\times L}$
by
\begin{equation*}
\ul{\bA}^\omega
\left(X_1,\ldots,X_{\ell}\right)
:=\bA_{i_1}(X_1)\cdots
\bA_{i_\ell}(X_{\ell})
\end{equation*}
and the multilinear mapping 
$\ul{\bA}^{\omega}\cdot\bB_k:
(\KK^{\sts})^{\ell+1}\rightarrow
\KK^{L\times s}$ 
by
\begin{equation*}
(\ul{\bA}^{\omega}\cdot\bB_k)
(X_1,\ldots,X_{\ell+1}):=
\ul{\bA}^{\omega}(X_1,\ldots,X_{\ell})
\bB_k(X_{\ell+1}).
\end{equation*}
\begin{definition}
\label{def:25Jan19a}
Let
$\bA_1,\ldots,\bA_d:\KK^{\sts}
\rightarrow\KK^{L\times L}$
and
$\bB_1,\ldots,\bB_d:
\KK^{\sts}\rightarrow\KK^{L\times s}$ 
be linear mappings, and 
$C\in\KK^{s\times L}$. 
\begin{itemize}
\item[1.]
The \textbf{controllable subspace}
$\cC_{\ul{\bA},\ul{\bB}}$ 
is defined by 
$$\bigvee_{\omega\in\cG_d,\,
X_1,\ldots,X_{|\omega|+1}
\in\KK^{s\times s},\,
1\le k\le d} 
\Ima
\left(\ul{\bA}^\omega
(X_1,\ldots,
X_{|\omega|})
\bB_k
(X_{|\omega|+1})
\right).$$
If 
$\cC_{\ul{\bA},\ul{\bB}}=\KK^L$,
then the tuple 
$(\ul{\bA},\ul{\bB})$ 
is called \textbf{controllable}.
\item[2.]
The \textbf{un-observable subspace}
$\obs{C}{\ul{\bA}}$ 
is defined by
$$\bigcap_{\omega\in\cG_d,\,
X_1,\ldots,X_{|\omega|}
\in\KK^{s\times s}}
\ker \left(C\ul{\bA}^\omega
\left(X_1,\ldots,
X_{|\omega|}\right)\right).$$
If 
$\obs{C}{\ul{\bA}}=\{\ul{0}\}$,
the tuple
$(C,\ul{\bA})$ 
is called \textbf{observable}.
\end{itemize}
\end{definition}
The multilinear mapping 
$\ul{\bA}^\omega$ 
can be viewed as a linear mapping from
$(\KK^{\sts})^{\ell}$ 
to 
$\KK^{L\times L}$. 
Then one can use the faux product, as introduced in 
(\ref{eq:21Mar19a}), 
to define controllability and observability 
not only on the level of 
$\sts$ 
matrices, but also on the levels of 
$sm\times sm$ 
matrices, for every 
$m\in\NN$,  
using the subspaces  
$$\cC_{\ul{\bA},\ul{\bB}}^{(m)}=
\bigvee_{
\omega\in\cG_d,\,
X_1,\ldots,X_{|\omega|+1}
\in\KK^{sm\times sm},\,1\le k\le d} 
\Ima
\big((X_1\odot_s\cdots
\odot_sX_{|\omega|})
\ul{\bA}^\omega(
X_{|\omega|+1})
\bB_{k}
\big)$$
and
$$\obs{C}{\ul{\bA}}^{(m)}
=\bigcap_{\omega\in\cG_d,\,
X_1,\ldots,X_{|\omega|}\in\KK^{sm\times sm}}
\ker 
\big((I_m\otimes C)
(X_1\odot_s\cdots\odot_sX_{|\omega|})
\ul{\bA}^\omega\big).$$
\begin{proposition}
\label{prop:28Sep18b}
If 
$m\in\NN$, 
then
$\cC_{\ul{\bA},\ul{\bB}}^{(m)}
=\cC_{\ul{\bA},\ul{\bB}}^{(1)}
\otimes \KK^m$
and
$\obs{C}{\ul{\bA}}^{(m)}=
\obs{C}{\ul{\bA}}^{(1)}\otimes
\KK^m.$
\end{proposition}
\begin{proof}
Let 
$m\in\NN$.

$\bullet$ If 
$\ul{u}\in\cC^{(m)}_{\ul{\bA},\ul{\bB}}$,
then 
$\ul{u}$
is a linear combination of vectors of the form
$$(X_1)\bA_{i_1}\cdots(X_k)\bA_{i_k}
(X_{k+1})\bB_{i_{k+1}}\ul{u}_i,$$
where 
$1\le i_1,\ldots,i_{k+1}\le d,\,
X_1,\ldots,X_{k+1}\in
\KK^{sm\times sm}$ 
and 
$\ul{u}_i\in\KK^{sm}$. 
As the mappings
$\bA_i,\bB_i (1\le i\le d)$
act on 
$sm\times sm$ 
matrices by acting on their $\sts$ blocks, 
we get that 
$(X_1)\bA_{i_1}\cdots(X_k)\bA_{i_k}
(X_{k+1})\bB_{i_{k+1}}\ul{u}_i\in
\cC^{(1)}_{\ul{\bA},\ul{\bB}}\otimes\KK^m$
and as a linear combination of such vectors, 
we get that 
$\ul{u}\in\cC^{(1)}_{\ul{\bA},
\ul{\bB}}\otimes\KK^m$.

$\bullet$ 
On the other hand, let 
$\ul{u}\in\cC^{(1)}_{\ul{\bA},
\ul{\bB}}\otimes\KK^m$ 
and write 
$\ul{u}=
\begin{bmatrix}
\ul{u}_1^T&\ldots&\ul{u}_m^T
\end{bmatrix}^T$ 
where 
$\ul{u}_1,\ldots,\ul{u}_m
\in\cC^{(1)}_{\ul{\bA},\ul{\bB}}.$
Thus 
\begin{multline*}
\ul{u}=\sum_{i=1}^m 
\ul{u}_i\otimes \ul{e}_i=
\sum_{i=1}^m
\Big(
\sum_{j=1}^{k_i}
\ul{\bA}^{\omega_{j,i}}\big(
X_{j,1}^{(i)},\ldots,X^{(i)}_{j,|
\omega_{j,i}|}
\big)
\bB_{\ell_{i,j}}
\big(\wt{X}^{(i)}_j
\big)
\ul{w}_{j,i}
\Big)
\otimes \ul{e}_i
=\sum_{i=1}^m
\sum_{j=1}^{k_i} \\
\Big[
I_m\otimes
\Big(
\ul{\bA}^{\omega_{j,i}}
\big(
X_{j,1}^{(i)},\ldots,X^{(i)}_{j,|
\omega_{j,i}|}\big)
\bB_{\ell_{i,j}}
\big(\wt{X}^{(i)}_{j}\big)
\ul{w}_{j,i}
\Big)
\Big] 
\ul{e}_i
=\sum_{i=1}^m\sum_{j=1}^{k_i}
\big( (I_m\otimes X^{(i)}_{j,1})\odot_s
\cdots\\\odot_s 
(I_m\otimes X^{(i)}_{j,|\omega_{j,i}|})
\big)
\ul{\bA}^{\omega_{j,i}}
\big(
I_m\otimes \wt{X}^{(i)}_{j}
\big)
\bB_{\ell_{i,j}}(I_m\otimes 
\ul{w}_{j,i})\ul{e}_i
\in\cC^{(m)}_{\ul{\bA},\ul{\bB}},
\end{multline*}
where 
$\ul{e}_1,\ldots,\ul{e}_m$ 
is the standard basis of 
$\KK^m,\,k_i\in\NN,\, 
\omega_{j,i}\in\cG_d,\, 
1\le\ell_{i,j}\le d$
and
$X_{j,1}^{(i)},\ldots,
X^{(i)}_{j,|\omega_{j,i}|},
\wt{X}^{(i)}_{j}
\in\KK^{\sts}$ 
for 
$1\le j\le k_i$ 
and 
$1\le i\le m $.

$\bullet$
Next, let 
$$\ul{u}=
\begin{bmatrix}
\ul{u}_1^T&\ldots&\ul{u}_m^T\\
\end{bmatrix}^T\in 
\obs{C}{\ul{\bA}}^{(m)},$$ 
then for all 
$\omega\in\cG_d,\,1\le j\le m$ 
and 
$X_1,\ldots,X_{|\omega|}\in\KK^{\sts}$, 
we have
\begin{multline*}
\ul{0}=(I_m\otimes C)\big(
(E_{jj}\otimes X_1)\odot_s\cdots\odot_s 
(E_{jj}
\otimes X_{|\omega|}) 
\big)\ul{\bA}^\omega \ul{u}
\\=(I_m\otimes C)\big(
E_{jj}\otimes 
\ul{\bA}^{\omega}
(X_1,\ldots,X_{|\omega|})
\big)\ul{u}=
C\ul{\bA}^{\omega}
(X_1,\ldots,X_{|\omega|})
\ul{u}_j,
\end{multline*} 
i.e., 
$\ul{u}_j\in \obs{C}{\ul{\bA}}^{(1)}$ 
and hence 
$\ul{u}\in 
\obs{C}{\ul{\bA}}^{(1)}\otimes\KK^m$.

$\bullet$
On the other hand, let 
$$\ul{u}=
\begin{bmatrix}
\ul{u}_1^T&\ldots&\ul{u}_m^T\\
\end{bmatrix}^T\in
\obs{C}{\ul{\bA}}^{(1)}
\otimes\KK^m,$$
then
for every 
$\omega=g_{i_1}\ldots g_{i_k}\in\cG_d$ 
and 
$Z_1,\ldots,Z_{k}\in\KK^{sm\times sm}$
we have
\begin{multline*}
(I_m\otimes C)(
Z_1\odot_s\cdots\odot_sZ_{|\omega|})
\ul{\bA}^{\omega} \ul{u}
=(I_m\otimes C)(Z_1)\bA_{i_1}
\cdots(Z_k)\bA_{i_k} \ul{u}
=\\
(I_m\otimes C)
\Big[
\bA_{i_1}\big(
Z^{(1)}_{p,q}\big)
\Big]_{1\le p,q\le m}\cdots
\Big[\bA_{i_k}\big(Z^{(k)}_{
p,q}\big)
\Big]_{1\le p,q\le m}
\begin{bmatrix}
\ul{u}_1\\
\vdots\\
\ul{u}_m\\
\end{bmatrix}=\ul{0}
\end{multline*}
as 
$\ul{u}_1,\ldots,\ul{u}_m\in
\obs{C}{\ul{\bA}}^{(1)}$ 
and each of the entries in the product is a linear combination 
of vectors of the form 
$C 
\bA_{i_1}\big(Z^{(1)}_{p_1,q_1}\big)
\cdots
\bA_{i_k}\big(Z^{(k)}_{p_k,q_k}\big)\ul{u}_i$ 
where 
$1\le i\le m$, 
which are all 
$\ul{0}$. 
\end{proof}
Two immediate consequences of Proposition
\ref{prop:28Sep18b}  
are the following.
If 
$(\ul{\bA},\ul{\bB})$ 
is  controllable, then 
$\cC_{\ul{\bA},\ul{\bB}}^{(m)}=\KK^{Lm}$
for all $m\in\NN$;
whereas if 
$\cC_{\ul{\bA},\ul{\bB}}^{(m)}=\KK^{Lm}$
for some $m\in\NN$, then 
$(\ul{\bA},\ul{\bB})$ 
is controllable. Similarly,  
if 
$(C,\ul{\bA})$ 
is  observable, then 
$\obs{C}{\ul{\bA}}^{(m)}=\{\ul{0}\}$
for all $m\in\NN$;
whereas if 
$\obs{C}{\ul{\bA}}^{(m)}=\{\ul{0}\}$
for some $m\in\NN$, then 
$(C,\ul{\bA})$ 
is observable 

Next, we show how 
the original definitions of
controllability and observability may be reformulated 
using the standard basis 
$\cE_s$ 
of 
$\KK^{\sts}$.
\begin{proposition}
Let
$\bA_1,\ldots,\bA_d:\KK^{\sts}
\rightarrow\KK^{L\times L}$
and
$\bB_1,\ldots,\bB_d:\KK^{\sts}
\rightarrow\KK^{L\times s}$ 
be linear mappings, and 
$C\in\KK^{s\times L}$. 
Then
\begin{enumerate}
\item[1.]
$(\ul{\bA},\ul{\bB})$ 
is controllable 
if and only if
$$\bigvee_{\omega\in\cG_d,\,
X_1,\ldots,X_{|\omega|+1}
\in\cE_s,\,
1\le k\le d} 
\Ima
\big(\ul{\bA}^\omega
(X_1,\ldots,
X_{|\omega|})
\bB_k
(X_{|\omega|+1})
\big)
=\KK^{L}.$$
\item[2.]
$(C,\ul{\bA})$ 
is  observable
if and only if
$$\bigcap_{\omega\in\cG_d,\,
X_1,\ldots,X_{|\omega|}\in\cE_s}
\ker \big(C\ul{\bA}^\omega
(X_1,\ldots,
X_{|\omega|})\big)=\{\ul{0}\}.$$
\end{enumerate}
\end{proposition}
\begin{proof}
Since 
$\cE_s\subseteq \KK^{\sts}$
the direction 
$\Longleftarrow$ of part \textbf{1}
is trivial. 
To prove the other direction, suppose 
$(\ul{\bA},\ul{\bB})$
is controllable, let
 $X_1,\ldots,X_{\ell+1}\in\KK^{\sts}$ 
 and
 $\omega=g_{j_1}
 \ldots g_{j_{\ell}}\in\cG_d$.
 Thus, one can write
$X_t=\sum_{p,q=1}^s 
E_{pq}\otimes x_{pq}^{(t)} 
\text{ for }1\le t\le \ell+1$ 
and by linearity of 
$\bA_k,\bB_k$ 
we get
\begin{multline*}
\ul{\bA}^\omega(X_1,\ldots,
X_\ell)\bB_{k}(X_{\ell+1})
=\Big(\prod_{t=1}^\ell \sum_{p_t,q_t=1}^s
x^{(t)}_{p_t,q_t}\bA_{j_t}(E_{p_tq_t})
\Big)\sum_{p,q=1}^s
x^{(t+1)}_{p,q}
\bB_k(E_{pq})
\\\in \bigvee_{\nu\in\cG_d,\,
Z_1,\ldots,Z_{\ell+1}
\in\cE_s,\,1\le k\le d} 
\Ima\ 
\big(\ul{\bA}^\nu
(Z_1,\ldots,
Z_\ell)
\bB_{k}(Z_{\ell+1})\big).
\end{multline*}
Therefore
$\cC_{\ul{\bA},\ul{\bB}}\subseteq 
\bigvee_{\nu\in\cG_d,\,
Z_1,\ldots,Z_{\ell+1}
\in\cE_s,\,1\le k\le d} 
\Ima\ 
\big(\ul{\bA}^\nu
(Z_1,\ldots,
Z_\ell)
\bB_{k}(Z_{\ell+1})
\big)$
whereas 
$\cC_{\ul{\bA},\ul{\bB}}=\KK^{L}$ 
implies the wanted equality. 
Similar proof holds for part 
\textbf{2}.
\end{proof}
The last part of this subsection discusses observability and controllability matrices.
The infinite block matrix 
\begin{align}
\label{eq:14Oct17a}
\fC_{\ul{\bA},\ul{\bB}}:=
row\begin{bmatrix}
\fC_{\ul{\bA},\ul{\bB}}^{(\omega,k)}\\
\end{bmatrix}_{(\omega,k)\in
\cG_d\times\{1,\ldots,d\}}
\end{align}
is called the 
\textbf{controllability matrix} 
associated with the tuple 
$(\ul{\bA},\ul{\bB})$,
where 
\begin{align*}
\fC_{\ul{\bA},\ul{\bB}}^{(\omega,k)}
\in\KK^{L\times s^3(|\omega|+1)}
\text{ is given by }
\fC_{\ul{\bA},\ul{\bB}}^{(\omega,k)}
:=row
\begin{bmatrix}
(\ul{\bA}^\omega
\cdot\bB_k)(\ul{Z})
\end{bmatrix}_{\ul{Z}
\in\cE_s^{|\omega|+1}}
\end{align*}
for each 
$(\omega,k)\in
\cG_d\times \{1,\ldots,d\}$
and the infinite block matrix
\begin{align}
\label{eq:19Oct17a}
\fO_{C,\ul{\bA}}:=col
\begin{bmatrix}
\fO_{C,\ul{\bA}}^{(\omega)}
\end{bmatrix}_{\omega\in\cG_d}
\end{align}
is called the 
\textbf{observability matrix} 
associated to the tuple 
$(C,\ul{\bA})$, 
where
\begin{align*}
\fO_{C,\ul{\bA}}^{(\omega)}\in
\KK^{s^{3}|\omega|\times L}
\text{ is given by }
\fO_{C,\ul{\bA}}^{(\omega)}:=col
\begin{bmatrix}
C\cdot\ul{\bA}^\omega(\ul{Z})
\end{bmatrix}_{\ul{Z}\in\cE_s^{|\omega|}}
\end{align*}
 for each
$\omega\in\cG_d$.
The following is a characterization of controllability 
and observability using the
controllability and observability matrices. 
Most of the arguments in
the proof are taken from linear algebra. 
\begin{proposition}
\label{prop:22Oct17d}
Let
$\bA_1,\ldots,\bA_d:\KK^{\sts}
\rightarrow\KK^{L\times L}$
and
$\bB_1,\ldots,\bB_d:\KK^{\sts}
\rightarrow\KK^{L\times s}$ 
be linear mappings, and 
$C\in\KK^{s\times L}$.  
The following are equivalent:
\begin{enumerate}
\item[1.]
$(\ul{\bA},\ul{\bB})$ 
is controllable $[$resp., 
$(C,\ul{\bA})$
is observable$]$.
\item[2.]
The matrix 
$\fC_{\ul{\bA},\ul{\bB}}$ 
$\big[$resp., 
$\fO_{C,\ul{\bA}}\big]$  
is right $[$resp., left$]$ invertible.
\item[3.]
The finite block matrix
$row\begin{bmatrix}
\fC_{\ul{\bA},\ul{\bB}}^{(\omega,k)}\\
\end{bmatrix}_{|\omega|\le\ell,\,1\le k\le d}$ 
$\Big[$resp., 
$col\begin{bmatrix}
\fO_{C,\ul{\bA}}^{(\omega)}\\
\end{bmatrix}_{|\omega|\le\ell}\Big]$
is right $[$resp., left$]$ 
invertible for some
$\ell\in\NN$.
\end{enumerate}
In that case, we can choose 
$\ell\le L-1$.
\end{proposition}
\begin{proof}
\uline{\textbf{1 $\Longrightarrow$ 2:}}
If 
$(\ul{\bA},\ul{\bB})$ 
is controllable, then
$\ul{e}_j\in\KK^L$ 
can be written as
$$\ul{e}_j=\sum_{i=1}^{k_j}
\ul{\bA}^{\omega_{j,i}}
\big(X^{(j)}_{i,1},\ldots,
X^{(j)}_{i,|\omega_{j,i}|}\big)
\bB_{\ell_{j,i}}
\big(X^{(j)}_{i,|\omega_{j,i}|+1}\big)
\ul{u}_{j,i},$$
where 
$k_j\in\NN,\,\omega_{j,i}\in\cG_d,\,
X_{i,1}^{(j)}
,\ldots,X^{(j)}_{i,|\omega_{j,i}|+1}
\in\KK^{\sts}$
and
$1\le\ell_{j,i}\le d$, 
for every 
$1\le i\le k_j$ 
and 
$1\le j\le L$.
Thus 
$\ul{e}_1,\ldots,\ul{e}_L$ 
belong to the column span of the matrix 
$\fC_{\ul{\bA},\ul{\bB}}$ 
and hence 
$\fC_{\ul{\bA},\ul{\bB}}$ 
is right invertible. 

\uline{\textbf{2 $\Longrightarrow$ 3:}}
If the infinite matrix 
$\fC_{\ul{\bA},\ul{\bB}}$
is right invertible, it means that its column span is equal to
$\KK^L$, 
however this span of infinitely many vectors of length 
$L$ 
must coincide with a span of finitely many of the columns, 
which easily implies part 
\textbf{3}.

\uline{\textbf{3 $\Longrightarrow$ 1:}}
If 
\textbf{3}
holds then the column span of 
$\fC_{\ul{\bA},\ul{\bB}}$
contains 
$\ul{e}_1,\ldots,\ul{e}_L$ 
and thus is equal to 
$\KK^L$, i.e., 
$\cC_{\ul{\bA},\ul{\bB}}=\KK^L$
and the tuple
$(\ul{\bA},\ul{\bB})$ 
is controllable.

$\bullet$ Suppose next that 
$(\ul{\bA},\ul{\bB})$
is controllable, and define
$$\cC_{\ell}:=\bigvee_{|\omega|\le\ell,\,
1\le k\le d}
\Ima\big(
\fC_{\ul{\bA},\ul{\bB}}^{(\omega,k)}
\big)=
\bigvee_{|\omega|\le\ell,\,
1\le k\le d}
\Ima
(\ul{\bA}^{\omega}\cdot\bB_k
)$$
for 
$\ell\ge0$. 
Then 
$\cC_0\subseteq \cC_1\subseteq
\ldots \subseteq \cC_{\ell}
\subseteq \cC_{\ell+1}\subseteq\ldots$
are all subspaces of 
$\KK^L$,
whereas 
 the controllability of 
$(\ul{\bA},\ul{\bB})$ 
implies that 
$\bigcup_{\ell=0}^\infty \cC_{\ell}
=\KK^L$ 
and 
$\cC_0\ne\{0\}$.
Moreover, it is easily seen that if 
$\cC_{\ell_0}=\cC_{\ell_0+1}$ 
for some 
$\ell_0\ge0$, 
then 
$\cC_{\ell_0+k}=\cC_{\ell_0}$ 
for all 
$k\ge0$ 
and hence
$$\KK^L=\bigcup_{\ell=0}^\infty \cC_{\ell}
=\cC_0\cup\ldots\cup \cC_{\ell_0}
=\cC_{\ell_0}.$$
Therefore, the sequence
$1\le\dim(\cC_0)\le \dim(\cC_1)\le
\ldots\le\dim(\cC_L)\le\ldots \le L$
must coincide after at most 
$L-1$ 
inequalities, i.e.,  
$\dim(\cC_{L-1})=L$ 
which means that 
$\KK^{L}=\cC_{\ul{\bA},\ul{\bB}}
=\cC_{L-1}$.

$\bullet$ As for observability, to prove that \textbf{1 $\Longrightarrow$ 2 $\Longrightarrow$ 3
$\Longrightarrow$ 1} 
we use the same arguments as above; so we only show how to get the bound on the size of the matrix.
Suppose that 
$(C,\ul{\bA})$ 
is observable and define 
$$\cN\cO_{\ell}:=
\bigcap_{|\omega|\le\ell,\,
\ul{X}\in(\KK^{\sts}
)^{|\omega|}}\ker 
\big(C\ul{\bA}^{\omega}
(\ul{X})\big)$$
for 
$\ell\ge0$. Then 
$\cN\cO_0\supseteq
\cN\cO_1\supseteq\ldots
\supseteq\cN\cO_\ell\supseteq
\cN\cO_{\ell+1}\supseteq\ldots$
are all subspaces of 
$\KK^L$,
whereas the observability of
$(C,\ul{\bA})$
implies that 
$\bigcap_{\ell\ge0}
\cN\cO_{\ell}=\{\ul{0}\}$
and
$\cN\cO_0\ne\KK^L$.
It is easily seen that if 
$\cN\cO_{\ell_0}=
\cN\cO_{\ell_0+1}$ 
for some 
$\ell_0\ge0$, 
then
$\cN\cO_{\ell_0+k}=
\cN\cO_{\ell_0}$
for all $k\ge0$ and hence
$$\{\ul{0}\}=\bigcap_{\ell\ge0}
\cN\cO_{\ell}=
\cN\cO_{\ell_0}.$$
Therefore the sequence
$L >\dim(\cN\cO_0)\ge
\dim(\cN\cO_1)\ge\ldots\ge
\dim(\cN\cO_L)\ge\ldots$
must coincide after at most
$L-1$
inequalities, i.e., 
$\cN\cO_{L-1}=\{\ul{0}\}$.
\end{proof}

\subsection{Minimal realizations}
\label{subsec1}
A nc Fornasini--Marchesini realization of the form
(\ref{eq:18Aug18a}) 
is said to be 
\begin{enumerate}
\item[$\bullet$]
\textbf{controllable}  
if the tuple 
$(\ul{\bA},\ul{\bB})$ 
is controllable; 
\item[$\bullet$]
\textbf{observable} if the tuple 
$(C,\ul{\bA})$ 
is observable. 
\end{enumerate}
If 
$\cR$
is a nc Fornasini--Marchesini realization of a nc rational expression 
$R$ 
centred at
$\ul{Y}$, 
then it is said to be
\begin{enumerate}
\item[$\bullet$]
\textbf{minimal} 
if the dimension 
$L$ 
is the smallest integer for which 
$R$ 
admits such 
a realization, i.e., if 
$\cR^\prime$
is a nc Fornasini--Marchesini realization of 
$R$ 
centred at 
$\ul{Y}$ 
of dimension 
$L^\prime$, 
then 
$L\le L^\prime$.
\end{enumerate}
\begin{remark}
\label{rem:14Apr19a}
In fact, the minimality of a realization 
$\cR$ 
is w.r.t rational functions, 
meaning that if 
$\cR$
is a minimal nc Fornasini--Marchesini realization of 
$R$ 
centred at 
$\ul{Y}$, 
then it is also a minimal nc Fornasini--Marchesini realization 
of any nc rational expression 
$\wt{R}$ 
which is
$(\KK^d)_{nc}-$evaluation equivalent to 
$R$ (cf. Lemma
\ref{lem:30Jan18b}).
\end{remark}

We proceed by showing that every two controllable and observable
nc Fornasini--Marchesini realizations centred at
$\ul{Y}$
of 
$(\KK^d)_{nc}-$evaluation equivalent nc rational 
expressions must be similar, where most of the ideas of the proof are taken  from 
\cite{BC,BGK}. 
We will also use the following facts, see
\cite[Theorem 4.8]{KV4} and  
\cite{KV1}: 

 $\bullet$
If $R$
is a nc rational expression
and 
$\ul{Y}\in dom_s(R)$, 
then 
\begin{align*}
Nilp(\ul{Y};sm)
:=\left\{
\ul{X}\in(\KK^{sm\times sm})^d: 
\ul{X}-I_m\otimes \ul{Y}
\text{ is jointly nilpotent}
\right\}
\subseteq dom_{sm}(R)
\end{align*}
for every 
$m\in\NN$,
where a tuple
$\ul{Z}=(Z_1,\ldots,Z_d)
\in(\KK^{sm\times sm})^d$ 
is called jointly nilpotent if there exists 
$\kappa\in\NN$ 
such that 
$\ul{Z}^{\odot_s \omega}=0$ 
for all
$\omega\in\cG_d$ 
satisfying 
$|\omega|\ge\kappa$.
 
 $\bullet$ 
$R\mid _{Nilp(\ul{Y})}$
is a nc function on the nilpotent ball around 
$\ul{Y}$, 
that is
\begin{align*}
Nilp(\ul{Y}):=
\coprod_{m=1}^\infty Nilp(\ul{Y};sm).
\end{align*}

$\bullet$ 
Every nc function on 
$Nilp(\ul{Y})$ 
has a power series expansion around 
$\ul{Y}$ 
of the form
\begin{align*}
\sum_{\omega\in\cG_d}
(\ul{X}-I_m\otimes \ul{Y}
)^{\odot_s\omega}\cR_{\omega},\, 
\ul{X}\in 
(\KK^{sm\times sm})^d
\end{align*}
where
$\cR_{\omega}$ 
are
$|\omega|-$linear mappings from 
$(\KK^{\sts})^{|\omega|}$ 
to
$\KK^{\sts}$, called the Taylor--Taylor coefficients and
are uniquely determined, see 
\cite[Theorem 5.9]{KV1}; 
notice that the sum is actually finite.
\begin{lemma}
\label{lem:8Nov18a}
If 
$R$ 
is a nc rational expression in
$x_1,\ldots,x_d$
over
$\KK$ 
and 
$\cR$ 
is a nc Fornasini--Marchesini realization of
$R$ 
centred at
$\ul{Y}\in dom_s(R)$, 
of the form 
$(\ref{eq:18Aug18a})$,
then the Taylor--Taylor coefficients of 
$R$ 
are given by 
$\cR_{\emptyset}:=D$
and the multilinear mappings
$\cR_{\omega g_k}:
=C\cdot\ul{\bA}^\omega\cdot
\bB_k:\left(\KK^{s\times s}\right)^{
\ell+1}\rightarrow 
\KK^{s\times s}$
which act as
\begin{align}
\label{eq:25Jan19b}
\cR_{\omega g_k}
(Z_1,\ldots,Z_{\ell+1})
=C\bA_{i_1}(Z_1)
\cdots
\bA_{i_\ell}(Z_{\ell})\bB_k(Z_{\ell+1})
\end{align}
for $Z_1,\ldots,Z_{\ell+1}
\in\KK^{s\times s},\,
\omega=g_{i_1}\ldots g_{i_\ell}\in\cG_d$ 
and 
$1\le k\le d$. 
Moreover, if 
$Z_1,\ldots,Z_{\ell+1}\in
\KK^{sm\times sm}$, 
then
\begin{align*}
(Z_1\odot_s \cdots\odot_s 
Z_{\ell+1})\cR_{\omega g_k}
=(I_m\otimes C)(Z_1)\bA_{i_1}
\cdots
(Z_{\ell})\bA_{\ell}(Z_{\ell+1})\bB_k.
\end{align*}
\end{lemma}
\begin{proof}
Let
$\ul{X}\in Nilp(\ul{Y};sm)$, 
then
we use the Neumann series w.r.t 
nilpotent elements in
$\big(\mathbf{T}(\KK^{\sts})
\big)^{\mtm}$--- 
where 
$\mathbf{T}(\KK^{\sts})$ 
is the tensor algebra of 
$\KK^{\sts}$--- 
to obtain that
\begin{align*}
\Big(
I_{Lm}-\sum_{k=1}^d
(X_k-I_m\otimes Y_k)\bA_k
\Big)
\sum_{\omega\in\cG_d}
(\ul{X}-I_m\otimes \ul{Y})
^{\odot_s\omega}\ul{\bA}^{\omega}
=I_{Lm},
\end{align*}
where the second multiplicative term 
is actually a finite sum. As 
$\ul{Y}\in dom_s(R)$ 
it follows that 
$Nilp(\ul{Y};sm)\subseteq dom_{sm}(R)$, 
thus 
$\ul{X}\in dom_{sm}(R)
\subseteq DOM_{sm}(\cR)$ 
and
\begin{multline*}
R(\ul{X})=\cR(\ul{X})=(I_m\otimes D)
+(I_m\otimes C)
\Big(
\sum_{\omega\in
\cG_d}(\ul{X}-I_m\otimes \ul{Y}
)^{\odot_s\omega}
\ul{\bA}^\omega
\Big)
\sum_{k=1}^d 
(X_k-I_m\otimes Y_k)
\bB_k
\\=I_m\otimes D+\sum_{\omega
\in\cG_d,\,1\le k\le d}
(\ul{X}-I_m\otimes \ul{Y})^{\odot_s
(\omega g_k)}\cR_{\omega g_k}
=\sum_{\nu\in\cG_d}(\ul{X}-I_m\otimes 
\ul{Y})^{\odot_s\nu}\cR_{\nu},,
\end{multline*}
where the multilinear mappings
$(\cR_{\omega g_k})$ 
are given by 
(\ref{eq:25Jan19b}). 
However,
$R\mid _{Nilp(\ul{Y})}$ 
is a nc function on
$Nilp(\ul{Y})$ 
and so it has a unique Taylor--Taylor 
expansion, given by the coefficients 
$(\cR_\nu)_{\nu\in\cG_d}$.
\end{proof}
\begin{theorem}[Similarity of minimal realizations]
\label{thm:2May16b}
Let 
$R_1$ 
and 
$R_2$
be two nc rational expressions in 
$x_1,\ldots,x_d$ 
over 
$\KK$, 
which admit nc Fornasini--Marchesini realizations
$$\cR_1(\ul{X})=D^1+C^1
\Big(
I_{L_1}-\sum_{k=1}^d
\bA_k^1(X_k-Y_k)
\Big)^{-1}
\sum_{k=1}^d
\bB_k^1(X_k-Y_k)$$
 and 
$$\cR_2(\ul{X})=D^2+C^2
\Big(
I_{L_2}-\sum_{k=1}^d
\bA_k^2(X_k-Y_k)\Big)^{-1}
\sum_{k=1}^d
\bB_k^2(X_k-Y_k),$$ 
respectively, both centred at 
$\ul{Y}\in(\KK^{\sts})^d$.  
Assume  both 
$\cR_1$ and
$\cR_2$ are
controllable and observable. 

If 
$R_1$ 
and 
$R_2$ 
are 
$(\KK^d)_{nc}-$evaluation equivalent, then 
$\cR_1$ 
and
$\cR_2$
are \textbf{uniquely similar}, i.e.,
$L_1=L_2,\,D^1=D^2$
and there exists a unique invertible matrix 
$T\in\KK^{L_1\times L_1}$ 
such that 
\begin{align}
\label{eq:19Oct17c}
C^2=C^1T^{-1},\, 
\bB^2_k=T\cdot\bB^1_k 
\text{ and }\bA_k^2=
T\cdot\bA_k^1\cdot T^{-1},\,
1\le k\le d.
\end{align}
Moreover,
$$DOM_{sm}(\cR_1)=DOM_{sm}(\cR_2)
\text{ and } 
\cR_1(\ul{X})=\cR_2(\ul{X}),\,
\forall \ul{X}\in 
DOM_{sm}(\cR_1)$$
for every
$m\in\NN$,
and for any unital 
$\KK-$algebra $\cA$:
$$DOM^{\cA}(\cR_1)=DOM^{\cA}(\cR_2)
\text{ and }\cR_1^{\cA}(\ul{\fA})
=\cR_2^{\cA}(\ul{\fA}),\,
\forall \ul{\fA}\in DOM^{\cA}(\cR_1).$$
\end{theorem}
\begin{proof}
From Lemma
\ref{lem:8Nov18a}, the Taylor--Taylor coefficients of 
the nc rational expressions 
$R_1$ 
and
$R_2$ (w.r.t the centre
$\ul{Y}$) are  
\begin{align*}
\cR_{\omega g_k}^{(1)}
=C^1\cdot\left(\ul{\bA}^1\right)^\omega\cdot
\bB^1_k
\text{ and }
\cR_{\omega g_k}^{(2)}
=C^2\cdot\left(\ul{\bA}^2\right)^\omega\cdot
\bB^2_k,
\end{align*}
respectively.
Since 
$R_1$ 
and 
$R_2$ 
are 
$(\KK^d)_{nc}-$evaluation equivalent, their restrictions to
$Nilp(\ul{Y})$ 
produce the same nc function and
therefore, by the uniqueness of the Taylor--Taylor coefficients, 
$\cR^{(1)}_{\emptyset}=\cR^{(2)}_{\emptyset}$
and
$\cR^{(1)}_{\omega g_k}=\cR^{(2)}_{\omega g_k}$ 
as multilinear mappings
for every 
$\omega\in\cG_d$ 
and 
$1\le k\le d$, 
i.e., 
$D^1=D^2$ 
and
\begin{align}
\label{eq:19Aug18a}
C^1\cdot
\left(\ul{\bA}^1\right)^\omega\cdot
\bB_k^1=C^2\cdot\left(\ul{\bA}^2\right)^\omega\cdot
\bB_k^2.
\end{align}

Define a mapping 
$\cT$ 
in the following way: for every 
$\omega\in
\cG_d,\,1\le k\le d,\,
\ul{X}=(X_1,\ldots,X_{|\omega|})
\in(\KK^{\sts})^d,\,X_{|\omega|+1}
\in\KK^{\sts}$
and 
$\ul{u}\in\KK^s$, 
let
\begin{align}
\label{eq:17Ap16a}
\cT\big((\ul{\bA}^1)^\omega
(\ul{X})
\bB^1_k
(X_{|\omega|+1})\ul{u}\big):=
(\ul{\bA}^2)^\omega(
\ul{X})
\bB^2_k
(X_{|\omega|+1})\ul{u}
\end{align}
and extend it by linearity.
We proceed by showing some  properties of 
$\cT$.

$\bullet$\textbf{ \uline{The domain of 
$\cT$ 
is 
$\KK^{L_1}$:}}
The domain of 
$\cT$
consists of all the vectors in 
$\KK^{L_1}$ 
which are in
$$\bigvee_{\omega\in\cG_d,\,\ul{X}\in(\KK^{\sts})^{|\omega|},\,
X_{|\omega|+1}\in\KK^{\sts},\,1\le k\le d,\,\ul{u}\in\KK^s}
\big(\ul{\bA}^1\big)^{\omega}
\big(\ul{X}\big)
\bB^1_{k}
\big(X_{|\omega|+1}\big)\ul{u}$$
and that is exactly 
$\cC_{\ul{\bA},\ul{\bB}}=\KK^{L_1}$, 
by the controllability of 
$\cR_1$ 
and hence of 
$(\ul{\bA}^1,\ul{\bB}^1)$.

$\bullet$\textbf{ \uline{$\cT$ is well-defined:}}
Let 
$\ul{w}_1,\ul{w}_2\in\KK^{L_1}$, 
then they can be written as
$$\ul{w}_1=\sum_{j=1}^{p_1}
\big(\ul{\bA}^1\big)^{\omega_{1,j}}
\big(X^{(1)}_{j,1}\ldots,
X^{(1)}_{j,|\omega_{1,j}|}\big)
\bB^1_{k_{1,j}}
\big(X^{(1)}_{j,|\omega_{1,j}|+1}\big)
\ul{u}_{1,j}$$
and
$$\ul{w}_2=\sum_{i=1}^{p_2}
\big(\ul{\bA}^1\big)^{\omega_{2,i}}
\big(X^{(2)}_{i,1},\ldots,
X^{(2)}_{i,|\omega_{2,i}|}
\big)
\bB^1_{k_{2,i}}
\big(X^{(2)}_{i,|\omega_{2,i}|+1}
\big)\ul{u}_{2,i},$$
where
$p_1,p_2\in\NN,\,
\omega_{1,j},\omega_{2,i}\in\cG_d,\, 
1\le k_{1,j},k_{2,i}\le d,\,
\ul{u}_{1,j},
\ul{u}_{2,i}\in\KK^s$ 
and 
$X^{(1)}_{j,\alpha},X^{(2)}_{i,\beta}
\in\KK^{\sts}$, 
for every 
$1\le \alpha\le |\omega_{1,j}|,
\,1\le \beta\le |\omega_{2,i}|,\,
1\le j\le p_1$ and
$1\le i\le p_2$.
Apply  
(\ref{eq:19Aug18a}), 
so for every 
$\omega\in\cG_d$, 
$$C^1\cdot\big(\ul{\bA}^1\big)^{\omega}
\cdot \big(
\ul{\bA}^1\big)^{\omega_{1,j}}
\cdot\bB_{k_{1,j}}^1  
=C^2\cdot\big(
\ul{\bA}^2\big)^{\omega}
\cdot \big(\ul{\bA}^2
\big)^{\omega_{1,j}}
\cdot\bB_{k_{1,j}}^2$$ 
and 
$$C^1\cdot\big(\ul{\bA}^1\big)^{\omega}
\cdot 
\big(\ul{\bA}^1\big)^{\omega_{2,i}}
\cdot\bB_{k_{2,i}}^1  
=C^2\cdot
\big(\ul{\bA}^2\big)
^{\omega}
\cdot \big(\ul{\bA}^2\big)^{\omega_{2,i}}
\cdot\bB_{k_{2,i}}^2,$$
which imply that for every 
$\ul{X}\in
(\KK^{\sts})^{
|\omega|}$, 
\begin{multline*} 
C^2\big(\ul{\bA}^2\big)^\omega(\ul{X})
\Big[\sum_{j=1}^{p_1}
\big(\ul{\bA}^2\big)^{\omega_{1,j}}
\big(X_{j,1}^{(1)},\ldots,
X_{j,|\omega_{1,j}|}^{(1)}
\big)
\bB^2_{k_{1,j}}
\big(X^{(1)}_{j,|\omega_{1,j}|+1}
\big)
\ul{u}_{1,j}\Big]
-C^2\big(\ul{\bA}^2\big)^\omega(\ul{X})
\\\Big[\sum_{i=1}^{p_2}
\left(\ul{\bA}^2\right)^{\omega_{2,i}}
\big(X_{i,1}^{(2)},\ldots,
X_{i,|\omega_{2,i}|}^{(2)}\big)
\bB^2_{k_{2,i}}
\big(X^{(2)}_{i,|\omega_{2,i}|+1}\big)
\ul{u}_{2,i}\Big]
=
C^1\big(\ul{\bA}^1\big)^\omega(\ul{X})
\Big[\sum_{j=1}^{p_1}
\big(\ul{\bA}^1\big)^{\omega_{1,j}}
\\\big(X_{j,1}^{(1)},\ldots,
X_{j,|\omega_{1,j}|}^{(1)}
\big)
\bB^1_{k_{1,j}}
\big(X^{(1)}_{j,|\omega_{1,j}|+1}
\big)
\ul{u}_{1,j}\Big]
-C^1\left(\ul{\bA}^1\right)^\omega(\ul{X})
\Big[\sum_{i=1}^{p_2}
\left(\ul{\bA}^1\right)^{\omega_{2,i}}
\big(X_{i,1}^{(2)},\ldots,\\
X_{i,|\omega_{2,i}|}^{(2)}\big)
\bB^1_{k_{2,i}}
\big(X^{(2)}_{i,|\omega_{2,i}|+1}\big)
\ul{u}_{2,i}\Big],
\end{multline*}
i.e., 
\begin{align}
\label{eq:19Aug18b}
C^2\big(\ul{\bA}^2
\big)^{\omega}(\ul{X})
\big(\cT(\ul{w}_1)-\cT(\ul{w}_2)\big)=
C^1\big(\ul{\bA}^1
\big)^{\omega}(\ul{X})(\ul{w}_1-\ul{w}_2).
\end{align}
Finally, if 
$\ul{w}_1=\ul{w}_2$, 
then 
$C^2\big(\ul{\bA}^2\big)^{\omega}(\ul{X})
(\cT(\ul{w}_1)-\cT(\ul{w}_2))=\ul{0}$
for all 
$\omega\in\cG_d$
and
$\ul{X}\in(\KK^{\sts})^{|\omega|}$, 
whereas
the observability of $\cR_2$
and hence of 
$(C^2,\ul{\bA}^2)$ 
guarantees that 
$\cT(\ul{w}_1)=\cT(\ul{w}_2)$.

$\bullet$\textbf{ \uline{$\cT$ is $1-1$:}}
If 
$\ul{w}_1,\ul{w}_2\in\KK^{L_1}$ 
such that 
$\cT(\ul{w}_1)=\cT(\ul{w}_2)$, 
it follows from 
(\ref{eq:19Aug18b}) 
that 
$C^1\big(\ul{\bA}^1\big)^{\omega}
(\ul{X})(\ul{w}_1-\ul{w}_2)=\ul{0}$
for every 
$\omega\in\cG_d$
and
$\ul{X}\in(\KK^{\sts})^{|\omega|}$, 
whereas the observability of 
$\cR_1$ 
and hence of 
$(C^1,\ul{\bA}^1)$ 
guarantees that 
$\ul{w}_1=\ul{w}_2$.

$\bullet$\textbf{ \uline{$\cT$ is onto $\KK^{L_2}$:}}
The tuple $(\ul{\bA}^2,\ul{\bB}^2)$
is controllable and hence every
$\ul{e}\in\KK^{L_2}$ 
can be written as
\begin{align*}
\ul{e}=\sum_{j=1}^p
\left(\ul{\bA}^2\right)^{\nu_j}
\big(\ul{X}^{(j)}
\big)
\bB^2_{k_j}
\big(
X_{j+1}
\big)
\ul{v}_j
=\cT\Big(\sum_{j=1}^p
\left(\ul{\bA}^1\right)^{\nu_j}
\big(
\ul{X}^{(j)}
\big)
\bB^1_{k_j}\big(
X_{j+1}\big)
\ul{v}_j\Big),
\end{align*}
where
$p\in\NN,\,
\nu_j\in\cG_d,\,
\ul{X}^{(j)}\in 
(\KK^{\sts})^{|\nu_j|},\,
X_{j+1}\in\KK^{\sts},\,
1\le k_j\le d$ 
and
$\ul{v}_j\in\KK^s$ 
for every 
$1\le j\le p$, 
i.e., 
$\KK^{L_2}\subseteq 
\Ima(\cT)$ and hence $\cT$ 
is onto
$\KK^{L_2}$.

Therefore, 
$\cT:\KK^{L_1}\rightarrow\KK^{L_2}$
is an isomorphism, 
$L_1=L_2:=L$,
the representative matrix 
$T:=[\cT]_{\cE_L}\in\KK^{L\times L}$ 
is invertible and 
$\cT(\ul{w})=T\ul{w}$ 
for all 
$\ul{w}\in\KK^L$. 

$\bullet\textbf{ \uline{The realizations 
$\cR_1$ and $\cR_2$ are similar:}}$
If 
$\ul{w}_1=\ul{0}$ 
and 
$\omega=\emptyset$, 
then 
(\ref{eq:19Aug18b}) 
implies
$$C^2T\ul{w}_2
=C^1\ul{w}_2,\,
\forall \ul{w}_2\in\KK^L
\Longrightarrow C^1=C^2T,$$ 
while applying 
(\ref{eq:19Aug18b}) 
again with 
$\ul{w}_1=0$ 
and using the observability of 
$(C^2,\ul{\bA}^2)$, lead to 
$$\bA_k^1(X)=T^{-1}\bA_k^2(X) T,
\,\forall X\in\KK^{\sts},1\le k\le d.$$ 
From the definition of 
$\cT$ 
we have 
$$T\bB_k^1(X)\ul{u}
=\cT\big(\bB_k^1(X)\ul{u}\big)
=\bB_k^2(X) \ul{u},\,
\forall \ul{u}\in\KK^L
\Longrightarrow \bB_k^1(X)=T^{-1}\bB_k^2(X)$$
for every $X\in\KK^{\sts}$ 
and 
$1\le k\le d$. Thus, we proved the realizations are similar.

$\bullet$ 
\textbf{\uline{$T$ is uniquely determined:}}
Let 
$T_2\in\KK^{L\times L}$ 
be such that
$\bB_k^1=T_2^{-1}\cdot\bB_k^2$
and
$\bA_k^1=T_2^{-1}\cdot\bA_k^2\cdot T_2 $ 
for every 
$1\le k\le d$, 
then it is easily seen that 
$T_2$ 
satisfies the relation in
(\ref{eq:17Ap16a})  
and thus the controllability of
$(\ul{\bA}^2,\ul{\bB}^2)$ 
implies that 
$T_2=T$.

$\bullet$ Moreover, 
\begin{multline*}
I_{Lm}-\sum_{k=1}^d
(X_k-I_m\otimes Y_k)
\bA_k^2=I_{Lm}-\sum_{k=1}^d
\big[(I_m\otimes T)
(X_k-I_m\otimes Y_k)
\bA_k^1(I_m\otimes T^{-1})\big]
\\=(I_m\otimes T)
\Big(
I_{Lm}-\sum_{k=1}^d
(X_k-I_m\otimes Y_k)
\bA_k^1 \Big)
(I_m\otimes T)^{-1}
\end{multline*}
for every
$m\in\NN$,
which implies that
$\ul{X}\in DOM_{sm}
(\cR_2)\iff \ul{X}\in DOM_{sm}(\cR_1),$
i.e., that
$DOM_{sm}(\cR_1)=
DOM_{sm}(\cR_2)$.
It is easily seen that using the relations in 
(\ref{eq:19Oct17c}), 
we have
$\cR_1(\ul{X})=\cR_2(\ul{X})$
for every 
$\ul{X}\in DOM_{sm}(\cR_1)$. 

$\bullet$ Finally, the relations in 
(\ref{eq:19Oct17c}) also imply that 
$C^2\otimes 1_{\cA}=(C^1\otimes 1_{\cA})(T^{-1}
\otimes 1_{\cA})$
and that for every
$\ul{\fA}=(\fA_1,\ldots,\fA_d)
\in(\cA^{\sts})^d$, we have
\begin{align*}
(\fA_k)(\bB_k^2)^{\cA}
=(T\otimes 1_{\cA}) 
(\fA_k)(\bB_k^1)^{\cA}
\text{ and }
(\fA_k)(\bA_k^2)^{\cA}=
(T\otimes 1_{\cA})(\fA_k)(\bA_k^1)^{\cA}
(T^{-1}\otimes 1_{\cA})
\end{align*}
for every $1\le k\le d$.
Therefore,
\begin{multline*} 
I_L\otimes 1_{\cA}-\sum_{k=1}^d 
(\fA_k-Y_k\otimes 1_{\cA})(\bA_k^2)^{\cA}=
I_L\otimes 1_{\cA}-\sum_{k=1}^d 
\big[(T\otimes 1_{\cA})(\fA_k
-Y_k\otimes 1_{\cA})(\bA_k^1)^{\cA}
\\(T^{-1}\otimes 1_{\cA})\big]
=(T\otimes 1_{\cA})\Big(
I_L\otimes 1_{\cA}-\sum_{k=1}^d 
(\fA_k-Y_k\otimes 1_{\cA})(\bA_k^1)^{\cA}
\Big)(T\otimes 1_{\cA})^{-1},
\end{multline*}
which implies that
$DOM^{\cA}(\cR_1)=DOM^{\cA}(\cR_2)$
and that
$\cR_1^{\cA}(\ul{\fA})=\cR_2^{\cA}(\ul{\fA})$ 
for every
$\ul{\fA}\in DOM^{\cA}(\cR_1)$.
\end{proof} 
\begin{remark}
\label{rem:14Apr19b}
One can obtain the equality in
(\ref{eq:19Aug18a})
and hence prove Theorem
\ref{thm:2May16b} 
without using Lemma 
\ref{lem:8Nov18a},
by evaluating nc rational expressions on generic matrices and then considering power series in commuting variables (the entries of the generic matrices).
\end{remark}

\subsection{Kalman decomposition}
\label{sec:Kal}
We proceed next to obtain a Kalman decomposition for nc Fornasini--Marchesini realizations centred at a matrix point 
$\ul{Y}\in(\KK^{\sts})^d$, 
which generalizes the Kalman decomposition for nc Fornasini--Marchesini realizations centred at a scalar point
(as in
\cite{BGM1}), 
where the later decomposition is a generalization of the classical Kalman decomposition (see
\cite{BGK,KAF}). 

This is the first place in our analysis where 
$\cA$ 
is no longer an arbitrary
unital 
$\KK-$algebra, but has
to be stably finite.
The stably finiteness is used to deduce the invertibility of one of the blocks in a block upper triangular matrix that is invertible (cf. Lemma 
\ref{lem:19Oct18a}).
\begin{theorem}[Kalman Decomposition]
\label{thm:2May16a}
Let 
\begin{align*}
\cR(\ul{X})= D+C
\Big(I_{L}-\sum_{k=1}^d 
\bA_k(X_k-Y_k)\Big)^{-1}
\sum_{k=1}^d \bB_k(X_k- Y_k)
\end{align*}
be a nc Fornasini--Marchesini realization centred at 
$\ul{Y}\in(\KK^{\sts})^d$. 
There exists a nc Fornasini--Marchesini realization 
$\wt{\cR}$ 
centred at 
$\ul{Y}$, 
that is controllable and observable, of dimension 
$\wt{L}=\dim(\cC_{\ul{\bA},\ul{\bB}})
-\dim(\cC_{\ul{\bA},
\ul{\bB}}\cap \obs{C}{\ul{\bA}}),$
such 
that
\begin{align*}
DOM_{sm}(\cR)\subseteq 
DOM_{sm}(\wt{\cR})
\text{ and }\cR(\ul{X})=\wt{\cR}(\ul{X}),\,
\forall\ul{X}\in DOM_{sm}(\cR)
\end{align*}
for every
$m\in\NN$,
and for any unital
\textbf{stably finite}
$\KK-$algebra 
$\cA$,
\begin{align*}
DOM^{\cA}(\cR)\subseteq 
DOM^{\cA}(\wt{\cR})
\text{ and }\cR^{\cA}(\ul{\fA})
=\wt{\cR}^{\cA}(\ul{\fA}),
\,\forall \ul{\fA}\in DOM^{\cA}(\cR).
\end{align*}
\end{theorem}
As the proof shows (see
(\ref{eq:13Ap16b}) 
below), 
$\wt{\cR}$
is obtained from 
$\cR$ 
analogously to the classical case,
by restricting to a joint invariant 
subspace of the operators
$\bA_1,\ldots,\bA_d$ 
and then compressing to a co-invariant subspace.
\begin{proof}
Using the controllability and 
un-observability subspaces of 
$\KK^{L}$, 
which correspond to 
$(\ul{\bA},\ul{\bB})$
and
$(C,\ul{\bA})$,
define
$\cC:=
\cC_{\ul{\bA},\ul{\bB}},
\cN\cO:=
\obs{C}{\ul{\bA}}$,
\begin{align}
\label{eq:13Ap16a}
\cH_1:=\cN\cO\cap\cC,
\end{align}
$\cH_2$
a complementary subspace 
of
$\cH_1$ 
in
$\cC$ 
and 
$\cH_3$
a complementary subspace 
of
$\cH_1$
in 
$\cN\cO$, thus
\begin{align}
\label{eq:13Ap16c}
\cC=\cH_2\dotplus\cH_1
\text{ and }
\cN\cO=\cH_1\dotplus\cH_3.
\end{align}
If 
$h_1+h_2+h_3=0$
where
$h_1\in\cH_1,h_2\in\cH_2$
and
$h_3\in\cH_3$,
then 
$h_1+h_2=-h_3\in\cC\cap
\cN\cO\cap\cH_3=\cH_1\cap\cH_3=\{\ul{0}\}$, 
which implies that
$h_1+h_2=-h_3=\ul{0}$ and hence 
$h_1=h_2=h_3=\ul{0}$.
Therefore, the sum 
$\cH_1+\cH_2+\cH_3$
is a direct sum; let 
$\cH_4$ 
be a complementary subspace 
of
$\cH_1+\cH_2+\cH_3$ 
in
$\KK^L$, thus
\begin{align}
\label{eq:13Ap16b}
\KK^{L}=\cH_2
\dotplus\cH_1\dotplus
\cH_4\dotplus\cH_3.
\end{align}
Notice that we are interested in
$\cH_2$
as it is a subspace of a controllable (invariant) subspace
and a complementary subspace of an un-observable subspace.  
Define the dimensions
$L_j:=\dim(\cH_j)$
for
$1\le j\le 4,$ 
so 
$L=L_1+\ldots+L_4$.
From Proposition
\ref{prop:28Sep18b}
we have 
$\cN\cO^{(m)}=
\cN\cO\otimes\KK^m$ 
and
$\cC^{(m)}=\cC\otimes\KK^m$ 
for every 
$m\in\NN$, 
hence
$$\cH_1^{(m)}:=
\cN\cO^{(m)}\cap\cC^{(m)}=
(\cN\cO\otimes\KK^m)
\cap (\cC\otimes\KK^m)
=\cH_1\otimes\KK^m.$$
Define 
$\cH_2^{(m)}:=\cH_2\otimes\KK^m,
\cH_3^{(m)}:=\cH_3\otimes\KK^m$
and
$\cH_4^{(m)}:=\cH_4\otimes\KK^m$, 
thus
\begin{align*}
\cH_2^{(m)}\dotplus\cH_1^{(m)}
=\cC^{(m)},
\cH_1^{(m)}\dotplus\cH_3^{(m)}=\cN\cO^{(m)}
\end{align*}
and 
\begin{align}
\label{eq:18Oct18b}
\KK^{Lm}=\cH_2^{(m)}\dotplus
\cH_1^{(m)}\dotplus\cH_4^{(m)}\dotplus
\cH_3^{(m)}.
\end{align}

$\bullet$ With respect to the decomposition
(\ref{eq:13Ap16b})
of 
$\KK^{L}$,
there exists
$P\in\KK^{L\times L}$ 
invertible such that
for every 
$X\in\KK^{sm\times sm}$ 
and 
$1\le k\le d$,
the matrices 
$(X)\bA_k\in\KK^{Lm\times Lm},\,
(X)\bB_k\in\KK^{Lm\times sm}$ 
and
$I_m\otimes C\in\KK^{sm\times Lm}$
can be decomposed, w.r.t 
(\ref{eq:18Oct18b}), as 
\begin{multline*}
\big(P^{(m)}\big)^{-1}(X)\bA_kP^{(m)}=
(X)\begin{bmatrix}
\bA_k^{1,1}&\bA_k^{1,2}&
\bA_k^{1,3}&\bA_k^{1,4}\\
\bA_k^{2,1}&\bA_k^{2,2}&
\bA_k^{2,3}&\bA_k^{2,4}\\
\bA_k^{3,1}&\bA_k^{3,2}&
\bA_k^{3,3}&\bA_k^{3,4}\\
\bA_k^{4,1}&\bA_k^{4,2}&
\bA_k^{4,3}&\bA_k^{4,4}\\
\end{bmatrix},\, 
\big(P^{(m)}\big)^{-1}(X)\bB_k=
(X)\begin{bmatrix}
\bB_k^1\\
\bB_k^2\\
\bB_k^3\\
\bB_k^4\\
\end{bmatrix}\\
\text{and }
(I_m\otimes C)P^{(m)}
=\begin{bmatrix}
C^1&C^2&C^3&C^4\\
\end{bmatrix},
\end{multline*}
where 
$P^{(m)}:=I_m\otimes P$.
If
$\ul{u}_1\in\cH_1^{(m)},
\ul{u}_2\in\cH_2^{(m)},
\ul{u}_3\in\cH_3^{(m)}$
and
$\ul{v}\in\KK^{sm}$,
then  
\begin{multline*}
(X)\bA_k\ul{u}_1\in\cN\cO^{(m)}
\cap\cC^{(m)}=\cH_1^{(m)},\,
(X)\bA_k\ul{u}_2\in\cC^{(m)},\,
(X)\bA_k\ul{u}_3\in\cN\cO^{(m)},\\ 
(X)\bB_k\ul{v}\in\cC^{(m)}
\text{ and }(I_m\otimes C)\ul{u}_1=\ul{0},
\end{multline*}
which imply that 
\begin{equation*}
(X)\bA_k^{1,2}
,(X)\bA_k^{3,2}
,(X)\bA_k^{4,2},
(X)\bA_k^{3,1}
,(X)\bA_k^{4,1},
(X)\bA_k^{1,4},
(X)\bA_k^{3,4}, 
(X)\bB_k^3,(X)\bB_k^4, 
C^2,C^4
\end{equation*}
all vanish.
Therefore, we get 
\begin{multline*}
\big(P^{(m)}\big)^{-1}(X)\bA_kP^{(m)}=
(X)\begin{bmatrix}
\bA_k^{1,1}&0&\bA_k^{1,3}&0\\
\bA_k^{2,1}&\bA_k^{2,2}&
\bA_k^{2,3}&\bA_k^{2,4}\\
0&0&\bA_k^{3,3}&0\\
0&0&\bA_k^{4,3}&\bA_k^{4,4}\\
\end{bmatrix},\, 
\big(P^{(m)}\big)^{-1}(X)\bB_k=(X)
\begin{bmatrix}
\bB_k^1\\
\bB_k^2\\
0\\
0\\
\end{bmatrix}\\
\text{and }
(I_m\otimes C)P^{(m)}=
\begin{bmatrix}
C^1&0&C^3&0\\
\end{bmatrix},
\end{multline*}
hence for every 
$\ul{X}\in DOM_{sm}(\cR)$,
\begin{multline*}
\cR(\ul{X})=I_m\otimes D+
(I_m\otimes C)P^{(m)}
\Big( 
I_{Lm}-\sum_{k=1}^d
\big[\big(P^{(m)}\big)^{-1}(X_k-I_m\otimes Y_k)
\bA_kP^{(m)}\big]
\Big)
^{-1}
\\\sum_{k=1}^d 
\big[\big(P^{(m)}\big)^{-1}
(X_k-I_m\otimes Y_k)\bB_k\big]
=I_m\otimes D+ 
\begin{bmatrix}
C^1&0&C^3&0\\
\end{bmatrix}
\begin{bmatrix}
\lam_{1,1}&0&\lam_{1,3}&0\\
\lam_{2,1}&\lam_{2,2}&
\lam_{2,3}&\lam_{2,4}\\
0&0&\lam_{3,3}&0\\
0&0&\lam_{4,3}&\lam_{4,4}\\
\end{bmatrix}^{-1}
\\\sum_{k=1}^d
(X_k-I_m\otimes Y_k)
\begin{bmatrix} 
\bB_k^1\\
\bB_k^2\\0\\0\\ 
\end{bmatrix},
\end{multline*}
where 
\begin{multline*}
\Lambda_X:=\begin{bmatrix}
\lam_{11}&0&\lam_{13}&0\\
\lam_{21}&\lam_{22}&
\lam_{23}&\lam_{24}\\
0&0&\lam_{33}&0\\
0&0&\lam_{43}&\lam_{44}\\
\end{bmatrix}
=I_{Lm}-\sum_{k=1}^d
(X_k-I_m\otimes Y_k)
\begin{bmatrix}
\bA_k^{1,1}&0&
\bA_k^{1,3}&0\\
\bA_k^{2,1}&
\bA_k^{2,2}&
\bA_k^{2,3}&
\bA_k^{2,4}\\
0&0&
\bA_k^{3,3}&0\\
0&0&
\bA_k^{4,3}&\bA_k^{4,4}\\
\end{bmatrix}
\end{multline*}
is invertible. Thus
\begin{align}
\label{eq:14Apr19c}
\det(\lam_{11})=
\det\Big(I_{\wt{L}m}-\sum_{k=1}^d
(X_k-I_m\otimes Y_k)\bA^{1,1}_k\Big)\ne0,
\end{align} 
where 
$\wt{L}=L_2=\dim(\cC)-
\dim(\cN\cO\cap\cC)$
and the inverse of 
$\Lambda_X$ 
is given by
\begin{align*}\Lambda_X^{-1}=
\begin{bmatrix}
\lam_{11}^{-1}&0&-\lam_{11}^{-1}
\lam_{13}\lam_{33}^{-1}&0\\
-\lam_{22}^{-1}\lam_{21}\lam_{11}^{-1}
&\lam_{22}^{-1}&\lam_{22}^{-1}
(\lam_{24}\lam_{44}^{-1}
\lam_{43}-\lam_{23}+\lam_{21}\lam_{11}^{-1}
\lam_{13})\lam_{33}^{-1}&-\lam_{22}^{-1}
\lam_{24}\lam_{44}^{-1}\\
0&0&\lam_{33}^{-1}&0\\
0&0&-\lam_{44}^{-1}
\lam_{43}\lam_{33}^{-1}&\lam_{44}^{-1}\\
\end{bmatrix}.
\end{align*}
Therefore, for every 
$\ul{X}\in DOM_{sm}(\cR)$, we have
\begin{multline*}
\cR(\ul{X})=\wt{\cR}(\ul{X}):= 
I_m\otimes D+(I_m\otimes C^1)
\Big(I_{m\wt{L}}-
\sum_{k=1}^d 
(X_k-I_m\otimes Y_k)\bA_k^{1,1}
\Big)^{-1}
\sum_{k=1}^d 
(X_k-I_m\otimes Y_k)\bB_k^1 
\end{multline*} 
where
$\wt{\cR}$ 
is the nc Fornasini--Marchesini realization 
described by
$\big(\wt{L},D,C^1,\ul{\bA}^{1,1},
\ul{\bB}^1\big)$ 
and centred at
$\ul{Y}$, 
and 
(\ref{eq:14Apr19c}) 
implies that
$\ul{X}\in DOM_{sm}(\wt{\cR})$.

$\bullet$ 
Next, let
$\ul{\fA}=(\fA_1,\ldots,\fA_d)\in 
(\cA^{\sts})^d$ 
and write
$\fA_k=\sum_{i,j=1}^s
E_{ij}\otimes \fa^{(k)}_{ij}$, where $E_{ij}\in\cE_s$ and
$\fa_{ij}^{(k)}\in\cA$,
then
\begin{multline*}
I_L\otimes 1_{\cA}-\sum_{k=1}^d
(\fA_k-Y_k\otimes 1_{\cA})\bA_k^{\cA}
=I_L\otimes 1_{\cA}
-\sum_{k=1}^d 
\Big[\sum_{i,j=1}^s 
\bA_k(E_{ij})\otimes\fa^{(k)}_{ij}
-\bA_k(Y_k)\otimes 1_{\cA}\Big]
\\=I_L\otimes 1_{\cA}
-\sum_{k=1}^d 
\sum_{i,j=1}^s \Big[
P\begin{bmatrix}
\bA_k^{1,1}&0&\bA_k^{1,3}&0\\
\bA_k^{2,1}&\bA_k^{2,2}&
\bA_k^{2,3}&\bA_k^{2,4}\\
0&0&\bA_k^{3,3}&0\\
0&0&\bA_k^{4,3}&\bA_k^{4,4}\\
\end{bmatrix}(E_{ij})
P^{-1}\Big]\otimes \fa^{(k)}_{ij}
-\sum_{k=1}^d
\Big[P\\
\begin{bmatrix}
\bA_k^{1,1}&0&\bA_k^{1,3}&0\\
\bA_k^{2,1}&\bA_k^{2,2}&
\bA_k^{2,3}&\bA_k^{2,4}\\
0&0&\bA_k^{3,3}&0\\
0&0&\bA_k^{4,3}&\bA_k^{4,4}\\
\end{bmatrix}
(Y_k)P^{-1}\Big]\otimes 1_{\cA}
=(P\otimes 1_{\cA})
\begin{bmatrix}
\lam_{1,\cA}&0&*&0\\
*&*&*&*\\
0&0&*&0\\
0&0&*&*\\\end{bmatrix}
(P\otimes 1_{\cA})^{-1},
\end{multline*}
where
$$\lam_{1,\cA}:=I_{\wt{L}}\otimes 1_{\cA}-\sum_{k=1}^d
\big[\sum_{i,j=1}^s
\bA_k^{1,1}(E_{ij})\otimes
\fa^{(k)}_{ij}-\bA_k^{1,1}(Y_k)
\otimes 1_{\cA}\big].$$
Therefore, if 
$\ul{\fA}\in DOM^{\cA}(\cR)$, 
then 
\begin{align*}
(P\otimes 1_{\cA})^{-1}\Big(I_L\otimes 1_{\cA}-\sum_{k=1}^d
(\fA_k-Y_k\otimes 1_{\cA})\bA_k^{\cA}\Big)(P\otimes 1_{\cA})
=
\begin{bmatrix}
\lam_{1,\cA}&0&*&0\\
*&*&*&*\\
0&0&*&0\\
0&0&*&*\\\end{bmatrix}
\end{align*}
is invertible in
$\cA^{L\times L}$, 
while applying Lemma 
\ref{lem:19Oct18a} 
to obtain that 
\begin{equation*}
\lam_{1,\cA}
=I_{\wt{L}}\otimes 1_{\cA}
-\sum_{k=1}^d \
(\fA_k-Y_k\otimes 1_{\cA})
\big(\bA_k^{1,1}\big)^{\cA}
\end{equation*}
is invertible in
$\cA^{\wt{L}\times\wt{L}}$, 
i.e., that 
$\ul{\fA}\in DOM^{\cA}(\wt{\cR})$.
Moreover, a careful similar computation shows that if 
$\ul{\fA}\in DOM^{\cA}(\cR)$, 
then 
$\cR^{\cA}(\ul{\fA})
=\wt{\cR}^{\cA}(\ul{\fA})$.

$\bullet$
For every 
$\ul{u}\in\KK^s,\, 
1\le k\le d$,
a word
$\omega=g_{i_1}\ldots g_{i_\ell}\in\cG_d$
of length
$\ell\in\NN$
and
$X_1,\ldots,X_{\ell+1}\in\KK^{\sts}$, 
we have
\begin{align*}
\ul{\bA}^{\omega}(X_1,\ldots,X_\ell)
\bB_k(X_{\ell+1})\ul{u}
&=P(P^{-1}\bA_{i_1}(X)P)
\cdots(P^{-1}
\bA_{i_\ell}(X_\ell)P)P^{-1}\bB_k(X_{\ell+1})\ul{u}
\\&=P\begin{bmatrix}
\big(\ul{\bA}^{1,1}\big)^{\omega}(X_1,\ldots,X_\ell)
\bB^1_k(X_{\ell+1})\ul{u}\\
*\\0\\0\\
\end{bmatrix}.
\end{align*}
Thus, for every
$\ul{v}\in\cC$, 
we have
$P^{-1}\ul{v}\in 
\wt{\cC}\oplus\KK^{L_1}
\oplus \{\ul{0}\}\oplus\{\ul{0}\}
\subseteq\KK^L$
and hence 
$\dim(\cC)\le 
\dim(\wt{\cC})+L_1$.
As 
$\dim(\cC)=L_1+L_2$, 
we get that 
$L_2\le \dim(\wt{\cC})$,
while 
$\wt{\cC}\subseteq \KK^{L_2}$ 
yields that 
$\wt{\cC}=\KK^{L_2}=\KK^{\wt{L}}$, 
i.e., the realization 
$\wt{\cR}$ 
is controllable.

$\bullet$ 
For every
$\ul{w}_1\in\KK^{L_1},\,
\ul{w}_2\in\KK^{L_2},\,
\ul{w}_3\in\KK^{L_3}$,
a word
$\omega=g_{i_1}\ldots g_{i_\ell}\in\cG_d$
of length 
$\ell\in\NN$ 
and
$X_1,\ldots,X_{\ell}\in\KK^{\sts},$ 
we have
\begin{multline*}
C\ul{\bA}^{\omega}
(X_1,\ldots,X_\ell)P
\begin{bmatrix}
\ul{w}_2^T&
\ul{w}_1^T&
0&
\ul{w}_3^T\\
\end{bmatrix}^T
=CP(P^{-1}\bA_{i_1}(X_1)P)\cdots 
(P^{-1}\bA_{i_\ell}(X_\ell)P)\\
\begin{bmatrix}
\ul{w}_2^T&
\ul{w}_1^T&
0&
\ul{w}_3^T\\
\end{bmatrix}^T
=CP
\begin{bmatrix}
\bA_{i_1}^{1,1}(X_1)\cdots
\bA_{i_\ell}^{1,1}(X_\ell)\ul{w}_2\\
*\\
0\\
*\\
\end{bmatrix}=C^1
(\ul{\bA}^{1,1})^{\omega}(X_1,\ldots,X_\ell)
\ul{w}_2,
\end{multline*}
therefore 
$\ul{w}_2\in\wt{\cN\cO}:=
\obs{C^1}{\ul{\bA}^{1,1}}$
implies that 
$P\begin{bmatrix}
\ul{w}_2^T&
\ul{w}_1^T&
0&
\ul{w}_3^T\\
\end{bmatrix}^T\in 
\cN\cO.$
As 
$P$ 
is invertible,
$\dim(\cN\cO)\ge 
\dim
(\wt{\cN\cO})+
L_1
+L_3=
\dim(
\wt{\cN\cO})
+\dim(\cN\cO)$
which guarantees that 
$\wt{\cN\cO}=\{\ul{0}\}$,
i.e., that 
$\wt{\cR}$ is observable.
\end{proof}
As
a corollary we get that a nc Fornasini--Marchesini 
realization (of a nc rational expression) is 
minimal if and only if it is both controllable and 
observable:
\begin{theorem}
\label{thm:9May19a}
Let 
$R$
be a nc rational expression in 
$x_1,\ldots,x_d$ 
over 
$\KK$ 
and
$\cR$
be a nc Fornasini--Marchesini realization of
$R$
centred at 
$\ul{Y}\in(\KK^{\sts})^d$.
Then 
$\cR$ 
is minimal if and only if 
$\cR$ is controllable and observable. 
\end{theorem}
\begin{proof}
If 
$\cR$ 
is minimal, then using Theorem  
\ref{thm:2May16a} we must have
$$L\le \wt{L}=
\dim(\cC^{(1)})-
\dim(
\cC^{(1)}\cap\cN\cO^{(1)})\le L,$$
which implies that 
$\wt{L}=L,\,\dim(\cC^{(1)})=L$ 
and 
$\dim(\cN\cO^{(1)})=0$, 
i.e., that 
$\cR$ 
is controllable and observable.

On the other hand, let 
$\cR$
be both controllable and observable, and suppose 
$\cR^\prime$ 
is a minimal nc Fornasini--Marchesini realization of 
$R$ centred at 
$\ul{Y}$ 
of dimension 
$L^\prime$.
Thus, by the first part of the theorem,  
$\cR^\prime$ 
is controllable and observable, as well as 
$\cR$, 
so Theorem 
\ref{thm:2May16b}
implies that 
$L=L^\prime$  
and hence 
$\cR$ 
is minimal.
\end{proof}
Minimal nc Fornasini--Marchesini realizations 
are playing a central role in the analysis of the domains; one of the reasons is that they admit the maximal domain among all nc Fornasini--Marchesini realizations--- 
of a given nc rational expression--- which are centred at the same point:
\begin{lemma}
\label{lem:27Jan18a} 
Let 
$\cR_1,\cR_2$ 
be two nc Fornasini--Marchesini realizations
of a nc rational expression 
$R$, 
both centred at 
$\ul{Y}\in(\KK^{\sts})^d$. 
If 
$\cR_2$ 
is minimal, then
\begin{align}
\label{eq:27Jan18a}
DOM_{sm}(\cR_1)
\subseteq DOM_{sm}(\cR_2)
\text{ and }
\cR_1(\ul{X})=\cR_2(\ul{X}),
\forall
\ul{X}\in DOM_{sm}(\cR_1)
\end{align}
for every $m\in\NN$,
and for any 
unital stably finite $\KK-$algebra
$\cA$:
\begin{align*}
DOM^{\cA}(\cR_1)\subseteq
DOM^{\cA}(\cR_2)
\text{ and }\cR_1^{\cA}
(\ul{\fA})=
\cR_2^{\cA}(\ul{\fA}),\,
\forall \ul{\fA}\in DOM^{\cA}(\cR_1).
\end{align*}
\end{lemma}
\begin{proof}
Applying Theorem
\ref{thm:2May16a}
for the nc Fornasini--Marchesini realization 
$\cR_1$, 
there exists a 
minimal nc Fornasini--Marchesini realization 
$\wt{\cR_1}$
centred at
$\ul{Y}$
for which
\begin{align*}
DOM_{sm}(\cR_1)
\subseteq DOM_{sm}
(\wt{\cR_1})
\text{ and }\cR_1(\ul{X})
=\wt{\cR_1}(\ul{X}),
\forall \ul{X}\in 
DOM_{sm}(\cR_1)
\end{align*} 
for every
$m\in\NN$,
and for any unital stably finite $\KK-$algebra
$\cA$:
\begin{align*}
DOM^{\cA}(\cR_1)\subseteq DOM^{\cA}(\wt{\cR_1})
\text{ and }\cR_1^{\cA}(\ul{\fA})=
\wt{\cR_1}^{\cA}(\ul{\fA}),\,
\forall \ul{\fA}\in DOM^{\cA}(\cR_1).
\end{align*}
In particular 
$\wt{\cR_1}$ 
is a nc Fornasini--Marchesini realization of 
$R$.
As both 
$\cR_2$ 
and 
$\wt{\cR_1}$ 
are minimal nc Fornasini--Marchesini realizations of 
$R$, 
both centred at 
$\ul{Y}$, 
Theorem
\ref{thm:2May16b} 
implies that
\begin{align*}
DOM_{sm}(\wt{\cR_1})
=DOM_{sm}(\cR_2)
\text{ and }\wt{\cR_1}
(\ul{X})=\cR_2(\ul{X}),
\forall \ul{X}\in 
DOM_{sm}(\wt{\cR_1})
\end{align*}
and 
\begin{align*}
DOM^{\cA}(\wt{\cR_1})=DOM^{\cA}(\cR_2)
\text{ and }\wt{\cR_1}^{\cA}(\ul{\fA})=
\cR_2^{\cA}(\ul{\fA}),\,
\forall\ul{\fA}\in DOM^{\cA}(\wt{\cR_1}).
\end{align*}
Therefore, 
$DOM_{sm}(\cR_1)\subseteq 
DOM_{sm}(\wt{\cR_1})
=DOM_{sm}(\cR_2)$
and
$\cR_1(\ul{X})=\wt{\cR_1}
(\ul{X})=\cR_2(\ul{X})$
for every 
$\ul{X}\in 
DOM_{sm}(\cR_1)$.
Moreover, 
$DOM^{\cA}(\cR_1)\subseteq 
DOM^{\cA}(\wt{\cR_1})=
DOM^{\cA}(\cR_2)$
and 
$\cR_1^{\cA}(\ul{\fA})
=\wt{\cR_1}^{\cA}(\ul{\fA})
=\cR_2^{\cA}(\ul{\fA})$
for every 
$\ul{\fA}\in DOM^{\cA}(\cR_1)$.
\end{proof}
The following is a summary of all of the main results 
in this subsection. 
\begin{cor}
\label{cor:13Feb18a}
If 
$R$ 
is a nc rational expression in 
$x_1,\ldots,x_d$ 
over 
$\KK$ 
and 
$\ul{Y}\in dom_s(R)\subseteq(\KK^{\sts})^d$, 
then 
$R$ 
admits a unique (up to unique similarity) 
minimal nc Fornasini--Marchesini realization 
centred at 
$\ul{Y}$ 
that is also a realization of 
$R$ 
w.r.t any 
unital stably finite $\KK-$algebra.

Moreover, any minimal nc Fornasini--Marchesini realization of 
$R$
centred at 
$\ul{Y}$
is a realization of
$R$
w.r.t any
unital stably finite 
$\KK-$algebra. 
\end{cor}
\begin{proof}
From Theorem 
\ref{thm:21Ap17a}, 
$R$ admits a nc Fornasini--Marchesini realization 
$\cR$ 
centred at 
$\ul{Y}$ 
that is also a nc Fornasini--Marchesini realization of 
$R$
w.r.t any 
unital stably finite
$\KK-$algebra $\cA$,
while Theorem
\ref{thm:2May16a} 
guarantees
the existence of a minimal nc Fornasini--Marchesini realization 
$\wt{\cR}$
 centred at 
$\ul{Y}$, for which
$$DOM_{sm}(\cR)\subseteq 
DOM_{sm}(\wt{\cR})
\text{ and }
\cR(\ul{X})=\wt{\cR}(\ul{X}),\,
\forall \ul{X}\in 
DOM_{sm}(\cR)$$
for every
$m\in\NN$, and 
$$DOM^{\cA}(\cR)\subseteq DOM^{\cA}(\wt{\cR})
\text{ and }\cR^{\cA}(\ul{\fA})=\wt{\cR}^{\cA}(\ul{\fA}),\,
\forall \ul{\fA}\in DOM^{\cA}(\cR). $$
Therefore, for every $m\in\NN$,
\begin{equation*}
dom_{sm}(R)\subseteq 
DOM_{sm}(\wt{\cR})
\text{ and }
R(\ul{X})=\wt{\cR}(\ul{X}),\,
\forall \ul{X}
\in dom_{sm}(R),
\end{equation*}
i.e., 
$\wt{\cR}$ is a realization of $R$, 
while the uniqueness of 
$\wt{\cR}$ 
follows from Theorem 
\ref{thm:2May16b}. 

$\bullet$ Moreover, if 
$\ul{\fa}\in dom^{\cA}(R)$, then 
$I_s\otimes\ul{\fa}\in DOM^{\cA}(\wt{\cR})$
and
$I_s\otimes 
R^{\cA}(\ul{\fa})=
\cR^{\cA}(I_s\otimes \ul{\fa})=
\wt{\cR}^{\cA}(I_s\otimes \ul{\fa}),$
i.e., $\wt{\cR}$ is a realization of $R$ w.r.t $\cA$.

$\bullet$ Furthermore, if 
$\check{\cR}$ is a minimal nc Fornasini--Marchesini realization of 
$R$
centred at
$\ul{Y}$, then 
$\wt{\cR}$ 
and
$\check{\cR}$ 
are both minimal nc Fornasini--Marchesini realizations of 
$R$ 
centred at
$\ul{Y}$, 
hence by Lemma
\ref{thm:2May16b},
$DOM^{\cA}(\check{\cR})=
DOM^{\cA}(\wt{\cR})
\text{ and }
\check{\cR}^{\cA}(\ul{\fA})
=\wt{\cR}^{\cA}(\ul{\fA})$,
for every
$\ul{\fA}\in DOM^{\cA}(\check{\cR}).$
Therefore, for every 
$\ul{\fa}\in dom^{\cA}(R)$
we have
$I_s\otimes\ul{\fa}\in 
DOM^{\cA}(\wt{\cR})=DOM^{\cA}(\check{\cR})$ 
and 
$I_s\otimes R^{\cA}(\ul{\fa})=
\wt{\cR}^{\cA}(I_s\otimes\ul{\fa})
=\check{\cR}^{\cA}(I_s\otimes\ul{\fa})$, 
i.e.,
$\check{\cR}$ 
is a nc Fornasini--Marchesini realization of 
$R$ w.r.t 
$\cA$.
\end{proof}

\subsection{Example}
\label{subsec:ex}
Consider the nc rational expression 
$R(x_1,x_2)=(x_1x_2-x_2x_1)^{-1}$, 
with 
$\KK=\CC,\,s=2$ 
and
$\ul{Y}=(Y_1,Y_2)=\left(
\begin{pmatrix}
0&1\\
1&0\\
\end{pmatrix},
\begin{pmatrix}
1&0\\
0&-1\\
\end{pmatrix}\right)\in dom_2(R)$.
We use synthesis and follow the proof of 
Theorem \ref{thm:21Ap17a}, 
to find a nc Fornasini--Marchesini realization of 
$R$
centred at
$\ul{Y}$:\\
$\bullet\, R_1(\ul{x})=x_1$
admits a nc Fornasini--Marchesini realization centred at 
$\ul{Y}$ 
(see (\ref{eq:20Mar19b})),
described by
$$L_1=2,\,
D_1=\begin{pmatrix}0&1\\1&0\\
\end{pmatrix},\,
C_1=I_2,\,
\bA_1^1=\bA_2^1=0_2,\,
\bB_1^1=Id_2,\,
\bB_2^1=0_2.$$
$\bullet\, R_2(\ul{x})=x_2$
admits a  nc Fornasini--Marchesini  realization centred at $\ul{Y}$ 
(see
(\ref{eq:20Mar19b})), described by
$$L_2=2,\,
D_2=\begin{pmatrix}
1&0\\
0&-1\\
\end{pmatrix},\,
C_2=I_2,\,
\bA_1^2=\bA_2^2=0_2,\,
\bB_1^2=0_2,\,
\bB_2^2=Id_2.$$
$\bullet \,R_3(\ul{x})=R_1(\ul{x})R_2(\ul{x})=x_1x_2$ 
admits a nc Fornasini--Marchesini realization centred at 
$\ul{Y}$ 
(see 
(\ref{eq:20Mar19d})), described by 
\begin{multline*}
L_3=4,\,
D_3=\begin{pmatrix}0&-1\\1&0\\
\end{pmatrix},\,
C_3=
\begin{pmatrix}1&0&0&1\\
0&1&1&0\\
\end{pmatrix},\,
\bA_1^3(X)=(I_2\otimes X)
\begin{bmatrix}
0_2&I_2\\
0_2&0_2\\
\end{bmatrix},\,\\
\bA_2^3(X)=0_4,\,
\bB_1^3(X)=(I_2\otimes X)
\begin{bmatrix}
\begin{pmatrix}
1&0\\0&-1\\
\end{pmatrix}\\
0_2\\
\end{bmatrix},\,
\bB_2^3(X)=(I_2\otimes X)
\begin{bmatrix}
0_2\\I_2\\
\end{bmatrix}.
\end{multline*}
$\bullet\,R_4(\ul{x})=-R_2(\ul{x})R_1(\ul{x})=-x_2x_1$
admits a  nc Fornasini--Marchesini  realization centred at 
$\ul{Y}$
(see 
(\ref{eq:20Mar19d})),described by
\begin{multline*}
L_4=4,\,
D_4
=\begin{pmatrix}
0&-1\\
1&0\\
\end{pmatrix},\,
C_4=
\begin{pmatrix}
-1&0&-1&0\\
0&-1&0&1\\
\end{pmatrix},\,
\bA_1^4(X)=0_4,\,
\\\bA_2^4(X)=(I_2\otimes X)
\begin{bmatrix}
0_2&I_2\\
0_2&0_2\\
\end{bmatrix},\,
\bB_1^4(X)=(I_2\otimes X)
\begin{bmatrix}
0_2\\I_2\\
\end{bmatrix},\,
\bB_2^4(X)=(I_2\otimes X)
\begin{bmatrix}
\begin{pmatrix}
0&1\\1&0\\
\end{pmatrix}\\
0_2\\
\end{bmatrix}.
\end{multline*}
$\bullet\,R_5(\ul{x})=R_3(\ul{x})+R_4(\ul{x})
=x_1x_2-x_2x_1$
admits a  nc Fornasini--Marchesini  
realization centred at $\ul{Y}$
(see (\ref{eq:20Mar19c})), described by
\begin{multline*}
L_5=8,\,
D_5=
\begin{pmatrix}
0&-2\\
2&0\\
\end{pmatrix},\,
C_5=
\begin{pmatrix}
1&0&0&1&-1&0&-1&0\\
0&1&1&0&0&-1&0&1\\
\end{pmatrix},\\
\bA_1^5(X)=
(I_4\otimes X)
\begin{bmatrix}
0_2&I_2&0_2&0_2\\
0_2&0_2&0_2&0_2\\
0_2&0_2&0_2&0_2\\
0_2&0_2&0_2&0_2\\
\end{bmatrix},\,
\bA_2(X)=(I_4\otimes X)
\begin{bmatrix}
0_2&0_2&0_2&0_2\\
0_2&0_2&0_2&0_2\\
0_2&0_2&0_2&I_2\\
0_2&0_2&0_2&0_2\\
\end{bmatrix},\\
\bB_1(X)=(I_4\otimes X)
\begin{bmatrix}
\begin{pmatrix}
1&0\\0&-1\\
\end{pmatrix}\\
0_2\\
0_2\\
I_2\\
\end{bmatrix},\,
\bB_2(X)=(I_4\otimes X)
\begin{bmatrix}
0_2\\
I_2\\
\begin{pmatrix}
0&1\\1&0\\
\end{pmatrix}\\
0_2\
\end{bmatrix}.
\end{multline*}
$\bullet\,R(\ul{x})
=R_5(\ul{x})^{-1}=(x_1x_2-x_2x_1)^{-1}$
admits a  nc Fornasini--Marchesini  realization 
$\cR$ 
centred at 
$\ul{Y}$
(see
(\ref{eq:20Mar19e})), 
described by
\begin{multline*}
L=8,\,
D=\frac{1}{2}\begin{pmatrix}
0&1\\
-1&0\\
\end{pmatrix},\,
C=\frac{1}{2}
\begin{pmatrix}
0&1&1&0&0&-1&0&1\\
-1&0&0&-1&1&0&1&0\\
\end{pmatrix},
\\\bA_1(X)=
\frac{1}{2}(I_4\otimes X)
\begin{bmatrix}
\begin{pmatrix}0&-1\\-1&0\\
\end{pmatrix}&I_2&\begin{pmatrix}
0&1\\1&0\\
\end{pmatrix}&\begin{pmatrix}
0&-1\\1&0\\
\end{pmatrix}\\
0_2&0_2&0_2&0_2\\
0_2&0_2&
0_2&0_2\\
\begin{pmatrix}
0&-1\\1&0\\
\end{pmatrix}&\begin{pmatrix}
-1&0\\0&1\\
\end{pmatrix}&\begin{pmatrix}
0&1\\-1&0\\
\end{pmatrix}&\begin{pmatrix}
0&-1\\-1&0\\
\end{pmatrix}\\
\end{bmatrix},
\end{multline*}
\begin{align*}
\bA_2(X)&=\frac{1}{2}(I_4\otimes X)
\begin{bmatrix}
0_2&0_2&0_2&0_2\\
\begin{pmatrix}0&-1\\1&0\\
\end{pmatrix}&\begin{pmatrix}
-1&0\\0&1\\
\end{pmatrix}&\begin{pmatrix}
0&1\\-1&0\\
\end{pmatrix}&\begin{pmatrix}
0&-1\\-1&0\\
\end{pmatrix}\\
\begin{pmatrix}
1&0\\0&-1\\
\end{pmatrix}&\begin{pmatrix}
0&1\\-1&0\\
\end{pmatrix}&\begin{pmatrix}
-1&0\\0&1\\
\end{pmatrix}&I_2\\
0_2&0_2&0_2&0_2\\
\end{bmatrix},
\end{align*}
\begin{align*}
\bB_1(X)=\frac{1}{2}(I_4\otimes X)
\begin{bmatrix}
\begin{pmatrix}
0&-1\\
-1&0\\
\end{pmatrix}\\
0_2\\
0_2\\
\begin{pmatrix}
0&-1\\
1&0\\
\end{pmatrix}\\
\end{bmatrix},\,
\bB_2(X)=\frac{1}{2}
(I_4\otimes X)
\begin{bmatrix}
0_2\\ 
\begin{pmatrix}
0&-1\\1&0\\
\end{pmatrix}\\
\begin{pmatrix}
1&0\\0&-1\\
\end{pmatrix}\\
0_2\\
\end{bmatrix}.
\end{align*}
The nc Fornasini--Marchesini realization $\cR$ is controllable,
however $\cR$ is not observable, as
$$\obs{C}{\ul{\bA}}=
\bigcap_{\omega\in\cG_2,\,
Z_1,\ldots,Z_{\ell}\in\CC^{2\times2}}
\ker\big(C\ul{\bA}^{\omega}(Z_1,\ldots,Z_{\ell})\big)
=span\{\ul{e}_1+\ul{e}_5,\ul{e}_2+\ul{e}_6\},$$
hence 
$\cR$
is not minimal. By the Kalman decomposition 
(see Theorem  
\ref{thm:2May16a})
argument, we obtain a minimal nc Fornasini--Marchesini realization of 
$R$ centred at 
$\ul{Y}$
of dimension 
$\wt{L}=\dim(\cC_{\ul{\bA},\ul{\bB}})-\dim(\obs{C}{\ul{\bA}}\cap\cC_{\ul{\bA},\ul{\bB}})=6$, that is
\begin{align*}
\wt{\cR}(X_1,X_2)=
\wt{D}+\wt{C}
\big(I_6-
\wt{\bA_1}(X_1-Y_1)-\wt{\bA_2}(X_2-Y_2)\big)^{-1} 
\big(\wt{\bB_1}(X_1-Y_1)+\wt{\bB_2}(X_2-Y_2)\big),
\end{align*}
with  
\begin{multline*}
\wt{D}=\frac{1}{2}\begin{pmatrix}
0&1\\
-1&0\\
\end{pmatrix},\, 
\wt{C}
=\frac{1}{2}\begin{pmatrix}
1&0&0&1&0&2\\
0&-1&1&0&-2&0\\
\end{pmatrix},\\
\wt{\bA_1}(X)=\frac{1}{2}(I_3\otimes X)
\begin{bmatrix}
0_2&0_2&0_2\\
\begin{pmatrix}
-1&0\\0&1\\
\end{pmatrix}&
\begin{pmatrix}
0&-1\\-1&0\\
\end{pmatrix}&
\begin{pmatrix}
0&-2\\2&0\\
\end{pmatrix}\\
\begin{pmatrix}
\frac{1}{2}&0\\
0&\frac{1}{2}\\
\end{pmatrix}&\begin{pmatrix}
0&-\frac{1}{2}\\
\frac{1}{2}&0\\
\end{pmatrix}&
\begin{pmatrix}
0&-1\\ -1&0\\
\end{pmatrix}\\
\end{bmatrix},
\\\wt{\bA_2}(X)=\frac{1}{2}
(I_3\otimes X)
\begin{bmatrix}
\begin{pmatrix}
-1&0\\0&1\\
\end{pmatrix}&
\begin{pmatrix}
0&-1\\ -1&0\\
\end{pmatrix}&
\begin{pmatrix}
0&-2\\2&0\\
\end{pmatrix}\\
0_2&0_2&0_2\\
\begin{pmatrix}
0&-\frac{1}{2}\\ \frac{1}{2}&0\\
\end{pmatrix}
&\begin{pmatrix}
-\frac{1}{2}&0\\0&-\frac{1}{2}\\
\end{pmatrix}&
\begin{pmatrix}
-1&0\\0&1\\
\end{pmatrix}\\
\end{bmatrix},\\
\wt{\bB_1}(X)=
\frac{1}{2}(I_3\otimes X)
\begin{bmatrix}
0_2\\
\begin{pmatrix}
0&-1\\ 1&0\\
\end{pmatrix}\\
\begin{pmatrix}
0&-\frac{1}{2}\\
-\frac{1}{2}&0\\
\end{pmatrix}\\
\end{bmatrix},\,
\wt{\bB_2}(X)=
\frac{1}{2}(I_3\otimes X)
\begin{bmatrix}
\begin{pmatrix}
0&-1\\ 1&0\\
\end{pmatrix}\\
0_2\\
\begin{pmatrix}
-\frac{1}{2}&0\\
0&\frac{1}{2}\\
\end{pmatrix}\\
\end{bmatrix}.
\end{multline*}
In this example it is easy to check directly that
\begin{align*}
\,\det \Big(I_{6m}-\sum_{k=1}^2 
(X_k-I_m\otimes Y_k)\wt{\bA_k}\Big)\ne 0
\iff
\det(X_1X_2-X_2X_1)\ne0
\end{align*} 
for 
$\ul{X}=(X_1,X_2)\in(\CC^{2m\times 2m})^2$,
i.e., that 
$dom_{2m}(R)=DOM_{2m}(\wt{\cR})$ 
and also that 
$R(\ul{X})=\wt{\cR}(\ul{X})$.
Furthermore, for any unital stably finite 
$\CC-$algebra $\cA$ and  
$\ul{\fa}
=(\fa_1,\fa_2)\in\cA^2$, 
the element
$\fa_1\fa_2-\fa_2\fa_1$
is invertible in  
$\cA$
if and only if the matrix 
\begin{align*}
\Big(I_6\otimes 1_{\cA}-\sum_{k=1}^2 
\wt{\bA_k}
(I_2)\otimes \fa_k
-\wt{\bA_k}(Y_k)\otimes 1_{\cA}\Big)
\text{ is invertible in }\cA^{6\times6}
\end{align*}
i.e., 
$dom^{\cA}(R)=DOM^{\cA}(\wt{\cR})$ 
and also 
$I_2\otimes R^{\cA}(\ul{\fa})
=\cR^{\cA}(I_2\otimes\ul{\fa})$.
\subsection{Cohn's theorem}
As a corollary of the results in Subsection
\ref{sec:Kal}, we get a new proof 
of a theorem of Cohn, stating that if two nc rational expressions represents the same nc rational function, then they are 
$\cA-$evaluation equivalent for any unital stably finite 
$\KK-$algebra 
$\cA$. 
Cohn's Theorem was proved originally 
in 
\cite{Co82}
and afterwards in his book
\cite[Theorem 7.3.2]{Co95}. 
In  
\cite{HMS}, 
the authors proved a weaker version 
of the theorem, applicable only for nc rational 
expressions which are regular at 
the origin, using their theory of realizations for nc regular rational functions. 
We omit their assumption on regularity at
the origin and prove 
the theorem in its full version.
\begin{theorem}[Cohn's Theorem]
\label{thm:22Oct18a}
Let 
$R$ 
and 
$\wt{R}$ 
be 
$(\KK^d)_{nc}-$evaluation equivalent nc 
rational expressions in 
$x_1,\ldots,x_d$ 
over 
$\KK$, 
i.e.,
$R(\ul{X})=\wt{R}(\ul{X})$
for all 
$\ul{X}\in 
dom(R)\cap 
dom(\wt{R})$.
Then $R$ and $\wt{R}$ 
are 
$\cA-$evaluation equivalent
for any unital stably finite 
$\KK-$algebra $\cA$ , 
i.e.,
$$R^{\cA}(\ul{\fa})=\wt{R}^{\cA}(\ul{\fa}),\,
\forall\ul{\fa}
\in dom^{\cA}(R)
\cap dom^{\cA}(\wt{R}).$$
\end{theorem}
\begin{proof}
Assume
$R$ and $\wt{R}$ 
are non-degenerate nc rational expressions, so
there exists 
$s\in\NN$ such that
$dom_s(R),dom_s(\wt{R})\ne\emptyset$.
As
$dom_s(R),dom_s(\wt{R})$ 
are Zariski open sets in 
$\KK^{\sts}$,
there exists 
$\ell\in\NN$ 
such that 
$dom_{s\ell}(R)\cap 
dom_{s\ell}(\wt{R})\ne \emptyset$.
The reasoning for that is clear when 
$\KK$
is infinite (with 
$\ell=1$), 
while if 
$\KK$
is finite we use a similar argument as in
\cite[Remark 2.6]{KV3}. 
Let
\begin{equation*}
\ul{Y}=(Y_1,\ldots,Y_d)\in 
dom_{s\ell}(R)\cap dom_{s\ell}(\wt{R}).
\end{equation*}
From Corollary
\ref{cor:13Feb18a}, 
$R$ 
and 
$\wt{R}$ 
admit minimal nc Fornasini--Marchesini realizations 
$\cR$
and
$\wt{\cR}$, respectively, both centred at 
$\ul{Y}$ 
with the special properties: 
\begin{align*}
\ul{\fa}\in dom^{\cA}(R)\Longrightarrow 
 I_{s\ell}\otimes \ul{\fa}\in DOM^{\cA}(\cR)
\text{ and }I_{s\ell}\otimes 
R^{\cA}(\ul{\fa})
=\cR^{\cA}(I_{s\ell}\otimes \ul{\fa}),
\\
\ul{\fa}\in dom^{\cA}(\wt{R})
\Longrightarrow I_{s\ell}\otimes \ul{\fa}\in 
DOM^{\cA}(\wt{\cR})
\text{ and }I_{s\ell}\otimes 
\wt{R}^{\cA}(\ul{\fa})
=\wt{\cR}^{\cA}(I_{s\ell}\otimes \ul{\fa}).
\end{align*}
Moreover, Theorem 
\ref{thm:2May16b} implies that
\begin{align*}
DOM^{\cA}(\cR)=DOM^{\cA}(\wt{\cR})
\text{ and }\cR^{\cA}(\ul{\fA})
=\wt{\cR}^{\cA}(\ul{\fA}),\,
\forall \ul{\fA}\in DOM^{\cA}(\cR).
\end{align*}
Finally, if 
$$\ul{\fa}\in dom^{\cA}(R)
\cap dom^{\cA}(\wt{R}),$$ 
then
$I_{s\ell}\otimes\ul{\fa}\in 
DOM^{\cA}(\cR)=DOM^{\cA}(\wt{\cR})$ 
and
$I_{s\ell}\otimes 
R^{\cA}(\ul{\fa})=
\cR^{\cA}(I_{s\ell}\otimes \ul{\fa})=
\wt{\cR}^{\cA}(I_{s\ell}\otimes \ul{\fa})=
I_{s\ell}\otimes 
\wt{R}^{\cA}(\ul{\fa})$,
therefore 
$R^{\cA}(\ul{\fa})=
\wt{R}^{\cA}(\ul{\fa})$.
\end{proof}
\begin{remark}
Theorem 
\ref{thm:22Oct18a}
implies that one can evaluate any nc rational function by evaluating a minimal realization of the function. This proves that
$\KK\plangle\ul{x}\prangle$ 
is the universal skew field of fractions of 
$\KK\langle\ul{x}\rangle$, 
see
\cite{Co71a,Row80} for the original proofs and 
\cite{KV3}
for a modern proof.
We postpone a detailed discussion and  an application to an explicit
construction of 
$\KK\plangle\ul{x}\prangle$
to
\cite{PV3}.
\end{remark}
\subsection{The McMillan degree}
\label{subsec:McMillan}
For a nc rational expression 
$R$ 
in 
$x_1,\ldots,x_d$ 
over 
$\KK$
and 
$\ul{Y}\in dom(R)$, 
we define by 
$L_R(\ul{Y})$ 
the dimension of a minimal nc Fornasini--Marchesini realization of 
$R$ 
centred at 
$\ul{Y}$.
The first part of the next  theorem is an analog  of
\cite[Theorem 5.10]{V2}.
\begin{theorem}
\label{thm:22Oct17c}
Let 
$R$ 
be a non-degenerate nc rational expression in
$x_1,\ldots,x_d$
over
$\KK$
and let 
$\ul{Y}\in dom_s(R).$ 
\begin{enumerate}
\item[1.]
If
$\wt{\ul{Y}}\in dom_s(R)$,
then
$L_R(\ul{Y})=L_R(\wt{\ul{Y}}).$
\item[2.]
If 
$n\in\NN$, 
then
$I_n\otimes \ul{Y}\in dom_{sn}(R)$ 
and 
$L_R(I_n\otimes \ul{Y})=nL_R(\ul{Y}).$
\item[3.]
If
$s^\prime\in\NN$ 
and 
$\ul{Y}^\prime\in 
dom_{s^\prime}(R)$, 
then
$s^\prime L_R(\ul{Y})=s L_R(\ul{Y}^\prime)$.
\end{enumerate}
\end{theorem}
\begin{proof}
Applying Corollary
\ref{cor:13Feb18a},
the expression 
$R$ 
admits a minimal nc Fornasini--Marchesini realization 
$\cR$ centred at $\ul{Y}$, described by a tuple 
$(L,D,C,\ul{\bA},\ul{\bB})$. 
Let
\begin{align*}
T_1:=I_L-\sum_{k=1}^d
\bA_k(\wt{Y_k}-Y_k)
\text{ and }
T_2:=\sum_{k=1}^d 
\bB_k(Y_k-\wt{Y_k}),
\end{align*}
as
$\wt{\ul{Y}}=(\wt{Y_1},\ldots,\wt{Y_d})
\in dom_s(R)\subseteq
DOM_s(\cR)$,
the matrix
$T_1$
is invertible. Therefore, for every 
$\ul{X}\in DOM_{sm}(\cR)$ 
and
$m\in\NN$:
\begin{multline*}
\cR(\ul{X})=
I_m\otimes D+(I_m\otimes C)
\Big(I_{Lm}-\sum_{k=1}^d 
(X_k-I_m\otimes \wt{Y_k})
\bA_k-\sum_{k=1}^d\big(I_m\otimes
(\wt{Y_k}-Y_k)\big)\bA_k\Big)^{-1}
\Big(
\sum_{k=1}^d \\
(X_k-I_m\otimes \wt{Y_k})\bB_k
-
\sum_{k=1}^d(I_m\otimes(
Y_k-\wt{Y_k}))\bB_k\Big)
=I_m\otimes D+(I_m\otimes CT_1^{-1})
\Big(I_{Lm}-\sum_{k=1}^d 
(X_k-I_m\otimes 
\wt{Y_k})\wt{\bA_k}\Big)^{-1}
\\\sum_{k=1}^d 
(X_k-I_m\otimes \wt{Y_k})\bB_k
-(I_m\otimes CT_1^{-1})
\Big(I_{Lm}-\sum_{k=1}^d 
(X_k-I_m\otimes \wt{Y_k})
\wt{\bA_k}\Big)^{-1}
(I_m\otimes T_2),
\end{multline*}
where 
$\wt{\bA_k}=
\bA_k\cdot T_1^{-1}$, 
and since
\begin{multline*}
\Big(
I_{Lm}-\sum_{k=1}^d 
(X_k-I_m\otimes \wt{Y_k})
\wt{\bA_k}\Big)^{-1}
(I_m\otimes T_2)
=I_m\otimes T_2+
\Big(I_{Lm}-\sum_{k=1}^d 
(X_k-I_m\otimes \wt{Y_k})
\wt{\bA_k}\Big)^{-1}
\\\sum_{k=1}^d(X_k-I_m
\otimes\wt{Y_k})\widehat{\bB_k},
\end{multline*}
where 
$\widehat{\bB_k}=\wt{\bA_k}\cdot 
T_2,$
we conclude that
\begin{multline*}
\cR(\ul{X})=I_m\otimes D-I_m\otimes
(CT_1^{-1} 
T_2)
+(I_m\otimes CT_1^{-1})
\Big(I_{Lm}-\sum_{k=1}^d 
(X_k-I_m\otimes \wt{Y_k})
\wt{\bA_k}\Big)^{-1}\\
\sum_{k=1}^d (X_k-I_m\otimes 
\wt{Y_k})\bB_k
-(I_m\otimes CT_1^{-1})
\Big(I_{Lm}-\sum_{k=1}^d 
(X_k-I_m\otimes \wt{Y_k})
\wt{\bA_k}\Big)^{-1}
\sum_{k=1}^d(X_k-I_m
\otimes\wt{Y_k})\widehat{\bB_k}
\end{multline*}
i.e., that $\cR(\ul{X})=\wt{\cR}(\ul{X})$
where
$\wt{\cR}$ 
is a nc Fornasini--Marchesini realization of 
$R$, 
centred at
$\ul{\wt{Y}}$,
described by
\begin{multline}
\label{eq:21Mar19b}
\wt{D}=D-CT_1^{-1}
T_2,\,
\wt{C}=CT_1^{-1},\,
\wt{\bA_k}=\bA_k \cdot
T_1^{-1},\,
\wt{\bB_k}
=\bB_k-\bA_k\cdot
(T_1^{-1}
T_2).
\end{multline}
Thus, 
$L_R(\wt{\ul{Y}})\le L=L_R(\ul{Y})$ 
and by symmetry we get that
$L_R(\wt{\ul{Y}})=L_R(\ul{Y})$,
hence
$\wt{\cR}$ 
is a minimal nc Fornasini--Marchesini realization of $R$ centred at 
$\ul{\wt{Y}}.$

$\bullet$ Suppose next that $n\in\NN$. 
As 
$dom(R)$ 
is closed under direct sums it follows that 
$I_n\otimes \ul{Y}\in dom_{sn}(R)$,
while for every 
$p\in\NN$ 
letting 
$m=np$ 
yields for every 
$\ul{X}\in dom_{snp}(R)$:
\begin{multline*}
\cR(\ul{X})=I_{np}\otimes D
+(I_{np}\otimes C)
\Big(I_{Lnp}-\sum_{k=1}^d 
(X_k-I_{np}\otimes Y_k)\bA_k\Big)^{-1}
\sum_{k=1}^d (X_k-I_{np}
\otimes Y_k)\bB_k
\\=I_{p}\otimes D^{(n)}+(I_{p}\otimes 
C^{(n)})
\Big(I_{L^{(n)}p}-\sum_{k=1}^d 
(X_k-I_{p}\otimes 
Y_k^{(n)})\bA_k\Big)^{-1}
\sum_{k=1}^d (X_k-I_{p}
\otimes Y_k^{(n)})\bB_k
\end{multline*}
where
\begin{align}
\label{eq:21Mar19c}
L^{(n)}=Ln,\,
D^{(n)}=I_n\otimes D,\,
C^{(n)}=I_n\otimes C
\text{ and }
Y_k^{(n)}=I_n\otimes Y_k.
\end{align}
We obtained a nc Fornasini--Marchesini realization
$\cR^{(n)}$--- 
described by the tuple
$(L^{(n)},D^{(n)},C^{(n)},\ul{\bA},\ul{\bB})$---
of 
$R$ 
that is centred at 
$\ul{Y}^{(n)}:=I_n\otimes\ul{Y}$ 
and it is easily seen that controllability and observability of 
$\cR$ imply the  controllability and observability of
$\cR^{(n)}$ as well, 
thus $\cR^{(n)}$
is minimal and hence 
$L_R(I_n\otimes \ul{Y})=nL_R(\ul{Y}).$

$\bullet$ Finally,
let 
$\ul{Y}^\prime\in 
dom_{s^\prime}(R)$,
then part 
\textbf{2} 
implies that
$I_{s^\prime}\otimes 
\ul{Y},I_s\otimes \ul{Y}^\prime\in 
dom_{ss^\prime}(R)$,
$L_R(I_{s^\prime}\otimes \ul{Y})
=s^\prime L_R(\ul{Y})$
and 
$L_R(I_s\otimes \ul{Y}^\prime)
=sL_R(\ul{Y}^\prime),$
while from part \textbf{1},
$L_R(I_{s^\prime}\otimes \ul{Y})
=L_R(I_s\otimes \ul{Y}^\prime),$
therefore 
$s^\prime L_R(\ul{Y})
=sL_R(\ul{Y}^\prime)$.
\end{proof}
\begin{remark}
\label{rem:18Nov18b}
In the proof of Theorem
\ref{thm:22Oct17c}
we built explicit minimal  nc Fornasini--Marchesini 
realizations
$\wt{\cR}$ and $\cR^{(n)}$ of
$R$,
centred at 
$\wt{\ul{Y}}$ 
and 
$I_n\otimes \ul{Y}$, respectively, 
using a minimal realization 
$\cR$ of 
$R$ centred at
$\ul{Y}$. 
From Corollary
\ref{cor:13Feb18a} 
it follows right away that 
$\wt{\cR}$ 
and 
$\cR^{(n)}$
are also nc Fornasini--Marchesini realizations of 
$R$ 
w.r.t any unital stably finite 
$\KK-$algebra
$\cA$.
Moreover,  direct computations--- which are omitted--- easily show that
\begin{align*}
DOM^{\cA}(\cR)=DOM^{\cA}(\wt{\cR})
\text{ and }
\cR^{\cA}(\ul{\fA})=\wt{\cR}^{\cA}(\ul{\fA}),\,
\forall \ul{\fA}\in DOM^{\cA}(\cR)
\end{align*}
and
\begin{align*}
\ul{\fA}\in DOM^{\cA}(\cR)
\iff I_n\otimes\ul{\fA}\in 
DOM^{\cA}\big(\cR^{(n)}\big)
\text{ and }
I_n\otimes \cR^{\cA}(\ul{\fA})=
\big(\cR^{(n)}\big)^{\cA}(I_n\otimes\ul{\fA}).
\end{align*}
\end{remark}
The first part of Theorem
\ref{thm:22Oct17c} 
guarantees that the value 
$L_R(\ul{Y})$ 
does not depend on 
$\ul{Y}$
but only on 
$s$, 
so it will be denoted 
$L_R(s):=L_R(\ul{Y})$, 
while from the third part of the theorem it follows that there exists
$\fm(R)>0$ such that
\begin{align}
\label{eq:19Feb19a}
L_R(s)=\fm(R)s,
\quad\forall s\ge 1
\end{align}
where 
$\fm(R)$ 
depends only on 
$R$;
We define 
$\fm(R)$ 
as the 
\textbf{McMillan degree} of 
$R$.

In the next lemma we actually show that 
$\fm(R)\in\NN$.
This is a direct corollary and yet separated from the arguments of Theorem
\ref{thm:22Oct17c}, 
as it requires a non-trivial tool 
from PI-ring theory, that is  if 
$R$
is a non-degenerate nc rational expression, then there exists 
$n\in\NN$ 
such
that 
$dom_k(R)\ne\emptyset$ 
for every 
$k\ge n$;
see 
\cite[Chapter 8]{Row80} and \cite[Remarks 2.15 and 2.16]{KV4}
for a more detailed discussion. 
\begin{lemma}
\label{rem:18Nov18a}
If
$R$
is a non-degenerate nc rational expression in
$x_1,\ldots,x_d$
over
$\KK$ 
and 
$dom_s(R)\ne\emptyset$,
then
$s\mid L_R(s)$ 
and hence 
$\fm(R)\in\NN$. 
\end{lemma}
\begin{proof}
As 
$dom_s(R)\ne\emptyset$,
let
$\ul{Y}\in dom_s(R)$ 
and according to Corollary 
\ref{cor:13Feb18a}, let
$\cR$
be a minimal nc Fornasini--Marchesini 
realization of 
$R$,
centred at
$\ul{Y}$.
Since
$R$ 
is non-degenerate,
there exists 
$n\in\NN$ 
such that 
$dom_k(R)\ne\emptyset$ 
for all
$k\ge n$.
Consider the sequence 
$(k_j)_{j\ge1}$
given by
$k_j=sj+1$, 
clearly for 
$j$ 
large enough we get 
$k_j\ge n$ 
and hence 
$dom_{k_j}(R)\ne\emptyset$.
Let
$\ul{W}\in 
dom_{k_j}(R)$
and apply Theorem 
\ref{thm:22Oct17c} 
for 
$\ul{W}$ and $\ul{Y}$; 
we obtain that 
$k_jL_R(\ul{Y})=sL_R(\ul{W})$, 
but it is easily seen that
$s$ 
and 
$k_j$ 
are co-prime integers, thus
$s\mid L_R(\ul{Y})$
 and hence  
 $\fm(R)\in\NN$.
\end{proof}
\begin{remark}
\label{rem:15Apr19a}
If
$R$ 
is a nc rational expression in 
$x_1,\ldots,x_d$
over
$\KK,\,\ul{Y}_1\in 
dom_{s_1}(R)$
and
$\ul{Y}_2\in 
dom_{s_2}(R)$,
then
$\ul{Y}_1\oplus \ul{Y}_2\in
dom_{s_1+s_2}(R)$ 
and 
(\ref{eq:19Feb19a}) 
implies that
$$L_{R}(s_1+s_2)=(s_1+s_2)\fm(R)
=s_1\fm(R)+s_2\fm(R)=L_R(s_1)+L_R(s_2).$$
If 
$R$ 
admits two minimal nc Fornasini--Marchesini realizations 
$\cR_1$ 
and 
$\cR_2$, 
centred at 
$\ul{Y}_1$
and
$\ul{Y}_2$, 
respectively, which are described by the tuples
$(L_1,D^1,C^1,\ul{\bA}^1,\ul{\bB}^1)$
and
$(L_2,D^2,C^2,\ul{\bA}^2,\ul{\bB}^2)$,
it is very tempting to consider
$\cR$ 
to be a nc Fornasini--Marchesini realization of 
$R$ 
centred at
$\ul{Y}_1\oplus\ul{Y}_2$, 
where 
$\cR$
is described by
\begin{align*}
D=D^1\oplus D^2,\,C=C^1\oplus C^2,\,
\bA_k=\begin{bmatrix}
\bA_k^1&0\\
0&\bA_k^2\\
\end{bmatrix}
\text{ and }
\bB_k=\begin{bmatrix}
\bB_k^1\\
\bB_k^2\\
\end{bmatrix},\,
1\le k\le d
\end{align*}
but we only know that
$R(\ul{X})=\cR(\ul{X})$
whenever 
$\ul{X}\in 
dom_{s_1+s_2}(R)$
is of the form
$\ul{X}=\ul{X}^{(1)}
\oplus \ul{X}^{(2)}$, 
with 
$\ul{X}^{(1)}\in dom_{s_1}(R_1)$
and
$\ul{X}^{(2)}\in dom_{s_2}(R_2)$.
\end{remark}
\section{Realizations of NC Rational Functions}
\label{sec:main}
From the previous section (cf. Theorem 
\ref{thm:2May16b}) 
we know that--- given a nc rational function--- all of its minimal nc Fornasini--Marchesini realizations which are centred at 
\textbf{the same point}, must have the same 
domain (and 
$\cA-$domain) and same evaluation (w.r.t 
$\cA$ as well; here $\cA$ 
is a unital stably finite
$\KK-$algebra). 

In this section we continue to establish 
connections between \textbf{all} minimal nc Fornasini--Marchesini realizations (with centres of all possible sizes)
of a given nc rational function. 
Using Lemma
\ref{thm:22Oct17c} 
and Remark 
\ref{rem:18Nov18b}, 
the general case--- where the two 
centres of minimal realizations of a 
rational function are different--- is 
considered and solved, which then will 
lead us to the main conclusion, that is Theorem
\ref{thm:30Jan18c}.
\begin{lemma}
\label{lem:30Jan18b}
Let 
$R_1$
and
$R_2$
be nc rational expressions in
$x_1,\ldots,x_d$
over
$\KK$,
with
$\ul{Y}_1\in dom_{s_1}(R_1)$
and
$\ul{Y}_2\in dom_{s_2}(R_2)$, where $s_1,s_2\in\NN$.
Suppose 
$\cR_1$ 
and 
$\cR_2$ 
are minimal nc Fornasini--Marchesini realizations of 
$R_1$
and
$R_2$, centred at 
$\ul{Y}_1$ 
and 
$\ul{Y}_2$, 
respectively. If 
$R_1$ 
and 
$R_2$ 
are 
$(\KK^d)_{nc}-$evaluation equivalent, then 
\begin{align}
\label{eq:26Mar19c}
DOM_{pm}(\cR_1)
=DOM_{pm}(\cR_2)
\text{ and }\cR_1(\ul{X})=
\cR_2(\ul{X})
\end{align}
for every 
$m\in\NN$ 
and 
$\ul{X}\in 
DOM_{pm}(\cR_1)$,
where 
$p=l.c.m(s_1,s_2)$.
Moreover, for any unital stably finite 
$\KK-$algebra 
$\cA$ 
and
$\ul{\fa}\in\cA^d$:
$$I_{s_1}\otimes\ul{\fa}
\in DOM^{\cA}(\cR_1) \iff
I_{s_2}\otimes\ul{\fa}
\in DOM^{\cA}(\cR_2)$$
and for every such 
$\ul{\fa}$, 
we have 
$$I_{s_2}\otimes
\cR_1^{\cA}(I_{s_1}\otimes\ul{\fa})
=I_{s_1}\otimes
\cR_2^{\cA}(I_{s_2}\otimes \ul{\fa}).$$
\end{lemma}
\begin{proof}
We know that 
$dom_{s_1}(R_1), 
dom_{s_2}
(R_2)\ne\emptyset$, 
hence
$dom_p(R_1),dom_p(R_2)\ne\emptyset$ 
and as they are both open Zariski sets in 
$(\KK^{\ptp})^d$, 
there exists 
$\ell\in\NN$ 
such that 
$dom_{\ell p}(R_1)\cap 
dom_{\ell p}(R_2)\ne\emptyset$,
so let us fix 
$$\wt{\ul{Y}}\in
dom_{\ell p}(R_1)\cap dom_{\ell p}(R_2)
\subseteq  
DOM_{\ell p}
(\cR_1)\cap 
DOM_{\ell p}(\cR_2).$$
Once again (as in the proof of Theorem \ref{thm:22Oct18a}), the reasoning for that is clear when $\KK$
is infinite (with $\ell=1$), while if $\KK$
is finite we use a similar argument as in
\cite[Remark 2.6]{KV3}. 

$\bullet$ Let
$n_1$
and
$n_2$ 
be the integers for which
$p\ell=s_1n_1=s_2n_2$.
From Remark 
\ref{rem:18Nov18b}, 
there exist minimal nc Fornasini--Marchesini realizations
$\cR^{(n_k)}$
of 
$R_k$,
centred at
$I_{n_k}\otimes \ul{Y}_k$,
such that
\begin{equation*}
DOM_{p\ell m}(\cR_k)=
DOM_{p\ell m}\big(\cR^{(n_k)}\big)
\text{ and }
\cR^{(n_k)}(\ul{X})=
\cR_k(\ul{X}),
\,\forall \ul{X}\in DOM_{p\ell m}(\cR_k)\\
\end{equation*}
for every $m\in\NN,\,
\ul{\fA}\in DOM^{\cA}(\cR_k)\iff
I_{n_k}\otimes\ul{\fA}\in DOM^{\cA}
\big(\cR^{(n_k)}\big)$
and
$$\big(\cR^{(n_k)}\big)^{\cA}(I_{n_k}\otimes\ul{\fA})
=I_{n_k}\otimes
\cR_k^{\cA}(\ul{\fA}),\,
\forall\ul{\fA}\in DOM^{\cA}(\cR_k),$$
for $k=1,2$.
In addition, there exist minimal nc Fornasini--Marchesini realizations 
$\wt{\cR_k}$ 
of
$R_k$,
centred at 
$\wt{\ul{Y}}$, 
such that
$$
DOM_{p\ell m}\big(\cR^{(n_k)}\big)
=DOM_{p\ell m}\big(\wt{\cR_k}\big)
\text{ and }\cR^{(n_k)}(\ul{X})
=\wt{\cR_k}(\ul{X}),\,
\forall \ul{X}\in 
DOM_{p\ell m}
\big(\wt{\cR_k}\big)$$
for every
$m\in\NN,\,
DOM^{\cA}\big(\cR^{(n_k)}\big)
=DOM^{\cA}\big(\wt{\cR_k}\big)$ and
$$\big(\cR^{(n_k)}\big)^{\cA}(\ul{\fA})=
\wt{\cR_k}^{\cA}(\ul{\fA}),
\,\forall
\ul{\fA}\in DOM^{\cA}
\big(\cR^{(n_k)}\big),$$
for
$k=1,2$.
Therefore
$\wt{\cR_1}$ 
and 
$\wt{\cR_2}$
are minimal nc Fornasini--Marchesini 
realizations, both centred at 
$\wt{\ul{Y}}$, 
of 
$R_1$ 
and 
$R_2$--- 
which are 
$(\KK^d)_{nc}-$evaluation equivalent---
hence Theorem
\ref{thm:2May16b} 
implies 
$$DOM_{p\ell m}\big(\wt{\cR_1}\big)
=DOM_{p\ell m}\big(\wt{\cR_2}\big)
\text{ and }\wt{\cR_1}(\ul{X})
=\wt{\cR_2}(\ul{X}),
\forall \ul{X}\in 
DOM_{p\ell m}\big(\wt{\cR_1}\big)$$
for every $m\in\NN$ and
\begin{align*}
DOM^{\cA}\big(\wt{\cR_1}\big)
=DOM^{\cA}\big(\wt{\cR_2}\big)
\text{ and }
\wt{\cR_1}^{\cA}(\ul{\fA})=
\wt{\cR_2}^{\cA}(\ul{\fA}),
\,\forall\ul{\fA}\in DOM^{\cA}
\big(\wt{\cR_1}\big),
\end{align*}
which yield that
\begin{align}
\label{eq:26Mar19a}
DOM_{p\ell m}(\cR_1)
=DOM_{p\ell m}\big(\wt{\cR_1}\big)
=DOM_{p\ell m}\big(\wt{\cR_2}\big)
=DOM_{p\ell m}(\cR_2)
\end{align}
for every 
$m\in\NN$
and 
\begin{align}
\label{eq:26Mar19b}
\cR_1(\ul{X})=\wt{\cR_1}(\ul{X})
=\wt{\cR_2}(\ul{X})=\cR_2(\ul{X}),\,
\forall \ul{X}\in DOM_{p\ell m}(\cR_1).
\end{align}
It is easily seen that
$\ul{X}\in DOM_{pm}(\cR_k)\iff 
I_{\ell}\otimes\ul{X}\in DOM_{p\ell m}(\cR_k)$
and in that case 
$\cR_k(I_{\ell}\otimes\ul{X})=I_{\ell}\otimes \cR_k(\ul{X})$,
where $k=1$
or
$k=2$ and thus, from
(\ref{eq:26Mar19a})
and
(\ref{eq:26Mar19b}) 
one can get
(\ref{eq:26Mar19c}).
 
$\bullet$ Moreover,
$$DOM^{\cA}\big(\cR^{(n_1)}\big)
=DOM^{\cA}\big(\wt{\cR_1}\big)
=DOM^{\cA}\big(\wt{\cR_2}\big)
=DOM^{\cA}\big(\cR^{(n_2)}\big)$$ 
implies that for every 
$\ul{\fa}\in\cA^d$,
\begin{multline*}
I_{s_1}\otimes\ul{\fa}
\in DOM^{\cA}(\cR_1)\iff 
I_{n_1}\otimes(I_{s_1}\otimes\ul{\fa})
\in DOM^{\cA}\big(\cR^{(n_1)}\big)
\iff
\\ I_{n_2}\otimes (I_{s_2}\otimes\ul{\fa})
\in DOM^{\cA}\big(\cR^{(n_2)}\big)
\iff I_{s_2}\otimes\ul{\fa}\in DOM^{\cA}(\cR_2)
\end{multline*}
and for every such 
$\ul{\fa}$:
\begin{multline*}
I_{n_1}\otimes
\cR_1^{\cA}(I_{s_1}\otimes\ul{\fa})
=\big(\cR^{(n_1)}\big)^{\cA}
(I_{p\ell}\otimes\ul{\fa})
=\wt{\cR_1}^{\cA}(I_{p\ell}
\otimes\ul{\fa})
\\=\wt{\cR_2}^{\cA}(I_{p\ell}
\otimes\ul{\fa})
=\big(\cR^{(n_2)}\big)^{\cA}
(I_{p\ell}\otimes\ul{\fa})=
I_{n_2}\otimes 
\cR_2^{\cA}(I_{s_2}\otimes\ul{\fa}),
\end{multline*}
which then, as $s_1n_1=s_2n_2$, implies  
$I_{s_2}\otimes
\cR_1^{\cA}(I_{s_1}\otimes\ul{\fa})
=I_{s_1}\otimes
\cR_2^{\cA}(I_{s_2}\otimes\ul{\fa})$.
\end{proof}
\begin{remark}
\label{rem:26Mar19a}
If 
$R_1$
and
$R_2$
are 
$(\KK^d)_{nc}-$evaluation equivalent nc (non-degenerate) rational expressions in 
$x_1,\ldots,x_d$
over
$\KK$, 
as explained in the beginning of 
the proof, there exists
$\wt{\ul{Y}}\in dom_{\ell p}
(R_1)\cap dom_{\ell p}(R_2)$
for some 
$\ell\in\NN$, 
thus Theorem
\ref{thm:2May16b} implies that 
$L_{R_1}(\ell p)=L_{R_2}(\ell p)$  
and hence
$$\fm(R_1)=\frac{L_{R_1}(\ell p)}{\ell p}
=\frac{L_{R_2}(\ell p)}{\ell p}=\fm(R_2).$$
Therefore, we define the
\textbf{McMillan degree of a nc rational function}
$\fR$
to be 
$\fm(\fR):=\fm(R)$ 
for every 
$R\in\fR$.
\end{remark}
Recall that a nc rational function 
$\fR$ 
is an equivalence class of the form
$$\fR=\{R: R \text{ 
is a non-degenerate representative of 
}\fR\}$$ 
whose elements are   
$(\KK^d)_{nc}-$evaluation equivalent nc rational expressions in 
$x_1,\ldots,x_d$ over
$\KK$, whereas the
domain and $\cA-$domain of regularity of 
$\fR$  
are given by
\begin{align*}
dom(\fR)=
\bigcup_{R\in\fR} dom(R)
\text{ and }
dom^{\cA}(\fR)=\bigcup_{R\in\fR} dom^{\cA}(R),
\end{align*} 
respectively.
We now use Corollary
\ref{cor:13Feb18a}
and Lemma
\ref{lem:30Jan18b}, 
to show that the domain of regularity of
a nc rational function 
$\fR$ 
at the level of 
$\ntn$
matrices, i.e.,
$$dom_n(\fR)=\bigcup_{R\in\fR} dom_n(R),$$
lives inside the domain of \textbf{any} minimal nc 
Fornasini--Marchesini realization of a representative in
$\fR$, 
up to a tensor product with the identity matrix.
\begin{theorem}
\label{thm:30Jan18c}
Let 
$\fR\in\KK\plangle x_1,\ldots,x_d\prangle$ 
be a nc rational function. For every 
nc rational expression
$R\in\fR$, 
integer
$s\in\NN$,
point 
$\ul{Y}\in dom_{s}(\fR)$,
minimal nc Fornasini--Marchesini realization 
$\cR$ 
centred at 
$\ul{Y}$ 
of 
$R$,
and unital stably finite 
$\KK-$algebra
$\cA$,
we have the following properties:
\begin{itemize}
\item[1.]
If 
$\ul{Z}\in dom_n(\fR)$, 
then
$I_s\otimes \ul{Z}\in DOM_{sn}(\cR)$
and
$I_s\otimes
\fR(\ul{Z})=\cR(I_s\otimes \ul{Z}).$
\item[2.]
If 
$s\mid n$,  
then
$dom_{n}(\fR)\subseteq DOM_{n}(\cR)$ 
and 
$\fR(\ul{Z})=\cR(\ul{Z})$ for every 
$\ul{Z}\in dom_{n}(\fR)$.
\item[3.]
If
$\ul{\fa}\in dom^{\cA}(\fR)$,
then 
$I_s\otimes\ul{\fa}\in DOM^{\cA}(\cR)$
and 
$I_s\otimes
\fR^{\cA}(\ul{\fa})
=\cR^{\cA}(I_s\otimes\ul{\fa})$.
\end{itemize}

\end{theorem}
\begin{proof}
Let 
$\ul{Z}\in dom_n(\fR)$, 
so there exists a nc rational expression 
$\wt{R}\in\fR$ 
such that 
$\ul{Z}\in 
dom_n(\wt{R})$, 
while
Corollary
\ref{cor:13Feb18a}
implies the existence of a minimal nc Fornasini--Marchesini realization 
$\wt{\cR}$ 
of
$\wt{R}$, 
centred at 
$\ul{Z}$. 
Then 
$\cR$ 
and 
$\wt{\cR}$
are minimal nc Fornasini--Marchesini realizations of 
$R$ 
and 
$\wt{R}$, 
with centres in 
$(\KK^{\sts})^d$ 
and 
$(\KK^{\ntn})^d$, 
respectively.
Since 
$R,\wt{R}\in\fR$, 
it follows that 
$R$ 
and
$\wt{R}$ 
are
$(\KK^d)_{nc}-$evaluation equivalent, therefore Lemma 
\ref{lem:30Jan18b} 
guarantees that 
\begin{align}
\label{eq:15Apr19c}
DOM_{pm}\big(\wt{\cR}\big)
=DOM_{pm}(\cR)
\text{ and }
\wt{\cR}(\ul{X})=\cR(\ul{X}),\,
\forall \ul{X}\in 
DOM_{pm}(\cR)
\end{align}
for every 
$m\in\NN$,
where 
$p=l.c.m(s,n)$. 
Thus, 
$p\mid sn$ 
implies that
\begin{align*}
\ul{Z}\in dom_n\big(\wt{R}\big)
\Longrightarrow 
I_s\otimes \ul{Z}\in dom_{sn}
\big(\wt{R}\big)
\subseteq DOM_{sn}\big(\wt{\cR}\big)
=DOM_{sn}(\cR)
\end{align*}
and hence 
$$I_s\otimes 
\fR(\ul{Z})=
I_s\otimes 
\wt{R}(\ul{Z})=
\wt{R}(I_s\otimes \ul{Z})=
\wt{\cR}(I_s\otimes \ul{Z})
=\cR(I_s\otimes \ul{Z}),$$
which ends the proof of part 
\textbf{1}.

$\bullet$ Suppose next that 
$s\mid n$; 
in that case we have  
$p=l.c.m(s,n)=n$. 
Thus, in view of 
(\ref{eq:15Apr19c}) 
with 
$m=1$,
if 
$\ul{Z}\in dom_n\big(\wt{R}\big)$,
then
$$
\ul{Z}\in
DOM_n\big(\wt{\cR}\big)=DOM_n(\cR) \text{ 
and } 
\fR(\ul{Z})=\wt{R}(\ul{Z})
=\wt{\cR}(\ul{Z})=\cR(\ul{Z}),$$
which ends the proof of part 
\textbf{2}.

$\bullet$ 
Finally, let 
$\ul{\fa}\in dom^{\cA}(\fR)$, 
then there exists a non-degenerate nc rational expression
$\wh{R}\in\fR$
such that 
$\ul{\fa}\in dom^{\cA}\big(\wh{R}\big)$. As
$\wh{R}$
is non-degenerate, there exists 
$t\in\NN$
such that 
$dom_t\big(\wh{R}\big)\ne\emptyset$.
Let
$\wh{\ul{Y}}\in dom_t(\wh{R})$ 
and apply Corollary
\ref{cor:13Feb18a}: so there exists 
a minimal realization
$\wh{\cR}$ 
of 
$\wh{R}$, centred at 
$\wh{\ul{Y}}$,
such that
$$\ul{\fa}\in dom^{\cA}\big(\wh{R}\big)
\Longrightarrow I_t\otimes\ul{\fa}\in DOM^{\cA}
\big(\wh{\cR}\big)
\text{ and }
\wh{\cR}^{\cA}(I_t\otimes\ul{\fa})
=I_t\otimes 
\wh{R}^{\cA}(\ul{\fa}).$$
As 
$\cR$ 
and
$\wh{\cR}$ 
are minimal nc Fornasini--Marchesini realizations of 
$R\in\fR$
and
$\wh{R}\in\fR$, 
respectively,
Lemma 
\ref{lem:30Jan18b} 
guarantees that
$$I_t\otimes\ul{\fa}\in DOM^{\cA}
\big(\wh{\cR}\big)
\Longrightarrow I_s\otimes\ul{\fa}\in DOM^{\cA}(\cR)$$ 
and also that
$$I_t\otimes
\cR^{\cA}(I_s\otimes\ul{\fa})
=I_s\otimes 
\wh{\cR}^{\cA}(I_t\otimes\ul{\fa})
=I_s\otimes 
\big(I_t\otimes 
\wh{R}^{\cA}(\ul{\fa})\big),$$
thus 
$\cR^{\cA}(I_s\otimes\ul{\fa})=
I_s\otimes
\wh{R}^{\cA}(\ul{\fa})
=I_s\otimes
\fR^{\cA}(\ul{\fa})$.
\end{proof}
What we proved is that
$$dom_{sm}(\fR)\subseteq 
DOM_{sm}(\cR),
\,\forall m\in\NN$$
however in the case where 
$s=1$ and 
$\ul{Y}=(0,\ldots,0)$,
the nc Fornasini-Marchesini realization 
$\cR$ 
is actually a 
$1\times 1$ 
matrix valued nc rational expression 
(not a priori possible
if 
$s>1$);
by viewing 
$\fR$
as a 
$1\times 1$ 
matrix valued nc rational function (cf. Remark 
\ref{rem:27Mar19a}), 
it follows that the nc Fornasini--Marchesini realization
$\cR$ 
is a representative of 
$\fR$ 
and therefore
$$dom_{m}(\fR)\supseteq 
DOM_{m}(\cR),
\,\forall m\in\NN.$$
In other words, by applying Theorem 
\ref{thm:30Jan18c}
we actually obtain a proof for Theorem 
\ref{thm:realization} from the introduction which--- 
unlike the
original proof in 
\cite{KV4}---
does not make any use of the difference-differential calculus of nc functions.
\begin{cor}
\label{cor:14Oct18a}
If 
$\fR$
is a  nc rational function of 
$x_1,\ldots,x_d$ 
over
$\KK$
and 
$\fR$ 
is regular at 
$\ul{0}$, 
then
$\fR$ 
admits a unique 
(up to unique similarity) minimal
(observable and controllable) nc Fornasini--Marchesini realization
$$\fR(x_1,\ldots,x_d)=D+C
\Big(I_L-
\sum_{k=1}^d A_kx_k\Big)^{-1}
\sum_{k=1}^d 
B_kx_k,$$
where 
$A_1,\ldots,A_d\in\KK^{L\times L} 
,\,B_1,\ldots,B_d\in\KK^{L\times 1},\, 
C\in\KK^{1\times L},\,
D=\fR(\ul{0})$ 
and 
$L\in\NN$,
$$dom_{m}(\fR)=\big\{
(X_1,\ldots,X_d)\in
(\KK^{\mtm})^d:
\det\left(I_{Lm}-X_1\otimes A_1-\ldots-
X_d\otimes A_d\right)\ne0\big\}$$
for every 
$m\in\NN$ and 
\begin{equation*}
\fR(X_1,\ldots,X_d)=
I_m\otimes D+(I_m\otimes C)
\Big( I_{mL}-\sum_{k=1}^d
X_k\otimes A_k\Big)^{-1}\sum_{k=1}^d
X_k\otimes B_k
\end{equation*}
for every 
$(X_1,\ldots,X_d)\in dom_m(\fR)$.
\end{cor}
\section{Realizations of Matrix Valued NC Rational Functions}
\label{sec:mvf}
All of the analysis and results up to now can be generalized to the settings
of matrix valued nc rational functions. In this section we describe the relevant definitions and main results in the matrix valued case.

If 
$\alpha,\beta\in\NN$, 
we say that 
$\fr$ 
is a 
$\alpha\times\beta$
\textbf{matrix valued nc rational function} if 
$\fr$
is a 
$\alpha\times\beta$ 
matrix of nc rational functions, i.e., if
$$\fr=
\begin{bmatrix}
\fR_{ij}
\end{bmatrix}_{1\le i\le\alpha,\, 
1\le j\le\beta}$$
where 
$\fR_{ij}$ are nc rational functions. 
The domain of regularity of 
$\fr$ 
is then defined by
\begin{align}
\label{eq:2Mar19b}
dom(\fr):=\bigcap _{1\le i\le\alpha,\,1\le j\le\beta}
dom(\fR_{ij})
\end{align}
and for every 
$\ul{X}\in dom_n(\fr)$ 
the evaluation 
$\fr(\ul{X})$ 
is given by
$$\fr(\ul{X}):=
E(n,\alpha)
\begin{bmatrix}
\fR_{ij}(\ul{X})
\end{bmatrix}_{1\le i\le\alpha,\,
1\le j\le \beta}
E(n,\beta)^T,$$
where
$E(\ell_1,\ell_2)\in
\KK^{\ell_1\ell_2\times\ell_1\ell_2}$
defined in 
(\ref{eq:24Jan19a}). 
The need for the correction terms, which are shuffle matrices, is coming simply because otherwise evaluating 
$\fr$ 
term by term does not yield a nc function (it does not preserve direct sums), see e.g. 
\cite[pp. 17--18]{KV3}.
If
$\cA$
is a unital $\KK-$algebra, then the $\cA-$domain of 
$\fr$
is defined by
\begin{align*}
dom^{\cA}(\fr):=
\bigcap_{1\le i\le\alpha,\,1\le j\le\beta}
dom^{\cA}(\fR_{ij})
\end{align*}
and for every
$\ul{\fa}\in dom^{\cA}(\fr)$
the evaluation $\fr^{\cA}(\ul{\fa})$
is given by
$$\fr^{\cA}(\ul{\fa}):=
\begin{bmatrix}
\fR_{ij}^{\cA}(\ul{\fa})\\
\end{bmatrix}_{1\le i\le\alpha,\,
1\le j\le\beta}.$$
\begin{remark}
\label{rem:27Mar19a}
In 
\cite{KV4}
the authors define matrix valued nc rational functions and their domains of regularity using equivalence classes of matrix valued nc rational expressions.  
However, it follows from
\cite[Lemma 3.9]{V1}
that  one can also define matrix valued nc rational functions
as a matrix of nc rational functions, and the domains of regularity in both cases are equal.
\end{remark}
\begin{theorem}
\label{thm:7Nov18a}
For every 
$\alpha\times\beta$
matrix valued nc rational function 
$$\fr=
\begin{bmatrix}
\fR_{ij}
\end{bmatrix}_{1\le i\le\alpha,\, 
1\le j\le\beta}
\text{ and } 
\ul{Y}=(Y_1,\ldots,Y_d)\in dom_s(\fr),$$
there exist unique (up to unique similarity)
$L\in\NN$,\,
$D\in\KK^{\alpha s\times\beta s},\,
C\in\KK^{\alpha s\times L},$
linear mappings
$\bA_1,\ldots,\bA_d:\KK^{\sts}\rightarrow
\KK^{L\times L}$
and
$\bB_1,\ldots,\bB_d:
\KK^{\sts}\rightarrow
\KK^{L\times \beta s}$ 
such that
$(\ul{\bA},\ul{\bB})$
is controllable and
$(C,\ul{\bA})$ 
is observable,
for which
\begin{align}
\label{eq:2Mar19a}
dom_{sm}(\fr)\subseteq \Big\{
\ul{X}\in(\KK^{sm\times sm})^d: 
\det\Big(I_{Lm}-\sum_{k=1}^d
(X_k-I_m\otimes Y_k)\bA_k
\Big)\ne0\Big\}
\end{align}
for every $m\in\NN$
and 
$\fr(\ul{X})=\cR(\ul{X})$ 
for every
$\ul{X}\in dom_{sm}(\fr)$, 
where
\begin{align*}
\cR(\ul{X})=I_m\otimes D+(I_m\otimes C)
\Big(I_{Lm}-\sum_{k=1}^d
(X_k-I_m\otimes Y_k)\bA_k\Big)^{-1}
\sum_{k=1}^d
(X_k-I_m\otimes Y_k)\bB_k.
\end{align*} 
Moreover,
$$dom_n(\fr)\subseteq 
\Big\{
\ul{X}\in(\KK^{n\times n})^d: 
\det\Big(I_{nL}-\sum_{k=1}^d
(I_s\otimes X_k-I_n\otimes Y_k)\bA_k
\Big)\ne0\Big\}$$
for every
$n\in\NN$
and 
$I_s\otimes\fr(\ul{X})=\cR(I_s\otimes\ul{X})$ 
for every
$\ul{X}\in dom_n(\fr)$.
Furthermore,
for any unital stably finite 
$\KK-$algebra $\cA$:
\begin{align*}
dom^{\cA}(\fr)\subseteq 
\Big\{\ul{\fa}\in\cA^d:
\Big(I_L\otimes 1_{\cA}-
\sum_{k=1}^d (I_s\otimes \fa_k-Y_k\otimes 1_{\cA})\bA_k^{\cA}\Big)
\text{ is invertible in }\cA^{L\times L}
\Big\}
\end{align*}
and 
$I_s\otimes\fr^{\cA}(\ul{\fa})
=\cR^{\cA}(I_s\otimes\ul{\fa})$ 
for every 
$\ul{\fa}\in dom^{\cA}(\fr)$, i.e.,
$$I_s\otimes 
\fr^{\cA}(\ul{\fa})=D\otimes 1_{\cA}+(C\otimes 1_{\cA})
\Big( I_L\otimes 1_{\cA}-\sum_{k=1}^d
(I_s\otimes\fa_k-Y_k\otimes 1_{\cA})\bA_k^{\cA} 
\Big)^{-1}
\sum_{k=1}^d
(I_s\otimes \fa_k-Y_k\otimes 1_{\cA})\bB_k^{\cA}.$$
\end{theorem}
Similarity of nc Fornasini--Marchesini realizations of matrix valued nc rational functions is defined analogously to the case of (scalar) nc rational functions (cf. Theorem
\ref{thm:2May16b}),
as well as controllability and observability, which are defined via the controllability and un-observability subspaces of  
$\KK^L$ (cf. Definition \ref{def:25Jan19a} and notice
that the only difference is that the operators 
$\ul{\bA}^\omega\cdot \bB_k$ and
$C\cdot\ul{\bA}^{\omega}$
return matrices in 
$\KK^{L\times \beta s}$ and
$\KK^{\alpha s\times L}$, respectively).
\begin{proof}
Suppose 
$\fr=
\begin{bmatrix}
\fR_{ij}\\
\end{bmatrix}_{1\le i\le\alpha,\,1\le j\le\beta}$
is a 
$\alpha\times\beta$ 
matrix valued nc rational function and let
$\ul{Y}\in dom_{s}(\fr)$. 
For every 
$1\le i\le\alpha$ and
$\,1\le j\le\beta$ we have $
\ul{Y}\in dom_s(\fR_{ij})$,
while 
 applying Theorem 
\ref{thm:30Jan18c},
the nc rational function
$\fR_{ij}$ 
admits a minimal nc Fornasini--Marchesini realization  
$\cR_{ij}$ 
centred at
$\ul{Y}$, 
described by a tuple
$(L_{ij},D^{ij},C^{ij},
\ul{\bA}^{ij},\ul{\bB}^{ij}),$
which is also a realization of 
$\fR_{ij}$
w.r.t 
$\cA$.

$\bullet$ Let 
$\ul{X}\in dom_{sm}(\fr)$, 
then 
$\ul{X}\in dom_{sm}(\fR_{ij})$
and hence
$\ul{X}\in DOM_{sm}(\cR_{ij})$
and 
$\fR_{ij}(\ul{X})=\cR(\ul{X})$
for any
$1\le i\le\alpha$ 
and 
$1\le j\le\beta$.
Therefore
\begin{multline*}
\fr(\ul{X})=E(sm,\alpha)
\begin{bmatrix}
\fR_{11}&\ldots&\fR_{1\beta}\\
\vdots&&\vdots\\
\fR_{\alpha1}&\ldots&\fR_{\alpha\beta}\\
\end{bmatrix}(\ul{X})E(sm,\beta)^T
=E(sm,\alpha)
\begin{bmatrix}
\cR_{11}&\ldots&\cR_{1\beta}\\
\vdots&&\vdots\\
\cR_{\alpha1}&\ldots&\cR_{\alpha\beta}\\
\end{bmatrix}(\ul{X})
\\E(sm,\beta)^T
=I_m\otimes D+(I_m\otimes C)
\Big(I_{Lm}-\sum_{k=1}^d
(X_k-I_m\otimes Y_k)\bA_k\Big)^{-1}
\sum_{k=1}^d (X_k-I_m\otimes Y_k)\bB_k
=\cR(\ul{X}),
\end{multline*}
where the nc Fornasini--Marchesini realization
$\cR$ 
is described by
\begin{multline*}
L=\sum_{i=1}^\alpha\sum_{j=1}^\beta 
L_{ij},\, D=E(s,\alpha)
\begin{bmatrix}
D^{11}&\ldots&D^{1\beta}\\
\vdots&&\vdots\\
D^{\alpha1}&\ldots&D^{\alpha\beta}\\
\end{bmatrix}
E(s,\beta)^T
\in\KK^{\alpha s\times\beta s},\\
C=E(s,\alpha)
\begin{bmatrix}
C^{11}&\ldots&C^{1\beta}&0&\ldots&\ldots&\ldots&\ldots&0\\
0&\ldots&0&C^{21}&\ldots&C^{2\beta}&0&\ldots&0\\
\vdots&&\vdots&\ddots&\ddots&\ddots&\vdots&&\vdots\\
0&\ldots&\ldots&\ldots&\ldots&0&C^{\alpha1}&\ldots&C^{\alpha\beta}\\
\end{bmatrix}\in\KK^{\alpha s\times L},
\end{multline*}
with the linear mappings
\begin{align*}
\bA_k=
\begin{bmatrix}
\bA_k^{11}&\ldots&0&\ldots&\ldots&\ldots&0\\
\vdots&\ddots&&&&&\vdots\\
0&&\bA_k^{1\beta}&&&&\vdots\\
\vdots&&&\ddots&&&\vdots\\
\vdots&&&&\bA_k^{\alpha1}&&0\\
\vdots&&&&&\ddots&\vdots\\
0&\ldots&\ldots&\ldots&0&\ldots&\bA_k^{\alpha\beta}\\
\end{bmatrix},\, 
\bB_k=\begin{bmatrix}
\bB_k^{11}&\ldots&0\\
\vdots&\ddots&\vdots\\
0&\ldots&\bB_k^{1\beta}\\
\bB_k^{21}&\ldots&0\\
\vdots&\ddots&\vdots\\
0&\ldots&\bB_k^{2\beta}\\
\vdots&\vdots&\vdots\\
\vdots&\vdots&\vdots\\
\bB_k^{\alpha1}&\ldots&0\\
\vdots&\ddots&\vdots\\
0&\ldots&\bB_k^{\alpha\beta}\\
\end{bmatrix}E(s,\beta)^T.
\end{align*}
It is easily seen, by the diagonal structure of 
$\bA_k$, that
$$dom_{sm}(\fr)\subseteq
\bigcap_{1\le i\le\alpha,\,1\le j\le\beta} 
DOM_{sm}(\cR_{ij})=DOM_{sm}(\cR).$$

$\bullet$
Moreover, 
if
$\ul{X}\in dom_n(\fr)$, 
then 
$I_s\otimes\ul{X}\in dom_{sn}(\fr)$
and $
\fr(I_s\otimes\ul{X})=
I_s\otimes \fr(\ul{X})$, 
while by the first part of the theorem we know that
$I_s\otimes \ul{X}\in dom_{sn}(\cR)$ 
and
$$ I_s\otimes
\fr(\ul{X})
=\fr(I_s\otimes\ul{X})
=\cR(I_s\otimes\ul{X}).$$

$\bullet$ Furthermore, let
$\ul{\fa}\in dom^{\cA}(\fr)$, 
thus 
$\ul{\fa}\in dom^{\cA}(\fR_{ij})$
and therefore 
$I_s\otimes\ul{\fa}\in DOM^{\cA}(\cR_{ij})$ 
and
$\cR_{ij}^{\cA}
(I_s\otimes\ul{\fa})
=I_s\otimes 
\fR_{ij}^{\cA}(\ul{\fa})$,
for every
$1\le i\le\alpha$ and 
$1\le j\le\beta$. 
A direct and careful computation--- which is omitted--- shows that
\begin{multline*}
\cR^{\cA}(I_s\otimes\ul{\fa})=
\left(E(s,\alpha)\otimes 1_{\cA}\right)
\begin{bmatrix}
\cR_{ij}^{\cA}(I_s\otimes\ul{\fa})\\
\end{bmatrix}_{1\le 
i\le\alpha,\,1\le j\le\beta}
\left(E(s,\beta)\otimes 1_{\cA}\right)
\\=\left(E(s,\alpha)\otimes 1_{\cA}\right)
\begin{bmatrix}
I_s\otimes 
\fR_{ij}^{\cA}(\ul{\fa})\\
\end{bmatrix}_{1\le i\le\alpha,1\le j\le\beta}
\left(E(s,\beta)\otimes 1_{\cA}\right)
=I_s\otimes \fr^{\cA}(\ul{\fa}).
\end{multline*}

$\bullet$ 
This proves the existence of a nc Fornasini--Marchesini realization for 
$\fr$, 
centred at
$\ul{Y}$, 
that is also a realization of  
$\fr$ 
w.r.t  
$\cA$.
To obtain a minimal nc Fornasini--Marchesini realization, we use the Kalman decomposition same as in Lemma 
\ref{thm:2May16b}, corresponding to the
controllability and un-observability subspaces of  
$\KK^L$, whereas
the uniqueness (up to unique similarity) of such a minimal realization is proved with the same ingredients as in the proof of Theorem 
\ref{thm:2May16b}. The details of the proofs of the Kalman decomposition and the uniqueness are omitted.
\end{proof}
\begin{remark}
\label{rem:16Apr19a}
It is not hard to see that the nc Fornasini--Marchesini realization built in the proof is not necessarily minimal, even if 
$\cR_{ij}$ 
are all minimal nc Fornasini--Marchesini realizations. However, the opposite is true, i.e., if 
$\cR$ 
is a minimal nc Fornasini--Marchesini realization, then all of the nc Fornasini--Marchesini realizations 
$\cR_{ij}$ 
must be minimal as well.
\end{remark}
\begin{remark}
\label{rem:20Mar19b} 
(cf. Remark 
\ref{rem:27Mar19a})
The proof of the existence part in Theorem
\ref{thm:7Nov18a} can be done using matrix valued nc rational expressions and the usual process of synthesis, yielding
(\ref{eq:2Mar19a}) 
for the a priori bigger domain of 
$\fr$, 
which uses matrix valued nc rational expressions and once again by
\cite[Lemma 3.9]{V1} 
is equal to the domain in the sense of 
(\ref{eq:2Mar19b}).
\end{remark}
\begin{remark}[McMillan degree of a matrix valued nc rational function]
\label{rem:27Mar19b}
One can prove analogous versions of Theorem 
\ref{thm:22Oct17c}
and Remark
\ref{rem:26Mar19a},
for matrix valued nc rational expressions and functions, thereby there exists an integer $\fm(\fr)$ 
such that for any 
$\ul{Y}\in dom_s(\fr)$, 
we have 
$$L_{\fr}(\ul{Y})=
s\cdot\fm(\fr);$$ here 
$L_{\fr}(\ul{Y})$ 
is the dimension of a minimal nc Fornasini--Marchesini realization of $\fr$, centred at
$\ul{Y}$.
We call
$\fm(\fr)$ 
the 
\textbf{McMillan degree of }$\fr$.
It follows from Theorem 
\ref{thm:7Nov18a} that
$$\fm(\fr)=\frac{L_{\fr}(\ul{Y})}{s}\le \frac{\sum_{i=1}^\alpha \sum_{j=1}^\beta
L_{ij}}{s}=\sum_{i=1}^\alpha
\sum_{j=1}^\beta
\fm(\fR_{ij}).$$
\end{remark}
\section{Realizations of Hermitian NC Rational Functions}
\label{sec:hermitian}
In the case where 
$\KK=\RR$ or 
$\KK=\CC$, 
one often considers symmetric or hermitian nc rational expressions, 
specially with applications to free probability
\cite{BMS,HMS,Sp3} 
and in optimization theory
\cite{AK,HKMV,HM,HMcCV}.
Unlike the case of descriptor realizations (see 
\cite{HMS,KV1}), 
the expression for
$\cR_{\cF\cM}(\ul{X}^*)^*$ 
does not have the form  of a nc Fornasini--Marchesini realization, for a nc Fornasini--Marchesini realization $\cR_{\cF\cM}$. Nevertheless, we can use our methods to obtain an analogue of  Corollary 
\ref{cor:13Feb18a}
in the case where the function 
$\fR$
is \textbf{hermitian}, i.e., when
\begin{align*}
\fR^*(\ul{X}):=\fR(\ul{X}^*)^*=\fR(\ul{X})
\text{ for all }\ul{X}\in dom(\fR),
\end{align*}
with the matrix pencil to be  inverted having 
\textbf{hermitian coefficients}. 
We also get explicit (necessary and sufficient) 
conditions on the coefficients of the 
realization for the nc rational function to be hermitian.
\begin{remark}
One can define hermitian nc rational 
functions more precisely. First, 
one needs to define--- using synthesis--- a nc 
rational expression 
$R^*$, 
for any nc rational expression 
$R$. Then, one can 
show that 
$R_1^*\sim R_2^*$, 
whenever 
$R_1\sim R_2$ 
(i.e., whenever 
$R_1$ 
and
$R_2$ 
are
$(\KK^d)_{nc}-$evaluation equivalent). 
Finally, for every nc rational function 
$\fR$, 
let 
$\fR^*=\{R^*:R\in\fR\}$
and define 
$\fR$ 
to be hermitian if 
$\fR^*=\fR$, 
as equivalence classes.
\end{remark}

We use the following notions: if
$\bT$ 
is a linear mapping on matrices, then 
$\bT^*$ 
is the linear mapping given by
$\bT^*(X):=
\bT(X^*)^*$ and 
$\bT$ is called hermitian if 
$\bT^*=\bT$.
If 
$J$
is a square matrix 
of the form
$$J=
\begin{bmatrix}
I_p&0&0\\
0&-I_q&0\\
0&0&0_t\\
\end{bmatrix},
\text{ with }p,q,t\ge0,$$ 
then we say that 
$J$
is a \textbf{semi-signature} matrix; notice that if 
$t=0$, then 
$J$
is a signature matrix.
\begin{theorem}
\label{thm:22Feb19a}
Let
$\fR$
be an hermitian nc rational function
of
$x_1,\ldots,x_d$
over
$\KK$,
$\ul{Y}^*=\ul{Y}\in dom_s(\fR)$ 
and 
\begin{equation*}
\cR_{\cF\cM}(\ul{X})=D+C
\Big( I_{L}-\sum_{k=1}^d
\bA_k(X_k-Y_k)\Big)^{-1}
\sum_{k=1}^d
\bB_k(X_k-Y_k)
\end{equation*}
be a minimal nc Fornasini--Marchesini
realization 
of $\fR$, 
centred at
$\ul{Y}$.  
\begin{enumerate}
\item[1.]
There exists a unique
$S=S^*\in\KK^{L\times L}$
such that
\begin{align}
\label{eq:8Dec18d}
D^*=D,\,
\bA_k^*\cdot S=S\cdot\bA_k,\,
\bB_k^*\cdot S=C\cdot\bA_k
\text{ and }C\cdot\bB_k=(C\cdot\bB_k)^*,
\, 1\le k\le d.
\end{align}
\item[2.]
Once the relations in 
(\ref{eq:8Dec18d}) 
hold, we have
\begin{align}
\label{eq:8Dec18e}
\ker(S)=\bigcap_{1\le k\le d,
\, X\in\KK^{\sts}} 
\ker \big(\bA_k(X)\big),\, 
\Ima(S)=\bigvee_{1\le k\le d,
\,X\in\KK^{\sts}}\Ima
\left(\bA_k^*(X)\right)
\end{align}
and
$$\ker(C^*)=\bigcap_{1\le k\le d,\,
X\in\KK^{\sts}} \ker
\big(\bB_k(X)\big).$$
\item[3.] \textbf{Symmetry of the minimal nc Fornasini--Marchesini realization:}
There exist 
$\check{C}\in\KK^{s\times L}$,
a semi-signature matrix 
$J\in\KK^{L\times L}$
and hermitian linear mappings
$\check{\bA_1},\ldots,\check{\bA_d}:
\KK^{\sts}
\rightarrow\KK^{L\times L}$,
such that
\begin{align}
\label{eq:1Feb19a}
\cR_{\cF\cM}(\ul{X})=D+\check{C}
\Big( I_L-\sum_{k=1}^d
\check{\bA_k}(X_k-Y_k)J
\Big)^{-1}\sum_{k=1}^d 
\check{\bA_k}(X_k-Y_k)\check{C}^*,
\end{align}
with
$(\ul{\check{\bA}}\cdot J,
\ul{\check{\bA}}\cdot\check{C}^*)$ 
controllable and
$(\check{C},\ul{\check{\bA}}\cdot J)$
observable. 
Conversely, if 
$\fR$ 
admits a realization of the form 
(\ref{eq:1Feb19a}) with the controllability and observability conditions, then $\fR$ is an hermitian nc rational function.\\
\item[4.]
\textbf{Hermitian nc descriptor realization:}
There exist 
$D_{\cD}=D_{\cD}^*\in\KK^{\sts},\,
C_{\cD}\in\KK^{s\times (L+s)}$,
a signature matrix 
$J_{\cD}\in\KK^{(L+s)\times(L+s)}$
and hermitian linear mappings
$\bA_{1,\cD},\ldots,\bA_{d,\cD}:\KK^{\sts}\rightarrow\KK^{(L+s)\times(L+s)}$, 
such that
\begin{align}
\label{eq:1Feb19b}
\cR_{\cF\cM}(\ul{X})
=\cR_{\cD}(\ul{X}):=D_{\cD}+
C_{\cD}
\Big(J_{\cD}-\sum_{k=1}^d
\bA_{k,\cD}(X_k-Y_k)
\Big)^{-1}C_{\cD}^*
\end{align}
and 
\begin{multline*}
DOM_{s}(\cR_{\cF\cM})=DOM_{s}(\cR_{\cD})
:=\Big\{\ul{X}\in(\KK^{s\times s})^d:
\det\Big(J_{\cD}-\sum_{k=1}^d
\bA_{k,\cD}(X_k-Y_k)\Big)\ne0 \Big\}.
\end{multline*}
Moreover, similarly to Theorem 
\ref{thm:30Jan18c}, this also applies--- after a suitable tensoring---
to domains and evaluations on 
$\ntn$ 
matrices for all 
$n\in\NN$ 
and w.r.t any unital stably finite
$\KK-$algebra 
$\cA$. 
\item[5.]
The matrix $S$ is invertible
$\iff \bigcap_{1\le k\le d,\, X\in\KK^{\sts}} 
\ker \left(\bA_k(X)\right)=\{\ul{0}\}
\iff$ there exists 
$Q\in\KK^{L\times s}$ such that 
\begin{align}
\label{eq:9Jan19a}
\bB_k=\bA_k\cdot Q,
\,\forall 1\le k\le d.
\end{align} 
In that case,   
$J\in\KK^{L\times L}$
is invertible,
\begin{align}
\label{eq:4Jan19a}
\cR_{\cF\cM}(\ul{X})=D+\check{C}J
\Big(J-\sum_{k=1}^d \check{\bA_k}(X_k-Y_k)\Big)^{-1}
\sum_{k=1}^d \check{\bA_k}(X_k-Y_k)\check{C}^*, 
\end{align}
and
\begin{align}
\label{eq:1Feb19c}
\cR_{\cF\cM}(\ul{X})=
\wt{D}+\check{C}
\Big(J-\sum_{k=1}^d \wt{\bA_k}
(X_k-Y_k) \Big)^{-1}\check{C}^*,
\end{align}
where
$\wt{\bA_k}=J\check{\bA_k}J$
and
$\wt{D}=D-\check{C}J\check{C}^*$.
\end{enumerate}
\end{theorem}
\begin{proof}
1. The proof of the first part of the theorem follows the same ideas as the proof  of  Theorem 
\ref{thm:2May16b}.
As
$\fR^*=\fR$ we obtain that 
$\cR_{\cF\cM}^*=
\cR_{\cF\cM}$ 
and then compare the coefficients in 
the Taylor--Taylor power series expansions
$$\cR_{\cF\cM}(\ul{X})=\sum_{\nu\in\cG_d}
(\ul{X}-I_m\otimes\ul{Y})^{\odot_s\nu}\cR_{\nu}$$
and
$$\cR_{\cF\cM}(\ul{X})^*=
\sum_{\nu\in\cG_d}(\ul{X}-
I_m\otimes\ul{Y})^{\odot_s\nu}
\cR_{\nu,*},$$
we obtain that 
$\cR_\nu(Z_1,\ldots,Z_\ell)=\cR_{\nu,*}(Z_1,\ldots,Z_\ell)$ 
for every   
$\nu=g_{i_1}\ldots g_{i_\ell}\in\cG_d$
and
$Z_1,\ldots,Z_\ell\in\KK^{\sts}$, where
\begin{align*}
\cR_{\nu}(Z_1,\ldots,Z_\ell)=
\begin{cases}
D&: 
\text{ if }\ell=0 \\
C\bB_{i_1}(Z_1) &: \text{ if }\ell=1\\
C\bA_{i_1}(Z_1)\cdots\bA_{i_{\ell-1}}(Z_{\ell-1}) \bB_{i_\ell}(Z_\ell)&:\text{ if }\ell>1  \\
\end{cases}
\end{align*}
and
\begin{align*}
\cR_{\nu,*}(Z_1,\ldots,Z_\ell)=
\begin{cases}
D^*&:\text{ if }\ell=0\\
\bB_{i_1}^*(Z_1)C^*&:\text{ if }\ell=1\\
\bB_{i_1}^*(Z_1)
\bA_{i_2}^*(Z_2)
\cdots 
\bA_{i_\ell}^*(Z_\ell)C^*&: 
\text{ if }\ell>1.\\
\end{cases}
\end{align*}
For
$\nu=\emptyset$ we get that 
$D=D^*.$ 
Next, define a linear mapping 
\begin{align}
\label{eq:8Dec18a}
\cS\big(\ul{\bA}^\omega(W_1,\ldots,W_k)\bB_j(Z)\ul{u}
\big)=\left(\ul{\bA}^*\right)^{\omega}(W_1,\ldots,W_k)\bA_j^*(Z)C^*\ul{u}
\end{align}
for every
$\omega\in\cG_d$ 
of length 
$k\ge0,\,
1\le j\le d,\,
W_1,\ldots,W_k,\,
Z_j\in\KK^{\sts}$
and
$\ul{u}\in\KK^s$, 
then we extend 
$\cS$ 
by linearity.

$\bullet$ It is easily seen, from the controllability of
$(\ul{\bA},\ul{\bB})$ 
that 
$\cS:\KK^L\rightarrow\KK^L$ 
is well defined:
suppose
$$\ul{w}_1=\sum_{l\in I_L} 
\ul{\bA}^{\omega_l}
(\ul{W})
\bB_{j_l}(Z_l)\ul{u}_l=
\sum_{t\in I_T} 
\ul{\bA}^{\eta_t}
(\ul{Q})\bB_{i_t}(P_t)\ul{v}_t
=\ul{w}_2,$$
for
$\omega_l,\eta_t\in\cG_d,\,
\ul{W}= 
(W_1^{(l)},
\ldots,W_{k_l}^{(l)})\in (\KK^{\sts})^{k_l},\,
\ul{Q}=(Q_1^{(t)},
\ldots,Q^{(t)}_{m_t})\in(\KK^{\sts})^{m_t}
,\,Z_l,P_t\in\KK^{\sts},\,
1\le j_l,i_t\le d$
and
$\ul{u}_l,\ul{v}_t\in\KK^s$
for every
$l\in I_L$
and
$t\in I_T$.
Thus for every 
$1\le n\le d,\,\alpha\in\cG_d,\,
\wt{X}\in\KK^{\sts}$
and
$\ul{X}=(X_1,\ldots,
X_{|\alpha|})\in(\KK^{\sts})^{|\alpha|}$:
\begin{multline*}
\sum_{l\in I_L}\cR_{g_n\alpha\omega_lg_{j_l},*}
\big(\wt{X},\ul{X},\ul{W},Z_l\big)\ul{u}_l
=\sum_{l\in I_L}\cR_{g_n\alpha\omega_lg_{j_l}}
\big(\wt{X},\ul{X},\ul{W},Z_l\big)\ul{u}_l
=
C\bA_n(\wt{X})\ul{\bA}^{\alpha}
(\ul{X})\ul{w}_1\\=
C\bA_n(\wt{X})\ul{\bA}^{\alpha}
(\ul{X})\ul{w}_2=
\sum_{t\in I_T}
\cR_{g_n\alpha\eta_tg_{i_t}}
\big(\wt{X},\ul{X},\ul{Q},P_t
\big)\ul{v}_t
=\sum_{t\in I_T}\cR_{g_n\alpha
\eta_tg_{i_t},*}
\big(\wt{X},\ul{X},\ul{Q},
P_t\big)\ul{v}_t,
\end{multline*}
which implies that 
\begin{multline*}
\bB_n^*(\wt{X})(\ul{\bA}^*)^\alpha
(\ul{X})
\sum_{l\in I_L} 
(\ul{\bA}^*)^{\omega_l}
\big(\ul{W}\big)
\bA^*_{j_l}(Z_l)C^*\ul{u}_l
=\bB_n^*(\wt{X})(\ul{\bA}^*)^\alpha
(\ul{X})
\sum_{t\in I_T} 
(\ul{\bA}^*)^{\eta_t}
\big(\ul{Q}\big)
\bA^*_{i_t}(P_t)C^*\ul{v}_t,
\end{multline*}
whereas the controllability of 
$(\ul{\bA},\ul{\bB})$
implies that 
$\cS(\ul{w}_1)=\cS(\ul{w}_2)$.

$\bullet$ Let 
$S\in\KK^{L\times L}$ be the matrix such that 
$\cS(\ul{u})=S\ul{u}$ for every $\ul{u}\in\KK^L$.
We show that 
$S$
is self-adjoint:
for every 
$\ul{u},\ul{v}\in\KK^L$ we use the controllability of
$(\ul{\bA},\ul{\bB})$ to write them as
\begin{align*}
\ul{u}=\sum_{l\in I_L} 
\ul{\bA}^{\omega_l}
\big(\ul{W}\big)
\bB_{j_l}(Z_l)\ul{u}_l\text{ and }
\ul{v}=
\sum_{t\in I_T} 
\ul{\bA}^{\eta_t}
\big(\ul{Q}\big)\bB_{i_t}(P_t)\ul{v}_t,
\end{align*}
thus, using the notation 
$\ul{Q}^\#:=\big(Q_{m_t}^{(t)},\ldots,
Q_1^{(t)}\big),$
\begin{multline*}
\langle S\ul{u},\ul{v}\rangle=
\sum_{l\in I_L}\sum_{t\in I_T}
\left\langle S\ul{\bA}^{\omega_l}
\big(\ul{W}\big)
\bB_{j_l}(Z_l)\ul{u}_l,
\ul{\bA}^{\eta_t}
\big(\ul{Q}\big)\bB_{i_t}(P_t)\ul{v}_t\right\rangle
=\sum_{l\in I_L}\sum_{t\in I_T}
\langle (\ul{\bA}^*)^{\omega_l}
\big(\ul{W}\big)
\bA_{j_l}^*(Z_l)\\C^*\ul{u}_l,
\ul{\bA}^{\eta_t}
\big(\ul{Q}\big)\bB_{i_t}(P_t)\ul{v}_t
\rangle
=\sum_{l\in I_L}\sum_{t\in I_T}
\ul{v}_t^*\bB_{i_t}^*(P_t)(\ul{\bA}^*)^{\eta_t^T}
\big(\ul{Q}^\#\big)
(\ul{\bA}^*)^{\omega_l}
\big(\ul{W}\big)
\bA_{j_l}^*(Z_l)
C^*\ul{u}_l
\\=\sum_{l\in I_L}\sum_{t\in I_T} 
\ul{v}_t^* 
\cR_{g_{i_t}\eta_t^T\omega_l g_{j_l},*}
\big(P_t,\ul{Q}^\#,\ul{W},Z_l\big) \ul{u}_l
=\sum_{l\in I_L}\sum_{t\in I_T} 
\ul{v}_t^* \cR_{g_{i_t}\eta_t^T\omega_l g_{j_l}}
\big(P_t,\ul{Q}^\#,\ul{W},Z_l\big) \ul{u}_l
\\=\sum_{l\in I_L}\sum_{t\in I_T} 
\ul{v}_t^* C\bA_{i_t}(P_t)
\ul{\bA}^{\eta_t^T}
\big(\ul{Q}^\#\big)
\ul{\bA}^{\omega_l}
\big(\ul{W}
\big)\bB_{j_l}(Z_l)\ul{u}_l
=\sum_{l\in I_L}\sum_{t\in I_T} 
\langle 
\ul{\bA}^{\omega_l}\big(\ul{W}
\big)\bB_{j_l}(Z_l)\ul{u}_l,
(\ul{\bA}^*)^{\eta_t}
\\\big(\ul{Q}
\big)\bA_{i_t}^*(P_t)C^*\ul{v}_t
\rangle
=\sum_{l\in I_L}\sum_{t\in I_T} 
\langle \ul{\bA}^{\omega_l}\big(\ul{W}
\big)\bB_{j_l}(Z_l)\ul{u}_l,S\ul{\bA}^{\eta_t}
\big(\ul{Q}\big)\bB_{i_t}(P_t)\ul{v}_t
\rangle=\langle \ul{u},S\ul{v}\rangle,
\end{multline*}
i.e., $S=S^*$. 

$\bullet$ From 
(\ref{eq:8Dec18a}) it follows that 
\begin{align}
\label{eq:8Dec18b}
S\bB_k(Z)=\bA_k^*(Z)C^*,\,
\forall1\le k\le d,\,Z\in\KK^{\sts}
\end{align}
and for every 
$\ul{w}_1=\sum_{l\in I_L} 
\ul{\bA}^{\omega_l}
\big(\ul{W}\big)
\bB_{j_l}(Z_l)\ul{u}_l\in\KK^L$:
\begin{multline*}
S\bA_k(Z)\ul{w}_1=\sum_{l\in I_L} S\bA_k(Z)
\ul{\bA}^{\omega_l}
\big(\ul{W}\big)
\bB_{j_l}(Z_l)\ul{u}_l
=\sum_{l\in I_L}
\bA_k^*(Z)(\ul{\bA}^*)^{\omega_l}
\big(\ul{W}\big)
\bA_{j_l}^*(Z_l)C^*\ul{u}_l
\\=\bA_k^*(Z)\sum_{l\in I_L}S
\ul{\bA}^{\omega_l}
\big(\ul{W}\big)
\bB_{j_l}(Z_l)\ul{u}_l
=\bA_k^*(Z)S\ul{w}_1.
\end{multline*}
Once again, from the controllability of 
$(\ul{\bA},\ul{\bB})$ 
we have
\begin{align}
\label{eq:8Dec18c}
S\bA_k(Z)=\bA_k^*(Z)S,\,
\forall 1\le k\le d,\,
Z\in\KK^{\sts}.
\end{align}

$\bullet$
It is easily seen that every self-adjoint matrix $S$
which satisfies the relations in
(\ref{eq:8Dec18d}),
must satisfies
(\ref{eq:8Dec18a}) too. This implies the uniqueness of the matrix 
$S$.
\\\\
2. Suppose all the relations in 
(\ref{eq:8Dec18d})
hold.

$\bullet$ Let 
$\ul{v}\in\ker(S)$, thus
(\ref{eq:8Dec18c})
implies that 
$\ul{u}:=\bA_k(Z)\ul{v}\in\ker(S)$
while  
(\ref{eq:8Dec18b})
implies that 
$\ul{u}\in\ker(C)$. So 
$\ul{u}\in\ker(S)$ implies that 
$\bA^\omega(X_1,\ldots,
X_{|\omega|})\ul{u}\in \ker(C)$, 
but 
$(C,\ul{\bA})$ 
is observable, hence 
$\ul{u}=\ul{0}$
and
$\ul{v}\in \ker\bA_k(Z)$ 
for all 
$1\le k\le d$ 
and 
$Z\in\KK^{\sts}$. 

$\bullet$ On the other hand, if 
$\bA_k(Z)\ul{v}=\ul{0}$ 
for all 
$1\le k\le d$ 
and 
$Z\in\KK^{\sts}$, 
then 
$\bA_k^*(Z)S\ul{v}=\ul{0}$ 
and also 
$\bB_k(Z)S\ul{v}=\ul{0}$, 
whereas the controllability of
$(\ul{\bA},\ul{\bB})$ 
implies that 
$S\ul{v}=\ul{0}$.

$\bullet$ 
If 
$\bB_k(Z)\ul{v}=\ul{0}$ 
for all
$1\le k\le d$ 
and 
$Z\in\KK^{\sts}$, 
then 
$\bA_k^*(Z)C^*\ul{v}=\ul{0}$ 
and also 
$\bB_k^*(Z)C^*\ul{v}=\ul{0}$, 
which implies that 
$C^*\ul{v}=\ul{0}$.

$\bullet$ 
On the other hand, if 
$\ul{v}\in\ker(C^*)$, 
then 
$\bB_k(Z)\ul{v}\in\ker(C)$ 
and  
$S\bB_k(Z)\ul{v}=\ul{0}$, 
i.e., 
$\bB_k(Z)\ul{v}\in\ker(S)$.
As
$S$ 
is normal it implies that also 
$\bB_k(Z)\ul{v}\in\ker(S)$ 
and thus 
$\bB_k(Z)\ul{v}
\in\ker(\bA_j(\wt{Z}))$ 
for every 
$1\le j\le d$ 
and
$\wt{Z}\in\KK^{\sts}$, 
thus 
$\bB_k(Z)\ul{v}=\ul{0}$. 
\\
\\
3. The matrix
$\begin{bmatrix}
S\\
C\\
\end{bmatrix}$
is left invertible, as if
$\ul{u}\in\ker
\begin{bmatrix}
S\\
C\\
\end{bmatrix}$,
then
$\ul{u}\in\ker C$ and 
$$\ul{u}\in\ker (S)=
\bigcap_{1\le k\le d,\, X\in\KK^{\sts}} 
\ker (\bA_k(X)),$$ which implies 
$\ul{u}=\ul{0}$, 
since 
$(C,\ul{\bA})$ 
is observable. 

$\bullet$ Let
$K\in\KK^{L\times(L+s)}$ 
be a left inverse of
$\begin{bmatrix}
S\\C\\
\end{bmatrix}$ 
and 
$\wh{\bA_k}:=K\cdot
\begin{bmatrix}
\bA_k^*\\
\bB_k^*\\
\end{bmatrix},$
thus
\begin{align}
\label{eq:19Apr19a}
&\bA_k=K\cdot
\begin{bmatrix}
S\\
C\\
\end{bmatrix}\cdot\bA_k
=\wh{\bA_k}\cdot S,\,
\bB_k=K\cdot
\begin{bmatrix}
S\\
C\\
\end{bmatrix}\cdot\bB_k=
\wh{\bA_k}\cdot C^*
\end{align}
and
\begin{multline*}
\begin{bmatrix}
S\\
C\\
\end{bmatrix}\cdot
\wh{\bA_k}\cdot
\begin{bmatrix}
S&C^*\\
\end{bmatrix}
=\begin{bmatrix}
S\cdot\wh{\bA_k}\cdot S&S\cdot
\wh{\bA_k}\cdot C^*\\
C\cdot\wh{\bA_k}\cdot 
S&C\cdot\wh{\bA_k}\cdot C^*\\
\end{bmatrix}
=\begin{bmatrix}
S\cdot\bA_k&S\cdot\bB_k\\
C\cdot\bA_k&C\cdot\bB_k\\
\end{bmatrix}
\\=\begin{bmatrix}
(S\cdot\bA_k)^*&(C\cdot\bA_k)^*\\
(S\cdot\bB_k)^*&(C\cdot\bB_k)^*\\
\end{bmatrix}=
\begin{bmatrix}
S\cdot\wh{\bA_k}^*\cdot S
&S\cdot\wh{\bA_k}^*\cdot C^*\\
C\cdot\wh{\bA_k}^*\cdot S
&C\cdot\wh{\bA_k}^*\cdot C^*\\
\end{bmatrix}
=\begin{bmatrix}
S\\
C\\
\end{bmatrix}\cdot
\wh{\bA_k}^*\cdot
\begin{bmatrix}
S&C^*\\
\end{bmatrix},
\end{multline*}
which implies that
$\wh{\bA_k}=\wh{\bA_k}^*$. 
Moreover, 
$\wh{\bA_k}$ 
is independent of the choice of the left inverse
$K$: 
if 
$K^\prime$ 
is another left inverse of 
$\begin{bmatrix}  
S\\C\\
\end{bmatrix}$, 
then 
\begin{align*}
(K-K^\prime)\cdot
\begin{bmatrix}
\bA_k^*\\
\bB_k^*\\
\end{bmatrix}\cdot
\begin{bmatrix}
S&C^*\\
\end{bmatrix}=
\begin{bmatrix}
\bA_k-\bA_k& \bB_k-\bB_k\\
\end{bmatrix}=
\begin{bmatrix}
0&0\\
\end{bmatrix}
\end{align*}
and the right invertibility of
$\begin{bmatrix}
S&C^*\\
\end{bmatrix}$ 
implies that
$K\cdot\begin{bmatrix}
\bA_k^*\\
\bB_k^*\\
\end{bmatrix}
=K^\prime\cdot
\begin{bmatrix}
\bA_k^*\\
\bB_k^*\\
\end{bmatrix}$. Notice that from the controllability of
$(\ul{\bA},\ul{\bB})$ 
we get 
$$\bigcap_{1\le k\le d,\,X\in\KK^{\sts}} 
\ker\big(\wh{\bA_k}(X)\big)=
\{\ul{0}\}.$$

$\bullet$ Next, as 
$S=S^*$, 
one can write
$S=TJT^*$, 
where 
$T\in\KK^{L\times L}$ 
is invertible and 
$J\in\KK^{L\times L}$ 
is a semi-signature matrix. Thus, using
the relations in 
(\ref{eq:19Apr19a}), we obtain
\begin{multline*}
\cR_{\cF\cM}(\ul{X})=D+C
\Big(I_{L}-\sum_{k=1}^d 
\wh{\bA_k}(X_k-Y_k)S\Big)^{-1}
\sum_{k=1}^d
\wh{\bA_k}(X_k-Y_k)C^*
\\=D+\check{C}
\Big(I_L-\sum_{k=1}^d
\check{\bA_k}(X_k-Y_k)J\Big)^{-1}
\sum_{k=1}^d
\check{\bA_k}(X_k-Y_k)\check{C}^*,
\end{multline*}
where
$\check{C}=C(T^*)^{-1}$
and
$\check{\bA_k}=T^*\cdot\wh{\bA_k}\cdot T$ 
are hermitian, for 
$1\le k\le d$.
Moreover, it is easily seen  that the 
controllability of
$(\ul{\bA},\ul{\bB})$
and the observability of
$(C,\ul{\bA})$
imply 
the controllability of
$(\ul{\check{\bA}}\cdot J,\ul{\check{\bA}}\cdot \check{C}^*)$
and
the observability of 
$(\check{C},\ul{\check{\bA}}\cdot J)$, 
respectively.
\\
\\
4. Define
$E:=F+\check{C}J\check{C}^*$. 
The matrix 
$$\wt{S}=
\begin{bmatrix}
J&\check{C}^*\\
\check{C}&E
\end{bmatrix}\in\KK^{(L+s)\times(L+s)}$$ 
is hermitian and invertible, whenever 
$F\succ0$
or
$F\prec 0$: 
it is easily seen that 
$\wt{S}^*=\wt{S}$; 
let 
$\begin{bmatrix}
\ul{u}\\
\ul{v}\\
\end{bmatrix}
\in\ker(\wt{S})$, 
thus 
$J\ul{u}+\check{C}^*\ul{v}=\ul{0}$
and
$\check{C}\ul{u}+E\ul{v}=\ul{0}$, 
therefore 
$\check{C}^*\ul{v}=-J\ul{u}$ and when we plug it in to the other equation we get
$F\ul{v}=\check{C}(J^2-I_L)\ul{u}$.
Thus,
$-J\ul{u}=\check{C}^*F^{-1}\check{C}
(J^2-I_L)\ul{u}$ 
and multiplying both sides by 
$\ul{u}^*(J^2-I_L)$ 
on the left, to get
$$\ul{u}^*(J^2-I_L)\check{C}^*
F^{-1}\check{C}(J^2-I_L)\ul{u}=
-\ul{u}^*(J^2-I_L)J\ul{u}=
-\ul{u}^*(J^3-J)\ul{u}=\ul{0}.$$ 
Then
$\ul{v}=\check{C}(J^2-I_L)\ul{u}=\ul{0}$,
which implies that 
$J\ul{u}=\ul{0}$ 
and 
$\check{C}\ul{u}=\ul{0}$,
i.e., that 
$\begin{bmatrix}
J\\
\check{C}\\
\end{bmatrix}\ul{u}=\ul{0}$.
Recall that 
$\begin{bmatrix}
S\\
C\\
\end{bmatrix}=
\begin{bmatrix}
TJT^*\\
\check{C}T^*\\
\end{bmatrix}=
\begin{bmatrix}
T&0\\
0&I_s\\
\end{bmatrix}
\begin{bmatrix}
J\\\check{C}\\
\end{bmatrix}T^*$
is left invertible and hence 
$\begin{bmatrix}
J\\
\check{C}\\
\end{bmatrix}$
is left invertible and hence 
$\ul{u}=\ul{0}$.

$\bullet$ As
$\wt{S}$ 
is hermitian and invertible, there exist an invertible matrix 
$\wt{T}\in\KK^{(L+s)\times(L+s)}$ 
and a signature matrix
$J_{\cD}\in\KK^{(L+s)\times(L+s)}$ such that 
$\wt{S}^{-1}=\wt{T}J_{\cD}\wt{T}^*$, therefore
\begin{multline*}
\cR_{\cF\cM}(\ul{X})=D-E+
\begin{bmatrix}
\check{C}&E\\
\end{bmatrix}
\Big( I_{L+s}-\sum_{k=1}^d
\begin{bmatrix}
\check{\bA_k}J&\check{\bA_k}\check{C}^*\\
0&0\\
\end{bmatrix}(X_k-Y_k)\Big)^{-1}
\begin{bmatrix}
0\\I_s\\
\end{bmatrix}
=D\\-E
+
\begin{bmatrix}
0&I_s\\
\end{bmatrix}\wt{S}
\Big(
I_{L+s}-\sum_{k=1}^d
\begin{bmatrix}
\check{\bA_k}&0\\
0&0\\
\end{bmatrix}(X_k-Y_k)\wt{S}
\Big)^{-1}
\begin{bmatrix}
0\\
I_s\\
\end{bmatrix}
=D-E+
\begin{bmatrix}
0&I_s\\
\end{bmatrix}\\
\Big(\wt{S}^{-1}-\sum_{k=1}^d
\begin{bmatrix}
\check{\bA_k}&0\\
0&0\\
\end{bmatrix}(X_k-Y_k)\Big)^{-1}
\begin{bmatrix}
0\\
I_s\\
\end{bmatrix}
=D_{\cD}+C_{\cD}
\Big( J_{\cD}-\sum_{k=1}^d
\bA_{k,\cD}(X_k-Y_k)\Big)^{-1}C_{\cD}^*
=\cR_{\cD}(\ul{X}),
\end{multline*}
for every 
$\ul{X}\in DOM_s(\cR_{\cF\cM})$,
where
$$D_{\cD}:=D-F-\check{C}J\check{C}^*,\, 
C_{\cD}:=
\begin{bmatrix}
0&I_s\\
\end{bmatrix}
\wt{T}^{-*} 
\text{ and } 
\bA_{k,\cD}:=\wt{T}^{-1}
\begin{bmatrix}
\check{\bA_k}&0\\
0&0\\
\end{bmatrix}\wt{T}^{-*},$$
for every 
$1\le k\le d$.
We showed that 
$\cR_{\cF\cM}(\ul{X})=
\cR_{\cD}(\ul{X})$ 
and it is easily seen from 
the last computation that 
$$\ul{X}\in 
DOM_s(\cR_{\cF\cM})\iff 
\det\Big(J_{\cD}-\sum_{k=1}^d
\bA_{k,\cD}(X_k-Y_k)\Big)\ne0,$$
i.e., that 
$DOM_s(\cR_{\cD})=DOM_s(\cR_{\cF\cM})$. 

$\bullet$ Furthermore, straight forward computations 
show that
for every 
$m\ge 1$:
$$\ul{X}\in DOM_{sm}(\cR_{\cF\cM})\iff
\det\Big(I_m\otimes J_{\cD}-\sum_{k=1}^d
(X_k-I_m\otimes Y_k)\bA_{k,\cD}\Big)\ne0,$$
i.e., that 
$DOM_{sm}(\cR_{\cF\cM})=
DOM_{sm}(\cR_{\cD})$  and also that 
$\cR_{\cF\cM}(\ul{X})=\cR_{\cD}(\ul{X})$,
as well as that
$DOM^{\cA}(\cR_{\cF\cM})=DOM^{\cA}(\cR_{\cD})$ 
for any unital stably finite 
$\KK-$algebra $\cA$, i.e., for every 
$\ul{\fA}\in 
(\cA^{\sts})^d$: 
$$\ul{\fA}\in
DOM^{\cA}(\cR_{\cF\cM})\iff 
\Big(J_{\cD}\otimes 1_{\cA}-\sum_{k=1}^d
(\fA_k-Y_k\otimes 1_{\cA})\bA_{k,\cD}^{\cA}\Big)
\text{ is invertible in }\cA^{(L+s)\times (L+s)}$$
and for such $\ul{\fA}$
we have $\cR_{\cF\cM}^{\cA}(\ul{\fA})=\cR_{\cD}^{\cA}(\ul{\fA})$.
\\\\
5. If
$S$
is invertible, then 
$\bB_k^*\cdot S=C\cdot\bA_k$
implies that
$\bB_k=S^{-1}\cdot\bA_k^*\cdot C^*=\bA_k\cdot S^{-1}C^*$
so we can choose 
$Q=S^{-1}C^*$ to get 
(\ref{eq:9Jan19a}).

$\bullet$ On the other hand, suppose there exists 
$Q\in\KK^{L\times s}$
such that 
$\bB_k=\bA_k\cdot Q$ 
for all
$1\le k\le d$.
Thus,
$C\cdot\bA_k=\bB_k^*\cdot S=Q^*\cdot\bA_k^*\cdot S=Q^*S\cdot\bA_k$
and also 
$C\cdot\bB_k=C\cdot\bA_k\cdot Q
=Q^*S\cdot\bA_k\cdot Q=Q^*S\cdot\bB_k,$
but the controllability of 
$(\ul{\bA},\ul{\bB})$ 
implies that 
$C=Q^*S$.
Therefore, $\ker(S)\subseteq \ker(C)$ 
but as
$$\ker(S)=\bigcap_{1\le k\le d,\, 
X\in\KK^{\sts}} 
\ker \left(\bA_k(X)\right),$$ 
it follows from the observability of 
$(C,\ul{\bA})$ 
that 
$\ker(S)=\{\ul{0}\}$, 
i.e., that 
$S$ 
is invertible.

$\bullet$
If 
$S$
is invertible, then 
$J$ is invertible and the realization 
(\ref{eq:4Jan19a}) 
is obtained from the realization 
(\ref{eq:1Feb19a}) 
immediately. 

$\bullet$ Finally, from 
(\ref{eq:4Jan19a}) we get
\begin{multline*}
\cR_{\cF\cM}(\ul{X})
=D+\check{C}J
\Big(J-\sum_{k=1}^d \check{\bA_k}(X_k-Y_k)\Big)^{-1}
\Big(\sum_{k=1}^d \check{\bA_k}(X_k-Y_k)-J+J\Big)
\check{C}^*
\\=D-\check{C}J\check{C}^*
+\check{C}
\Big(J-\sum_{k=1}^d \wt{\bA_k}(X_k-Y_k)\Big)^{-1}
\check{C}^*,
\end{multline*}
that is the realization in 
(\ref{eq:1Feb19c})
\end{proof}
\begin{remark}
If 
$\fR$ is an hermitian nc rational function, there exists $\ul{Y}\in dom(\fR)$ such that $\ul{Y}^*=\ul{Y}$;
for a proof see  
\cite[pp. 28--29]{V2}.
\end{remark}
\begin{remark}

We leave it to a future work, to describe connections between properties of the (semi-)signature matrices,
$J$ and
$J_{\cD}$, appear in Theorem 
\ref{thm:22Feb19a} 
and properties of the  function 
$\fR$, such as perhaps (matrix) convexity (cf. 
\cite{HMcCV}).
\end{remark}
\begin{remark}
\label{rem:9May19b}
We say that a nc rational function
$\fR$ 
admits a 
descriptor realization 
$$\cR_{\cD}(\ul{X})=C_{\cD}
\Big(I_{L_{\cD}}-\sum_{k=1}^d
\bA_{k,\cD}(X_k-Y_k)\Big)^{-1}B_{\cD}$$
centred at 
$\ul{Y}\in(\KK^{\sts})^d$, if $dom_{sm}(\fR)\subseteq DOM_{sm}(\cR_{\cD})$
and
$\fR(\ul{X})=\cR_{\cD}(\ul{X})$
for every 
$\ul{X}\in dom_{sm}(\fR)$, 
cf. Definition
\ref{def:25Sep18a}.
We present some relations between Fornasini--Marchesini realizations and descriptor realizations (not necessarily in the symmetric case), without precise definitions of 
controllability and observability, as well as the McMillan degree 
(denoted by $Deg_{s,\cD}(\fR)$), of descriptor realization:

$\bullet$ If 
$\fR$ 
admits a descriptor realization centred at 
$\ul{Y}$, described by 
$(L_\cD,C_{\cD},
\ul{\bA}_{\cD},B_{\cD})$, 
then it admits a
nc Fornasini--Marchesini realization described by
$$L_{\cF\cM}=L_{\cD},\,
D_{\cF\cM}=C_{\cD}B_{\cD},\,
C_{\cF\cM}=C_{\cD},\,
\bA_{k,\cF\cM}=\bA_{k,\cD}
\text{ and } 
\bB_{k,\cF\cM}=\bA_{k,\cD}B_{\cD},$$ 
for 
$1\le k\le d$, 
$\cR_{\cD}$ 
is observable if and only if 
$\cR_{\cF\cM}$ 
is observable, and if
$\cR_{\cF\cM}$
is controllable then
$\cR_{\cD}$
is controllable.
Therefore we  have the relation
\begin{align}
\label{eq:5Jan19b}
\fm(\fR)s\le Deg_{s,\cD}(\fR).
\end{align}

$\bullet$
If
$\fR$ 
admits a nc
Fornasini--Marchesini realization centred at 
$\ul{Y}\in(\KK^{\sts})^d$, 
described by
$(L_{\cF\cM},D_{\cF\cM},C_{\cF\cM},\ul{\bA}_{\cF\cM},\ul{\bB}_{\cF\cM})$,
then 
$\fR$ 
admits a descriptor realization described by
\begin{align*}
L_{\cD}=L_{\cF\cM}+s,\,
C_{\cD}=
\begin{bmatrix}
C_{\cF\cM}&D_{\cF\cM}\\
\end{bmatrix},\,
\bA_{k,\cD}=
\begin{bmatrix}
\bA_{k,\cF\cM}&\bB_{k,\cF\cM}\\
0&0_{\sts}\\
\end{bmatrix}
\text{ and }
B_{\cD}=
\begin{bmatrix}
0_{L_{\cF\cM}\times s}\\
I_s\\
\end{bmatrix},\end{align*}
for
$1\le k\le d$,
$\cR_{\cF\cM}$
is controllable if and only if
$\cR_{\cD}$
is controllable, and if 
$\cR_{\cD}$
is observable then $\cR_{\cF\cM}$
is observable. Therefore we have the relation
\begin{align}
\label{eq:5Jan19a}
Deg_{s,\cD}(\fR)\le 
(\fm(\fR)+1)s.
\end{align}

$\bullet$ 
The inequalities
(\ref{eq:5Jan19b}) and 
(\ref{eq:5Jan19a})
imply that
$$\fm(\fR)s\le Deg_{s,\cD}(\fR)\le 
(\fm(\fR)+1)s,$$
whereas an analogue of Lemma 
\ref{rem:18Nov18a} 
for descriptor realizations guarantees that 
$s\mid Deg_{s,\cD}(\fR)$,
then which then imply that
\begin{align*}
Deg_{s,\cD}(\fR)=\fm(\fR)s
\text{ or }
Deg_{s,\cD}(\fR)=(\fm(\fR)+1)s. 
\end{align*}
\end{remark}

\end{document}